\documentclass[11pt,a4paper,reqno]{amsart}
\usepackage{color}
\usepackage{amsfonts,amsmath,amssymb,amsxtra,url,float} 
\usepackage[colorlinks,linkcolor=RoyalBlue,anchorcolor=Periwinkle,citecolor=Orange,urlcolor=Green]{hyperref} 
\usepackage[usenames,dvipsnames]{xcolor} 
\usepackage{enumitem}
\usepackage{geometry} 
\usepackage{graphicx} 
\usepackage{subfigure} 
\usepackage{tikz}\usetikzlibrary{matrix}\usetikzlibrary{trees}
\usetikzlibrary{decorations.pathreplacing,decorations.markings}
 \tikzset{
  on each segment/.style={
    decorate,
    decoration={
      show path construction,
      moveto code={},
      lineto code={
        \path [#1]
        (\tikzinputsegmentfirst) -- (\tikzinputsegmentlast);
      },
      curveto code={
        \path [#1] (\tikzinputsegmentfirst)
        .. controls
        (\tikzinputsegmentsupporta) and (\tikzinputsegmentsupportb)
        ..
        (\tikzinputsegmentlast);
      },
      closepath code={
        \path [#1]
        (\tikzinputsegmentfirst) -- (\tikzinputsegmentlast);
      },
    },
  },
  mid arrow/.style={postaction={decorate,decoration={
        markings,
        mark=at position .5 with {\arrow[#1]{stealth}}
      }}},
}
\usetikzlibrary{arrows}

\usetikzlibrary{matrix}
\usetikzlibrary{patterns}
\usetikzlibrary{shadings}\usepgflibrary{shadings} 
\usepackage[all]{xy} 
\usepackage{mathdots} 
\usepackage{yhmath} %
\usepackage{dsfont} 
\usepackage{cite}
\usepackage{mathrsfs} 
\usepackage{multicol} 
\usepackage{changepage}
\numberwithin{figure}{section}

\usepackage{marginnote} 
\usepackage{graphicx} 
\usepackage{multicol} 

\newtheorem{theorem}{Theorem}[section]
\newtheorem{lemma}[theorem]{Lemma}
\newtheorem{corollary}[theorem]{Corollary}
\newtheorem{main theorem}[theorem]{Main Theorem}
\newtheorem{proposition}[theorem]{Proposition}
\newtheorem{definition}[theorem]{Definition}
\newtheorem{construction}[theorem]{Construction}
\newtheorem{remark}[theorem]{Remark}
\newtheorem{example}[theorem]{Example}

\usetikzlibrary{arrows}

\numberwithin{equation}{section}

\def\<{\langle} 
\def\>{\rangle} 
\def\NN{\mathbb{N}} 
\def\ZZ{\mathbb{Z}} 

\newcommand{\Pic}{F{\tiny{IGURE}}\ }
\newcommand{\modcat}{\mathsf{mod}}
\newcommand{\Dcat}{\mathcal{D}}
\newcommand{\ind}{\mathsf{ind}}
\newcommand{\kk}{\mathds{k}} 
\newcommand{\Q}{\mathcal{Q}} 
\newcommand{\I}{\mathcal{I}} 
\newcommand{\per}{\mathsf{per}} 
\newcommand{\Hom}{\mathrm{Hom}} %
\newcommand{\Ext}{\mathrm{Ext}} %
\renewcommand{\H}{\mathrm{H}} %
\newcommand{\SURF}{\mathbf{S}} 
\newcommand{\Surf}{\mathcal{S}} 
\newcommand{\bSurf}{\partial\mathcal{S}} 
\newcommand{\M}{\mathcal{M}} 
\newcommand{\MM}{\mathfrak{M}} 
  \newcommand{\gbullet}{{\color{blue}\bullet}} 
\newcommand{\rbullet}{{\color{red}\circ}} 
\newcommand{\E}{\mathcal{E}} 
  \newcommand{\Dblue}{\Delta_{\color{blue}\bullet}} 
\newcommand{\Dred}{\Delta_{\color{red}\circ}} 
\newcommand{\tDred}{\widetilde{\Delta}_{{\color{red}\circ}}} 
\newcommand{\PP}{\mathcal{P}} 
\newcommand{\innerSurf}{\mathcal{S}\backslash\partial\mathcal{S}} 
\newcommand{\tc}{\tilde{c}} 
\newcommand{\Y}{\mathcal{Y}} 
\newcommand{\F}{\mathcal{F}} 
\newcommand{\ii}{\mathfrak{i}} 
\newcommand{\PC}{\mathrm{PC}} 
\newcommand{\AC}{\widetilde{\mathrm{AC}}{}} 
\newcommand{\rota}{\circlearrowright}

\newcommand{\ta}{\tilde{a}}
\newcommand{\ared}[2]{\tilde{a}^{#1}_{\rbullet, #2}}

\newcommand{\agreen}[2]{a^{#1}_{\gbullet, #2}}

\newcommand{\dualgreen}[2]{{\textit{\tiny D}}^{#1}_{\gbullet, #2}}
\newcommand{\m}{\mathfrak{m}}

\newcommand{\X}{\mathfrak{X}}
\renewcommand{\top}{\mathrm{top}}
\newcommand{\soc}{\mathrm{soc}}
\newcommand{\simp}{\mathrm{simp}}
\newcommand{\proj}{\mathrm{proj}}
\newcommand{\longline}{-\!\!\!-\!\!\!-}
\newcommand{\bfv}{\mathfrak{v}}

\newcommand{\pdim}{\mathrm{proj.dim}}

\newcommand{\Egreen}{\mathfrak{E}_{\gbullet}}
\newcommand{\Ered}{\mathfrak{E}_{\rbullet}}
\newcommand{\comp}[1]{\pmb{#1}^{\bullet}}
\newcommand{\rmH}{\mathrm{H}}
\newcommand{\rmF}{\mathrm{F}}
\newcommand{\rmL}{\mathrm{L}}
\newcommand{\rmM}{\mathrm{M}}
\newcommand{\im}{\mathrm{im}}
\newcommand{\tru}{\mathfrak{T}} 
\newcommand{\hs}{\widehat{s}} 
\newcommand{\rmI}{\mathrm{I}}
\newcommand{\rmII}{\mathrm{II}}
\newcommand{\ocup}{\overrightarrow{\sqcup}}
\newcommand{\Ocup}{\mathop{\overrightarrow{\bigsqcup}}}
\newcommand{\hl}{\mathrm{hl}}

\newcommand{\Left}{\mathrm{L}}
\newcommand{\Right}{\mathrm{R}}

\def\defines{\it\color{black!100}}

\begin{document}

\title[Constructing projective resolution and taking cohomology]{Constructing projective resolution and taking cohomology for gentle algebras in the geometric model}
\author{Yu-Zhe Liu}
\address{Yu-Zhe Liu, School of Mathematics and statistics, Guizhou University, Guiyang 550025, P. R. China}
\email{yzliu3@163.com}
\author{Chao Zhang}
\address{Chao Zhang, School of Mathematics and statistics, Guizhou University, Guiyang 550025, P. R. China}
\email{zhangc@mass.ac.cn}
\author{Houjun Zhang}
\address{Houjun Zhang, School of Science, Nanjing University of Posts and Telecommunications, Nanjing 210023, P. R. China}
\email{zhanghoujun@nju.edu.cn}





\begin{abstract}
The geometric models for the module category and derived category of any gentle algebra were introduced to realize the objects in module category and derived category by permissible curves and admissible curves respectively.
The present paper firstly unifies these two realizations of objects in module category and derived category via same surface for any gentle algebra, by the rotation of permissible curves corresponding to the objects in the module category.
Secondly, the geometric characterization of the cohomology of complexes over gentle algebras is established by the truncation of projective permissible curves.
It is worth mentioning that the rotation of permissible curves and the truncation of projective permissible curves are mutually inverse processes to some extent.
As applications, an alternative proof of ``no gaps" theorem as to cohomological length for the bounded derived categories of gentle algebras is provided in terms of the geometric characterization of the cohomology of complexes. Moreover, we obtain a geometric proof for the strong Nakayama conjecture for gentle algebras.
Finally, we contain two examples to illustrate our results.
\end{abstract}

\maketitle
\tableofcontents

\setlength\parindent{0pt}
\setlength{\parskip}{5pt}

\section{Introduction}

Gentle algebras, introduced by Assem and Skowro\'{n}ski in \cite{AS1987} are of
particular interest and the module categories and derived categories of gentle algebras have been widely studied
\cite{ALP2016, BD2017, BM2003, BR1987, CB1989, Kra1991, R1997, WW1985, KY2018}.
Motivated by the surface models of cluster algebras in \cite{ABCJP2010, LF2009}, the surface models for gentle algebras were studied, including the geometric models of the module categories of gentle algebras recently \cite{BCS2019, BCS2023}.
The geometric models of derived categories of gentle algebras were introduced by Haiden, Katzarkov and Kontsevich in \cite{HKK2017} and Opper-Plamondon-Schroll in \cite{OPS2018}, respectively.
These geometric models provide a geometric realization of the indecomposables objects and morphisms in both the module and derived categories of algebras, and are used to solve many important questions in representation theory,
for instance, the connectedness of support $\tau$-tilting graphs \cite{FGLZ2023};
silting theory \cite{CS2023b, CJS2022}; $\tau$-tilting theory \cite{HZZ2020}; the theory of exceptional sequences \cite{CS2023} and so on.

Form the view point of homological algebra, it is well-known that the module category can be naturally embedded into its derived category (or homotopy category of perfect complexes) by sending any module to its projective complex. Conversely, any complex in derived category naturally yields a module by taking the cohomologies. To some extent, constructing the projective resolution and taking cohomology are mutually inversal when restrict to the module category.
However, for any gentle algebra, the geometric model of module category is different from the geometric model of derived category.
This difference makes it impossible to deal with the questions on the module and derived categories of a gentle algebra in a unified model. Recently, in order to give a geometric model for the length heart in the derived category of a gentle algebra, Chang \cite{Cpre} established a way to embed the geometric model for the module category to the geometric model for the derived category by using the simple coordinate and projective coordinate for any so-called zigzag curve.
Quite different from his method, we obtain an embedding of geometric models of gentle algebras from module categories to derived categories by rotating projective permissible curves.
This strategy was used by the first and the third auctors to classify the silted algebras of Dynkin type $A_{n}$ in \cite{ZLpre}, and it is shown that there are no strictly shod algebras in hereditary gentle algebras by using the embedding of permissible curves \cite{BCS2019} to admissible curves in \cite{OPS2018}.
Indeed, our first main result is as follows.

\begin{theorem} {\rm (Theorems \ref{thm:string embedding} and \ref{thm:band embedding})} \label{thm:embedding}
Let $A$ be a gentle algebra.
\begin{itemize}
\item[{\rm(1)}] If $c$ is a permissible curve and $\tc^{\rota}$ is the admissible curve which obtained by rotating $c$ $($see Definition \ref{def:rot}$)$, then there exists an isomorphism in $\Dcat^b(A)$ as follows
\begin{center}
$\rm \X(\tc^{\rota}) \cong \mathbf{P}(\MM(c))$;
\end{center}
\item[{\rm(2)}] If $c$ is a permissible curve without endpoint and $\tc^{\rota}$ is the admissible curve without endpoint which obtained by rotating $c$, then there exists an isomorphism in $\Dcat^b(A)$ as follows
    \begin{center}
{$\X(\tc^{\rota}, \pmb{J}_{n}(\lambda)) \cong \mathbf{P}(\MM(c, \pmb{J}_{n}(\lambda)))$,}
\end{center}where
\end{itemize}

\begin{itemize}
  \item $\pmb{J}_{n}(\lambda)$ is an $n\times n$ Jordan block with eigenvalue $0\ne \lambda \in \kk$;
  \item $\MM(c)\in \modcat A$ and $\X(\tc^{\rota})\in \per A$ are indecomposable objects corresponding to $c$ and $\tc^{\rota}$, respectively;
  \item $\MM(c, \pmb{J}_{n}(\lambda)) \in \modcat A$ and $\X(\tc^{\rota}, \pmb{J}_{n}(\lambda)) \in \per A$
    are indecomposable objects corresponding to $(c, \pmb{J}_{n}(\lambda))$
    and $(\tc^{\rota}, \pmb{J}_{n}(\lambda))$, respectively;
  \item $\mathbf{P}(M)$ is the deleted projective complex induced by the projective resolution of $M$.
\end{itemize}
\end{theorem}

Our next goal in this paper is to study the cohomology of complexes for gentle algebras by using our unified geometric model. In order to do this, we introduce the truncations of projective permissible curves (see Definition \ref{def:truncation}). Furthermore, for any two truncations $s$ and $s^\prime$ of some projective permissible curves, one can define the supplemental union $s\ocup s^{\prime}$ of $s$ and $s^{\prime}$ to be the completion of $s\cup s^{\prime}$. Then we have the following theorem.

\begin{theorem} {\rm (Theorem \ref{thm:homologies})}
Let $\tc$ be an admissible curve in $\AC_{\m}(\SURF^{\F})$.
Then, for any $n\in\ZZ$, there are integers $r_i$ and $s_i$ with $0 \le r_i \le s_i \le m(c_{\proj}(\dualgreen{c}{i}))+1$ and $1\le i\le n(\tc)$
such that
\begin{center}
  $\MM^{-1}(\rm H^n(\X(\tc))) = \Ocup\limits_{n=-\varsigma^{\tc}_{i}}\tru_{r_i}^{s_i}(c_{\proj}(\dualgreen{c}{i})),$
\end{center}
where $\tru_{r_i}^{s_i}(c_{\proj}(\dualgreen{c}{i}))$ are the truncations of projective permissible curves $c_{\proj}(\dualgreen{c}{i})$.
\end{theorem}
We call the truncation $\tru_{r_i}^{s_i}(c_{\proj}(\dualgreen{c}{i}))$ the $i$-th cohomological curve of $\tc$ at the intersection $c(i)$. Notice that for a given permissible curve $c$, we can obtain an admissible curve $\tc^{\rota}$ by the embedding given in Theorem \ref{thm:embedding}. In particular, we show that

\begin{theorem} {\rm (Theorem \ref{thm:string embedding 2})} \label{thm:inverse}
Any permissible curve $c$ is the completion of the supplemental union of all cohomological curves of $\tc^{\rota}$.
\end{theorem}

The above theorem actually shows that the rotation of permissible curves and the truncation of projective permissible curves are mutually inversal restricting to projective permissible curves of the form of $\tc^{\rota}$. 
 
Finally, there are two applications of our results.
In terms of the geometric characterization of the cohomology of complexes, we firstly give an alternative proof for the  ``no gaps" theorem as to cohomological length for the bounded derived categories of gentle algebras \cite{Z2019}, i.e.,  if there is an indecomposable object in $\Dcat^b(A)$ of cohomological length $n > 1$, then there exists
an indecomposable with cohomological length $n-1$, which answers affirmatively a question proposed by the second author and Han in \cite{ZH2016}.  Secondly, the embedding are used a geometric proof for the Strong Nakayama conjecture for gentle algebras. Moreover, we also contain two examples to illustrate our results. 

This paper is organized as follows. Section \ref{sec:preliminaries} is the preliminaries on gentle algebras and the geometric models. In Section \ref{sec:embedding}, we study the embedding of geometric models for gentle algebras from module categories to derived categories. In Section \ref{sec:truncations}, we give a geometric characterization of the cohomology of complexes for gentle algebras. Then we show that the embedding and taking cohomology are inverse in Section \ref{sec:inverse}. In Section \ref{sec:applications}, we give two applications. In the last Section, we contain two examples to illustrate our results.

\section{Gentle algebras and their geometric models} \label{sec:preliminaries}

Throughout this paper, we always assume that any surface $\Surf$ is homologically smooth and its boundary $\bSurf$ is non-empty
and every curve $c$ in surface is a function $[0,1]\to \Surf$ such that $c(t)\in\innerSurf$ for all $0<t<1$.  $\kk$ denotes an algebraically closed field and all algebras will be finite-dimensional $\kk$-algebras. All modules we consider are right modules.
Arrows in a quiver are composed from left to right: for arrows $\alpha$ and $\beta$ we write $\alpha\beta$ for the path from the source of $\alpha$ to the target of $\beta$. The maps are composed from right to left, that is if $f:X \rightarrow Y$ and $g: Y \rightarrow Z $ then $gf: X \rightarrow Z$.

In this section, we recall some notations and concepts of geometric models for gentle algebras in \cite{BCS2019, HKK2017, OPS2018, QZZ2022}.

\subsection{Gentle algebras}

A bounded quiver $(\Q, \I)$ is said to be {\defines gentle} if:
\begin{itemize}
  \item Any vertex of $\Q$ is the source and target of at most two arrows.

  \item For each arrow $\alpha:x\to y$, there is at most one arrow $\beta$
  whose source is $y$ such that $\alpha\beta\in\I$,
  and there is at most one arrow $\beta$ whose source is $y$
  such that $\alpha\beta\notin\I$.

  \item For each arrow $\alpha:x\to y$, there is at most one arrow $\beta$
  whose target is $x$ such that $\beta\alpha\in\I$,
  and there is at most one arrow $\beta$ whose target is $x$
  such that $\beta\alpha\notin\I$.

  \item $\I$ is admissible and it is generated by paths of length $2$.
\end{itemize}

\begin{definition}\rm
A {\defines graded gentle algebra} is a finite dimensional $\kk$-algebra
$A\cong\kk\Q/\I$ with a {\defines grading} $|\cdot|:\Q_1 \to \ZZ$, where $(\Q, \I)$ is a gentle quiver.
In particular, a {\defines gentle algebra} is a graded gentle algebra
satisfying $|\alpha|=0$ for all $\alpha \in \Q_1$.
\end{definition}

\subsection{Graded marked ribbon surfaces} \label{subsect:geo.mod.}
A {\defines marked surface} is a triple $(\Surf, \M, \Y)$, where $\M$ and $\Y$ are finite subsets of the boundary $\bSurf$ of $\Surf$ such that elements in $\M$ and $\Y$ are alternative in every boundary component.
Elements in $\M$ and $\Y$ are called $\gbullet$-marked points and $\rbullet$-marked points, respectively.

An $\gbullet$-curve (resp. a $\rbullet$-curve) is a curve in $\Surf$ whose endpoints are $\gbullet$-marked points (resp. $\rbullet$-marked points).
Especially, we always assume that
all points $c(t)$ ($0<t<1$) of curve $c$ in $\Surf$ lie in $\innerSurf$,
and arbitrary two curves are representatives in their homotopy classes such that their intersections are minimal.

A {\defines full formal arc system} (=$\gbullet$-FFAS) of $(\Surf, \M, \Y)$, say $\Dblue$,
is a set of some $\gbullet$-curves such that:
\begin{itemize}
  \item any two $\gbullet$-curves in $\Dblue$ have no intersection in $\innerSurf$;
  \item every {\defines elementary $\gbullet$-polygon}, the polygon obtained by $\Dblue$ cutting $\Surf$, has a unique edge in $\bSurf$.
\end{itemize}

All elements of $\Dblue$ are called $\gbullet$-arcs. For any elementary $\gbullet$-polygon $\PP$, we denote by $\Egreen(\PP)$ the set of all $\gbullet$-arcs which are edges of $\PP$ lying in $\Dblue$. Similarly, we can define $\rbullet$-FFAS.

Note that all $\rbullet$-marked points lying in digon are called {\defines extra marked points} and we denote by $\E$ the set of all extra marked points. Obviously, $\E$ is a subset of $\Y$.

A {\defines graded surface} $\Surf^{\F}=(\Surf, \F)$ is a surface $\Surf$ with a section $\F$ of the projectivized tangent bundle $\mathbb{P}(T\Surf)$. $\F$ is called a {\defines foliation} (or {\defines grading}) on $\Surf$.

A {\defines graded curve} in a graded surface $(\Surf, \F)$ is a pair $(c, \tc)$ of curve $c: [0,1] \to \Surf$ and {\defines grading} $\tc$,
where $\tc$ is a homotopy class of paths in the tangent space $\mathbb{P}(T_{c(t)}\Surf)$ of $\Surf$ at $c(t)$
from the subspace given by $\F$ to the tangent space of the curve, varying continuously with $t\in [0,1]$.
For simplicity, we denote by $c$ and $\tc$ the curve $c:[0,1]\to \Surf$ and the graded curve $(c, \tc)$, respectively.
Let $\tc_1$ and $\tc_2$ be two graded curves and $p$ be an intersection of $c_1$ and $c_2$.
The {\defines intersection index} $\ii_p(\tc_1, \tc_2)$ of $p$ is an integer given by
\[\ii_p(\tc_1, \tc_2) := \tc_1(t_1) \cdot \kappa_{12} \cdot \tc_2(t_2)^{-1} = k\pi
\in \pi_1(\mathbb{P}(T_p\Surf)) \cong \ZZ, \]
where $\tc_i(t_i)$ is the homotopy classes of paths from $\F(p)$ to the tangent space $\dot{c}_i(t_i)$ ($i\in \{1,2\}$) and $\kappa_{12}$ is that of paths from $\dot{c}_1(t_1)$ to $\dot{c}_2(t_2)$ given by clockwise rotation in $T_p\Surf$ by an angle $<\pi$,
and $\pi_1(\mathbb{P}(T_p\Surf))$  is the fundamental group defined on the projectivized of the tangent space $T_p\Surf$.

A {\defines graded full formal arc system} (=$\gbullet$-grFFAS) is an $\gbullet$-FFAS whose elements are graded $\gbullet$-curves. Similarly, we can define $\rbullet$-grFFAS.

\begin{definition}[Marked ribbon surfaces] \rm
\begin{itemize}
  \item[ ]
  \item[(1)] A {\defines marked ribbon surface} $\SURF=(\Surf, \M, \Y, \Dblue, \Dred)$
    is a marked surface $(\Surf, \M, \Y)$ with $\gbullet$-FFAS $\Dblue$ and $\rbullet$-FFAS $\Dred$
    such that $\Dred$ is the {\defines dual dissection} of $\SURF$ decided by $\Dblue$,
    that is, for any $\rbullet$-curve $a_{\rbullet}$, there is an unique $\gbullet$-curve $a_{\gbullet}$ intersecting with $a_{\rbullet}$. Moreover,
    $a_{\gbullet}$ and $a_{\rbullet}$ has only one intersection. We say $a_{\gbullet}$ is the {\defines dual arc} of $a_{\rbullet}$.

  \item[(2)] A {\defines graded marked ribbon surface} $\SURF^{\F}=(\Surf^{\F}, \M, \Y, \Dblue, \tDred)$ is a marked ribbon surface with a foliation $\F$
    such that all $\rbullet$-arcs in $\rbullet$-FFAS are graded $\rbullet$-curves.
\end{itemize}
\end{definition}

Any graded marked ribbon surface $\SURF^{\F}$ defines a graded algebra by the following construction.

\begin{construction} \label{construction} \rm
The graded algebra $A(\SURF^{\F})$ of $\SURF^{\F}$ is the finite dimensional algebra $\kk\Q/\I$ with grading $|\cdot|: \Q \to \ZZ$ given by the following steps.
\begin{itemize}
  \item[Step 1]
    There is a bijection $\mathfrak{v}: \Dblue \to \Q_0$, i.e., without causing confusion, $\Q_0 = \Dblue$.
  \item[Step 2]
    Any elementary $\gbullet$-polygon $\PP$ given by $\Dblue$ provides an arrow $\alpha: \mathfrak{v}(a^1_{\gbullet}) \to \mathfrak{v}(a^2_{\gbullet})$,
    where $a^1_{\gbullet}, a^2_{\gbullet}\in \Dblue$ are two edges of $\PP$ with common endpoints $p\in \M$
    and $a^2_{\gbullet}$ is left to $a^1_{\gbullet}$ at the point $p$.
  \item[Step 3]
    the grading $|\alpha|$ of $\alpha$ equals to the intersection index $1-\ii_p(\ta^1_{\rbullet}, \ta^2_{\rbullet})$, where $\ta^i_{\rbullet}$ is the dual arc of $\ta^i_{\gbullet}$ ($i\in\{1,2\}$).
  \item[Step 4]
  the ideal $\I$ is generated by $\alpha\beta$, where $\mathfrak{v}^{-1}(s(\alpha)), \mathfrak{v}^{-1}(t(\alpha))=\mathfrak{v}^{-1}(s(\beta)),\mathfrak{v}^{-1}(t(\beta))$
are edges of the same elementary $\gbullet$-polygon.

\end{itemize}
\end{construction}

\subsection{Permissible curves and admissible curves} \label{sect:admcurve}

In this subsection, we recall the definitions of permissible curves and admissible curves. Firstly, we introduce the notions of $\Dblue$-arc segment and $\Dred$-arc segment.
An {\defines $\Dblue$-arc segment} in $\SURF^{\F}$ is a homotopy class of segments in some elementary $\gbullet$-polygon which have three cases shown in \Pic \ref{fig:arc segment I},
and an {\defines $\Dred$-arc segment} is a homotopy class of segments in some elementary $\rbullet$-polygon which have two cases shown in \Pic \ref{fig:arc segment II},

\begin{multicols}{2}
\begin{figure}[H]
\definecolor{ffqqqq}{rgb}{1,0,0}
\definecolor{bluearc}{rgb}{0,0,1}
\begin{tikzpicture}
\draw[black] (-0.5,2)--( 0.5,2) [line width=1pt];
\draw[bluearc] ( 0, 2)--(-1, 0) [line width=1pt];
\draw[bluearc] ( 0, 2)--( 1, 0) [line width=1pt];
\fill[bluearc] ( 0, 2) circle (0.1cm);
\fill[bluearc] (-1, 0) circle (0.1cm);
\draw[orange][line width=1pt] (-1, 0) -- ( 0.5, 1);
\draw (0,-0.5) node{Case (1)};
\end{tikzpicture}
\begin{tikzpicture}
\draw[black] (-0.5,2)--( 0.5,2) [line width=1pt];
\draw[bluearc] ( 0, 2)--(-1, 0) [line width=1pt];
\draw[bluearc] ( 0, 2)--( 1, 0) [line width=1pt];
\fill[bluearc] ( 0, 2) circle (0.1cm);
\draw[orange][line width=1pt] (-0.5, 1) to[out=-45, in=-135] ( 0.5, 1);
\draw (0,-0.5) node{Case (2)};
\end{tikzpicture}
\begin{tikzpicture}
\draw[black] (0,2) to[out=180,in=90] (-2,0) [line width=1pt];
\draw[bluearc] ( 0, 2) to[out=-90,in=0] (-2, 0) [line width=1pt];
\fill[bluearc] ( 0, 2) circle (0.1cm);
\fill[bluearc] (-2, 0) circle (0.1cm);
\draw[red] (-1.41, 1.41) circle (0.1cm) [line width=1pt];
\draw[orange] (-1.41, 1.41) -- (-0.57, 0.57) [line width=1pt];
\draw (-1,-0.5) node{Case (3)};
\end{tikzpicture}
\caption{$\Dblue$-arc segments. }
\label{fig:arc segment I}
\end{figure}

\begin{figure}[H]
\definecolor{ffqqqq}{rgb}{1,0,0}
\definecolor{bluearc}{rgb}{0,0,1}
\begin{tikzpicture} [scale=0.83]
\draw[black] (-0.5*1.5, -0.86*1.5) -- ( 0.5*1.5, -0.86*1.5) [line width=1pt];
\draw[red] ( 0.5*1.5,-0.86*1.5) -- ( 1*1.5, 0) [line width=1pt];
\draw[red] ( 1*1.5, 0) -- ( 0.5*1.5, 0.86*1.5) [line width=1pt][dotted];
\draw[red] (-0.5*1.5, 0.86*1.5) -- ( 0.5*1.5, 0.86*1.5) [line width=1pt];
\draw[red] (-1*1.5, 0) -- (-0.5*1.5, 0.86*1.5) [line width=1pt][dotted];
\draw[red] (-0.5*1.5,-0.86*1.5) -- (-1*1.5, 0) [line width=1pt];
\fill[red] ( 0.5*1.5,-0.86*1.5) circle (0.1cm);
\fill[white] ( 0.5*1.5, -0.86*1.5) circle (1.8pt);
\fill[red] ( 1*1.5, 0) circle (0.1cm);
\fill[white] ( 1*1.5, 0) circle (1.8pt);
\fill[red] ( 0.5*1.5, 0.86*1.5) circle (0.1cm);
\fill[white] ( 0.5*1.5, 0.86*1.5) circle (1.8pt);
\fill[red] (-0.5*1.5, -0.86*1.5) circle (0.1cm);
\fill[white] (-0.5*1.5, -0.86*1.5) circle (1.8pt);
\fill[red] (-1*1.5, 0) circle (0.1cm);
\fill[white] (-1*1.5, 0) circle (1.8pt);
\fill[red] (-0.5*1.5, 0.86*1.5) circle (0.1cm);
\fill[white] (-0.5*1.5, 0.86*1.5) circle (1.8pt);
\fill[bluearc] (0,-0.86*1.5) circle (0.1cm) [line width=1pt];
\draw[violet] (0,-0.86*1.5) -- (0, 0.86*1.5) [line width=1pt];
\draw (0,-1.7) node{Case (1)};
\end{tikzpicture}
\begin{tikzpicture} [scale=0.83]
\draw[red] ( 0.5*1.5,-0.86*1.5) -- ( 1*1.5, 0) [line width=1pt][dotted];
\draw[red] ( 1*1.5, 0) -- ( 0.5*1.5, 0.86*1.5) [line width=1pt];
\draw[red] (-0.5*1.5, 0.86*1.5) -- ( 0.5*1.5, 0.86*1.5) [line width=1pt][dotted];
\draw[red] (-1*1.5, 0) -- (-0.5*1.5, 0.86*1.5) [line width=1pt];
\draw[red] (-0.5*1.5,-0.86*1.5) -- (-1*1.5, 0) [line width=1pt][dotted];
\fill[red] ( 0.5*1.5,-0.86*1.5) circle (0.1cm);
\fill[white] ( 0.5*1.5, -0.86*1.5) circle (1.8pt);
\fill[red] ( 1*1.5, 0) circle (0.1cm);
\fill[white] ( 1*1.5, 0) circle (1.8pt);
\fill[red] ( 0.5*1.5, 0.86*1.5) circle (0.1cm);
\fill[white] ( 0.5*1.5, 0.86*1.5) circle (1.8pt);
\fill[red] (-0.5*1.5, -0.86*1.5) circle (0.1cm);
\fill[white] (-0.5*1.5, -0.86*1.5) circle (1.8pt);
\fill[red] (-1*1.5, 0) circle (0.1cm);
\fill[white] (-1*1.5, 0) circle (1.8pt);
\fill[red] (-0.5*1.5, 0.86*1.5) circle (0.1cm);
\fill[white] (-0.5*1.5, 0.86*1.5) circle (1.8pt);
\draw[violet] (-1.15, 0.86*0.75) -- ( 1.15, 0.86*0.75) [line width=1pt];
\draw (0,-1.7) node{Case (2)};
\end{tikzpicture}
\caption{$\Dred$-arc segments}
\label{fig:arc segment II}
\end{figure}
\end{multicols}

\begin{definition} \label{def:permissible-admissible}
\rm Let $\SURF^{\F}$ be a graded marked ribbon surface and $c: [0,1] \to \Surf$ be a curve in $\SURF^{\F}$.
\begin{itemize}
\item[(1)]
  $c$ is called {\defines permissible} if it is a sequence of $\Dblue$-arc segments $\{c_{(i,i+1)}\}_{0\le i\le m(c)}$ ($m(c)\in\NN$) such that
  \begin{itemize}
    \item two adjacent $\Dred$-arc segments $c_{(i,i+1)}$ and $c_{(i+1,i+2)}$ are different,
      that is, $c_{(i,i+1)}\ne c_{(i+1,i+2)}$ and $c_{(i,i+1)}\ne c_{(i+1,i+2)}^{-1}$ hold;
    \item $c(t_i)=c_{(i,i+1)}(t_i)$ ($0\le i \le m(c)$)
      and $c(t_{m(c)+1})=c_{(m(c), m(c)+1)}(t_{m(c)+1})$, where $0=t_0 <t_1 <\cdots <t_{m(c)} < t_{m(c)+1} = 1$.
  \end{itemize}
  Moreover, we say that any curve $c: [0,1] \to \Surf$ with $c(0), c(1) \in \M\cap\E$ crossing no $\gbullet$-arc is a {\defines trivial permissible curve}.

\item[(2)]
  The graded curve $\tc$ in $\SURF^{\F}$ is {\defines admissible} if $c: [0,1] \to \Surf$ is a sequence of $\Dred$-arc segments $\{c_{[i,i+1]}\}_{0\le i\le n(\tc)}$ ($n(\tc)\in\NN$) such that
  \begin{itemize}
    \item two adjacent $\Dblue$-arc segments $c_{[i,i+1]}$ and $c_{[i+1,i+2]}$ are different;
    \item $c(t_i)=c_{[i,i+1]}(t_i)$ ($0\le i \le n(\tc)$)
      and $c(t_{n(\tc)+1})=c_{[n(\tc), n(\tc)+1]}(t_{n(\tc)+1})$,
      where $0=t_0 <t_1 <\cdots <t_{n(\tc)} < t_{n(\tc)+1} = 1$.
  \end{itemize}
\end{itemize}
We denote by $\dualgreen{c}{i}$ the $i$-th $\gbullet$-arc crossed by permissible curve $c$,
and denote by $\ared{c}{i}$ the $i$-th $\rbullet$-arc crossed by admissible curve $\tc$.
\end{definition}

For simplicity, we always assume that $c$ and $c^{-1}:=c(1-t)$ are the same permissible curve,
and denote by $0$ the trivial permissible curve. In this paper, we use the following notations.
\begin{itemize}
  \item $\PC_{\m}(\SURF^{\F})$: the set of all permissible curves with endpoints lying in $\M\cup\E$;
  \item $\PC_{\oslash}(\SURF^{\F})$: the set of all permissible curves without endpoints (up to homotopy);
  \item $\AC_{\m}(\SURF^{\F}):=\AC^{\m}_{\m}(\SURF^{\F})\cup\AC^{\oslash}_{\m}(\SURF^{\F})$, where
  \begin{itemize}
    \item $\AC^{\m}_{\m}(\SURF^{\F})$: the set of all admissible curves with endpoints lying in $\M$;
    \item $\AC^{\oslash}_{\m}(\SURF^{\F})$: the set of all admissible curves with only one endpoints and lies in $\M$;
  \end{itemize}
  \item $\AC^{\oslash}_{\oslash}(\SURF^{\F})$: the set of all admissible curves without endpoints (up to homotopy).
\end{itemize}

Let $A$ be a gentle algebra. The following theorem shows that all indecomposable objects in $\modcat A$ and $\per A$ can be described by permissible curves and admissible curves, respectively.

\begin{theorem} \label{thm:OPS and BCS corresponding}
Let $\mathscr{J}$ be the set of all Jordan blocks with non-zero eigenvalue.
\begin{itemize}
  \item[\rm(1)] {\rm \cite[Theorems 3.8 and 3.9]{BCS2019}}
    There exists a bijection
    \[ \MM: \PC_{\m}(\SURF(A)^{\F_A}) \cup (\PC_{\oslash}(\SURF(A)^{\F_A})\times\mathscr{J}) \to \ind(\modcat A). \]
    between the set $\PC_{\m}(\SURF(A)^{\F_A}) \cup (\PC_{\oslash}(\SURF(A)^{\F_A})\times\mathscr{J})$ of all permissible curves and the set $\ind(\modcat A)$ of all isoclasses of indecomposable modules in $\modcat A$.
    We denote by $\MM^{-1}$ the quasi-equivalence of $\MM$.

  \item[\rm(2)]{\rm \cite[Theorem 2.12]{OPS2018}}
    There exists a bijection
    \[\X: \AC_{\m}(\SURF(A)^{\F_A}) \cup (\AC^{\oslash}_{\oslash}(\SURF(A)^{\F_A}) \times \mathscr{J}) \to \ind(\per A). \]
    between the set $\AC_{\m}(\SURF(A)^{\F_A}) \cup (\AC^{\oslash}_{\oslash}(\SURF(A)^{\F_A}) \times \mathscr{J})$ and the set of all isoclasses of indecomposable objects in $\per A$.
    We denote by $\X^{-1}$ the quasi-equivalence of $\X$.
\end{itemize}
\end{theorem}

\noindent
For any set $S$ of some permissible curves (resp. admissible curves),
we define $\MM(S) = \bigoplus_{c\in S}\MM(c)$ (resp. $\X(S) = \bigoplus_{\tc\in S}\X(\tc)$).

The indecomposable object corresponding to $c\in\PC_{\m}(\SURF(A)^{\F_A})$
(resp. $\tc\in\AC_{\m}(\SURF(A)^{\F_A})$) is called a {\defines string module} (resp. {\defines string complex}).

The indecomposable object corresponding to $b\in\PC^{\oslash}_{\oslash}(\SURF(A)^{\F_A}) \ne \varnothing$
(resp. $\tilde{b} \in \AC_{\oslash}(\SURF(A)^{\F_A})$)
with Jordan block $\pmb{J}_{n}(\lambda)$ ($\lambda\ne 0$ is the eigenvalue of $\pmb{J}_{n}(\lambda)$)
is called a {\defines band module} (resp. {\defines band complex}).

\subsection{Top, socle and projective permissible curves}

In this subsection, we recall the definitions of top, socle and projective permissible curves in \cite{ZLpre}.

\begin{definition}\rm
Let $x$ be the common endpoint of $\agreen{c}{i-1}$ and $\agreen{c}{i}$ and $y$ be the common endpoint of $\agreen{c}{i}$ and $\agreen{c}{i+1}$, for any $0\leq i\leq m(c)$. Then
\begin{itemize}
 \item[(1)] $\agreen{c}{i}$ is called a {\defines top $\gbullet$-arc of $c$} if $x$ is on the right of $c$ and $y$ is on the left of $c$.

 \item[(2)] $\agreen{c}{i}$ is called a {\defines socle $\gbullet$-arc of $c$} if $x$ is on the left of $c$ and $y$ is on the right of $c$.
 \end{itemize}
Denote by $\top(c)$ (resp., $\soc(c)$) the set of all top (resp., socle) $\gbullet$-arcs of $c$.
\end{definition}

\begin{remark} \rm
If $c$ crosses only one $\gbullet$-arc, then this $\gbullet$-arc is not only a top but also a socle $\gbullet$-arc of $c$.
\end{remark}

Note that the following isomorphisms hold.
\begin{align}
  \MM(\top (c)) \cong \top(\MM(c)) \text{ and }
  \MM(\soc (c)) \cong \soc(\MM(c)). \label{formula:top and soc}
\end{align}

\begin{definition}\rm \label{def:projedtive curve}
We say $c\in\PC_{\m}(\SURF^{\F})$, denote by $c_{\proj}(a_{\gbullet})$, is a {\defines projective curve corresponding to $a_{\gbullet}=\agreen{c}{i}$},
if there is an unique integer $1\le i\le m(c)$ such that:
\begin{itemize}
  \item[(1)]
    $\agreen{c}{1}$, $\agreen{c}{2}$, $\ldots$, $\agreen{c}{i}$ have a common endpoint $x$ which is right to $c$, and
        there is no arc $a\in\Dblue$ with endpoint $x$ such that $a$ is left to the $\agreen{c}{1}$ at the point $x$;
  \item[(2)]
    $\agreen{c}{i}$, $\agreen{c}{i+1}$, $\ldots$, $\agreen{c}{m(c)}$ have a common endpoint $y$ which is left to $c$, and there is no arc $\hat{a}\in\Dblue$ with endpoint $y$ such that $\hat{a}$ is left to the $\agreen{c}{m(c)}$ at the point $y$.
\end{itemize}
\end{definition}

By Definition \ref{def:projedtive curve}, $\MM(c_{\proj}(a_{\gbullet}))$ is an indecomposable projective module.

For any $\gbullet$-arc $a_{\gbullet}$ crossed by $c\in\PC_{\m}(\SURF^{\F})$, if there exists a $\gbullet$-arc $a_{\gbullet}'$ such that
it is left to $a_{\gbullet}$ at the intersection $p$ in $a_{\gbullet}\cap a_{\gbullet}'$ and there is no other $\gbullet$-arc between $a_{\gbullet}$ and $a_{\gbullet}'$ at $p$, then we define $\overleftarrow{a_{\gbullet}}=a_{\gbullet}'$. Otherwise, we define $\overleftarrow{a_{\gbullet}}=\varnothing$.

If $a_{\gbullet}$ satisfies (P1),
assume $c=c_{\proj}(a_{\gbullet})=\{c_{(t,t+1)}\}_{1\le t\le m(c)}$ such that
$p$ is left to the $c_{\proj}(a_{\gbullet})$ and $c_{(t_{i-1}, t_i)}$ is the $\Dblue$-arc segment
obtained by $a_{\gbullet}$ and $a_{\gbullet}'$ cutting $c$.
In this case, we have $i\ge 2$. We denote by $\omega(\overleftarrow{a_{\gbullet}})$
the permissible curve $\{c_{(t,t+1)}\}_{0\le t\le i-2}\cup\{c_{(i-1,i)}'\}$ of $c$,
where $c_{(i-1,i)}': [t_{i-1}, t_i] \to \Surf$ is the $\Dblue$-arc segment
between $a_{\gbullet}'$ and $a_{\gbullet}$ such that $c_{(i-1,i)}'(t_{i-1})=c_{(i-1,i)}(t_{i-1})$
and $c_{(i-1,i)}'(t_i)$ is an endpoint of $a_{\gbullet}$ (see \Pic \ref{fig:omega}).
If $a_{\gbullet}$ satisfies (P2), we set $\omega(\overleftarrow{a_{\gbullet}})$ is trivial.
In this case $\MM(\omega(\overleftarrow{a_{\gbullet}}))=0$.

\begin{figure}[H]
\definecolor{ffqqqq}{rgb}{1,0,0}
\definecolor{bluearc}{rgb}{0,0,1}
\begin{tikzpicture}
\draw[line width=1pt][->] (-2,-1)--(2,-1);
\draw[line width=1pt][<-] (-2, 1)--(2, );
\draw (-2,-1) node[left]{$\bSurf$};
\fill[bluearc] (0,-1) circle (0.1cm) [line width=1pt] node[below]{$q$};
\fill[bluearc] (0, 1) circle (0.1cm) [line width=1pt] node[above]{$p$};
\draw[bluearc][line width=1pt] (0,-1)--(0,1);
\draw[bluearc][line width=1pt][rotate around={ 45.0:(0,-1)}] (0,-1)--(0,1);
\draw[bluearc][line width=1pt][rotate around={ 67.5:(0,-1)}][dotted] (0,-1)--(0,1);
\draw[bluearc][line width=1pt][rotate around={ 45.0:(0,1)}] (0,-1)--(0,1);
\draw[bluearc][line width=1pt][rotate around={ 67.5:(0,1)}][dotted] (0,-1)--(0,1);
\draw (0,-0.5) node[right]{$a_{\gbullet}$};
\draw [rotate around={ 45.0:(0,-1)}](0,1) node[left]{$a_{\gbullet}'=\overleftarrow{a_{\gbullet}}$};
\fill[white] (-1,-1) circle (1.8pt); \fill[blue] (-1,-1) node{$\pmb{\diamond}$};
\fill[white] ( 1, 1) circle (1.8pt); \fill[blue] ( 1, 1) node{$\pmb{\diamond}$};
\draw[orange][line width=1pt] (-1,-1) -- (1,1);
\draw[orange] (0.25,0.25) node[right]{\tiny$c_{\proj}(a_{\gbullet})$};
\draw[orange][line width=1pt] (-1,-1) -- (0,1) [dash pattern=on 2pt off 2pt];
\draw[orange] (-0.25,0.3) node[left]{\tiny$\omega(a_{\gbullet})$};
\end{tikzpicture}
\caption{The point ``${\color{blue}\pmb{\diamond}}$'' is either a $\gbullet$-marked point
or an extra marked point. }
\label{fig:omega}
\end{figure}

The following lemma comes from \cite[Lemma 3.4]{ZLpre}.

\begin{lemma} \label{lemm:proj cover}
Let $\widehat{\soc(c)}$ be the set of all socle $\gbullet$-arcs of $c\in\PC_{\m}(\SURF^{\F})$ except $c_{\proj}(\agreen{c}{1})$ and $c_{\proj}(\agreen{c}{m(c)})$.
Then the projective cover of $\MM(c)$ is
\[ p_0: P_0 = \bigoplus_{a_{\gbullet}\in\top(c)}\MM(c_{\proj}(a_{\gbullet})) \to \MM(c), \]
and the kernel $\ker p_0$ of $p_0$ is isomorphic to $Q_L\oplus Q\oplus Q_R$, where:
\[Q_L \cong \MM(\omega(\overleftarrow{\agreen{c}{1}})),\
Q \cong \bigoplus\limits_{a_{\gbullet}\in \widehat{\soc(c)}} \MM(c_{\proj}(a_{\gbullet}))\
\text{ and }~~Q_R \cong \MM(\omega(\overleftarrow{\agreen{c}{m(c)}})). \]
\end{lemma}

\section{The unification of geometric models for gentle algebras} \label{sec:embedding}

In this subsection, we study the projective resolution of indecomposable modules of a gentle algebra and give an embedding of geometric models for a gentle algebra from module category to derived category.

\subsection{The admissible curves corresponding to string modules}

Let $c$ be a permissible curve. By Definition \ref{def:permissible-admissible}, the $\Dblue$-arc segment $c_{(0,1)}$ (or $c_{(m(c), m(c)+1)}$) of $c$
is one of the two Cases (1) and (3) shown in \Pic \ref{fig:arc segment I}.
Let $\PP_L$ and $\PP_R$ be the elementary $\gbullet$-polygons such that
$c_{(0,1)}\subset \PP_L$ and $c_{(m(c), m(c)+1)} \subset \PP_R$, respectively.
Then $\agreen{c}{1}\in \Egreen(\PP_L)$ and $\agreen{c}{m(c)}\in \Egreen(\PP_R)$.
In order to give a characterization of the projective resolution of indecomposable modules, we need to define the rotating of permissible curves.
For permissible curve $c$,
\begin{itemize}
  \item If $c_{(0,1)}$ belongs to Case (1) shown in \Pic \ref{fig:arc segment I}, then we move the endpoint $c(0)$
    along the edges lying in $\Egreen(\PP_L)$ of $\PP_L$ in a clockwise direction as far as possible.
  \item If $c_{(m(c), m(c)+1)}$ belongs to Case (1) shown in \Pic \ref{fig:arc segment I}, then we move the endpoint $c(1)$
    along the edges lying in $\Egreen(\PP_R)$ of $\PP_R$ in a clockwise direction as far as possible.
  \item If $c_{(0,1)}$ belongs to Case (3) shown in \Pic \ref{fig:arc segment I},
    then we move the endpoint $c(0)$ to the next $\gbullet$-marked point along the positive direction\footnote{In our paper, we suppose that the positive direction of the boundary $\partial S$ of a surface/polygon $S$ is the following walking direction: walking along $\partial S$, the inner of $S$ is on the left.} of boundary $\bSurf$.
  \item If $c_{(m(c), m(c)+1)}$  belongs to Case (3) shown in \Pic \ref{fig:arc segment I},
    then we move the endpoint $c(1)$ to the next $\gbullet$-marked point along the positive direction of boundary $\bSurf$.
\end{itemize}

\begin{definition} \rm \label{def:rot} Let $c$ be a permissible curve. Then {\defines the $\gbullet$-curve $c^{\rota}$} is defined to be the curve obtained by rotating the segments $c_{(0,1)}$ and $c_{(m(c), m(c)+1)}$ of $c$ as above.
\end{definition}

Now we can give the main result of this subsection.

\begin{theorem} \label{thm:string embedding}
Let $A$ be a gentle algebra and $\SURF(A)^{\F_A}$ be its graded marked ribbon surface.
Then there is an embedding
\[(-)^\rota: \PC_{\m}(\SURF(A)^{\F_A}) \to \AC_{\m}(\SURF(A)^{\F_A}),
c\mapsto \tc^{\rota}\]
such that $\H^{0}(\X(\tc^{\rota}))$ is isomorphic to $\MM(c)$ and $\H^{i}(\X(\tc^{\rota}))=0$ for any $i\neq 0$.
\end{theorem}

\begin{proof} We only prove for the following two cases.

Case 1.  $c_{(0,1)}$ and $c_{(m(c),m(c)+1)}$ belong to Case (1) shown in \Pic \ref{fig:arc segment I}. Assume that the projective resolution of $\MM(c)$ is
\[ \cdots \mathop{\longrightarrow}\limits^{p_3} P_2
          \mathop{\longrightarrow}\limits^{p_2} P_1
          \mathop{\longrightarrow}\limits^{p_1} P_0
          \mathop{\longrightarrow}\limits^{p_0} \MM(c) \longrightarrow 0, \]
and define $K_t=\ker p_{t-1}$ to be the $t$-th syzygy of $\MM(c)$.
For simplicity, we denote by $a^{1}_{L}=\agreen{c}{1}$ and $a^{1}_{R}=\agreen{c}{m(c)}$.
Moreover, we assume that $a^{t+1}_{\ast}=\overleftarrow{a^{t}_{\ast}}$ for $t\geq 1$ and $\ast\in\{L, R\}$.
If $\PP_{\ast}$ is not an $\infty$-elementary $\gbullet$-polygon,
then there is an integer $N_{\ast}$ such that $a^{N_{\ast}}_{\ast}\ne\varnothing$ and $a^{\ge N_{\ast}+1}_{\ast}=\varnothing$.

By Lemma \ref{lemm:proj cover}, we obtain that
$K_1=\ker p_0 \cong Q_{L,1}\oplus Q_1\oplus Q_{R,1}$, where
\[Q_1 = \bigoplus\limits_{a_{\gbullet}\in \widehat{\soc(c)}} \MM(c_{\proj}(a_{\gbullet}))
   \mathop{=\!=\!=\!=\!=}\limits^{\text{denote by}}
   \bigoplus_{j=1}^{M} \MM(c_{\proj}(a_{\gbullet}^j))\
   \text{is projective},\]
and
\begin{center}
$Q_{L,1} \cong \MM(\omega(\overleftarrow{\agreen{c}{1}})) = \MM(\omega(a^2_L))$ and
$Q_{R,1} \cong \MM(\omega(\overleftarrow{\agreen{c}{m(c)}})) = \MM(\omega(a^2_R))$
\end{center}
are string modules.
Then the projective cover of $K_1$ is of the form
\[ p_1 =
\left(\begin{smallmatrix}
p_{L,1} & & \\
 & \text{id}_{Q_1} & \\
 & & p_{R,1}
\end{smallmatrix}\right):
P_1 = P(Q_{L,1}) \oplus Q_1 \oplus P(Q_{R,1}) \longrightarrow Q_{L,1}\oplus Q_1\oplus Q_{R,1}, \]
where
\begin{align}
& P(Q_{L,1})
  = \bigoplus_{a_{\gbullet} \in \top(\omega({a^2_L}))} \MM(c_{\proj}(a_{\gbullet}))
  = \MM(c_{\proj}(a^2_L)), \nonumber \\
 \ \ \  & P(Q_{R,1})
  = \bigoplus_{a_{\gbullet} \in \top(\omega({a^2_R}))} \MM(c_{\proj}(a_{\gbullet}))
  = \MM(c_{\proj}(a^2_R)).  \nonumber
\end{align}
Thus, $K_2:=\ker p_1 = \ker p_{L,1}\oplus \ker p_{R,1}$.
Furthermore, by Lemma \ref{lemm:proj cover}, we have
\begin{center}
$\ker p_{L,1} = Q_{L,2}\oplus Q_2 \oplus Q_{R,2}$
and $\ker p_{R,1} = Q_{L,2}'\oplus Q_2' \oplus Q_{R,2}'$,
\end{center}
where $Q_{L,2}=\MM(\omega(a^{3}_{L}))$, $Q_2=Q_{R,2}=0$,
$Q_{L,2}'=Q_2'=0$ and $Q_{R,2}'=\MM(\omega(a^{3}_{R}))$, that is,
\[K_2 = \MM(\omega(a^{3}_{L})) \oplus \MM(\omega(a^{3}_{R})).\]
By induction, we obtain
\[K_t = \MM(\omega(a^{t+1}_{L})) \oplus \MM(\omega(a^{t+1}_{R})) \ \text{for all}\ t\ge 1. \]

Case 1.1. $\PP_L$ and $\PP_R$ are not $\infty$-elementary $\gbullet$-polygons. By hypothesis, there exists a positive integer $N_L$ such that
\begin{center}
$\MM(\omega(a^{\theta_L+1}_{L}))=0$ if and only if $\theta_L\ge N_L-1$ (see \Pic \ref{fig:syzygy}).
\end{center}
Similarly, $\MM(\omega(a^{\theta_R+1}_{R}))=0$ if and only if $\theta_R\ge N_R-1$ for positive integer $N_R$.

\begin{figure}[htbp] \tiny
\definecolor{ffqqqq}{rgb}{1,0,0}
\definecolor{bluearc}{rgb}{0,0,1}
\begin{tikzpicture}[scale=0.7]
\draw[line width=1pt][->] (-17,-4)--(1,-4);
\draw (1,-4) node[below]{$\bSurf$};
\draw[line width=1pt] (1,-4) arc(-90:90:4) [dash pattern=on 4pt off 2pt];
\draw[line width=1pt] (1,4) -- (-8,4) [dash pattern=on 4pt off 2pt];
\fill[bluearc][opacity=0.15] (-6,4)
 to[out=-90,in=180] ( -3, 1) to[out=  0,in=-90] (  0, 4) -- (  0,-4)
 to[out= 90,in=  0] ( -2,-2) to[out=180,in= 90] ( -4,-4)
 to[out= 90,in=  0] ( -6,-2) to[out=180,in= 90] ( -8,-4)
 to[out= 90,in=  0] (-10,-2) to[out=180,in= 90] (-12,-4)
 to[out= 90,in=  0] (-14,-2) to[out=180,in= 90] (-16,-4) -- (-17 ,-4) -- (-17, 4) -- ( -6, 4);
\draw (-16,0) node{$\PP_L$};
\fill[bluearc] (  0, 4) circle (0.1*1.5cm) [line width=1pt];
\fill[bluearc] (  0,-4) circle (0.1*1.5cm) [line width=1pt];
\fill[bluearc] ( -4,-4) circle (0.1*1.5cm) [line width=1pt];
\fill[bluearc] ( -8,-4) circle (0.1*1.5cm) [line width=1pt];
\fill[bluearc] (-12,-4) circle (0.1*1.5cm) [line width=1pt];
\fill[bluearc] (-16,-4) circle (0.1*1.5cm) [line width=1pt];
\draw[bluearc][line width=1pt] ( 0, 4) -- ( 0,-4); \draw[bluearc] (0,2) node[right]{$a^1_L$};
\draw[bluearc][line width=1pt] ( 0,-4) arc(0:180:2);
\draw[bluearc][line width=1pt] ( 0,-4) arc(0:180:1.5) [dotted];
\draw[bluearc][line width=1pt] ( 0,-4) arc(0:180:1);
\fill[bluearc] (-3,-4) circle (0.1*1.5cm) [line width=1pt];
\fill[bluearc] (-2,-4) circle (0.1*1.5cm) [line width=1pt];
\draw[bluearc] (-2-1,-2-0.25) node[below]{$a^2_L$};
\draw[bluearc][line width=1pt] ( 0-4,-4) arc(0:180:2);
\draw[bluearc][line width=1pt] ( 0-4,-4) arc(0:180:1.5) [dotted];
\draw[bluearc][line width=1pt] ( 0-4,-4) arc(0:180:1);
\fill[bluearc] (-3-4,-4) circle (0.1*1.5cm) [line width=1pt];
\fill[bluearc] (-2-4,-4) circle (0.1*1.5cm) [line width=1pt];
\draw[bluearc] (-2-4-1,-2-0.25) node[below]{$a^3_L$};
\draw[bluearc][line width=1pt] ( 0-8,-4) arc(0:180:2  ) [dotted];
\draw[bluearc][line width=1pt] ( 0-8,-4) arc(0:180:1.5) [dotted];
\draw[bluearc][line width=1pt] ( 0-8,-4) arc(0:180:1  ) [dotted];
\fill[bluearc] (-3-8,-4) circle (0.1*1.5cm) [line width=1pt] [dotted];
\fill[bluearc] (-2-8,-4) circle (0.1*1.5cm) [line width=1pt] [dotted];
\draw[bluearc][line width=1pt] ( 0-12,-4) arc(0:180:2);
\draw[bluearc][line width=1pt] ( 0-12,-4) arc(0:180:1.5) [dotted];
\draw[bluearc][line width=1pt] ( 0-12,-4) arc(0:180:1);
\fill[bluearc] (-3-12,-4) circle (0.1*1.5cm) [line width=1pt];
\fill[bluearc] (-2-12,-4) circle (0.1*1.5cm) [line width=1pt];
\draw[bluearc] (-14,-2) node[above]{$a^{N_L}_L$};
\draw[bluearc][line width=1pt] ( 0,4) arc(0:-180:3) [dotted];
%
\draw[red][line width=1pt] (-10 , 4) to[out=-120,in=90] (-3.5-12,-4);
\draw[red][line width=1pt] (-10 , 4) -- (-3.5-8,-4) [dotted];
\draw[red][line width=1pt] (-10 , 4) -- (-3.5-4,-4);
\draw[red][line width=1pt] (-10 , 4) to[out=-60,in=90] (-3.5,-4);
\draw[red][line width=1pt] (-10 , 4) to[out=-45,in=180] (-1.5,0.5) -- (0,0.5);
\draw[red][line width=1pt] (0,0.5) -- (1.5,0.5) [dotted];
\draw[red][line width=1pt] (-10 , 4) to[out=-20,in=180] (-2,3) [dotted];
\draw[red] (1.5,0.5) node[above]{$(a^1_L)^{\star}$};
\draw[red] (-4.5,-1) node[right]{$(a^2_L)^{\star}$};
\draw[red] (-8.2,-1) node[right]{$(a^3_L)^{\star}$};
\draw[red] (-12.5,1) node[left]{$(a^{N_L}_L)^{\star}$};
\fill[red] (-10 , 4)  circle (0.1*1.5cm); \fill[white] (-10 , 4) circle (0.1cm);
\fill[red] (-1  ,-4)  circle (0.1*1.5cm); \fill[white] (-1  ,-4) circle (0.1cm);
\fill[red] (-2.5,-4)  circle (0.1*1.5cm); \fill[white] (-2.5,-4) circle (0.1cm);
\fill[red] (-3.5,-4)  circle (0.1*1.5cm); \fill[white] (-3.5,-4) circle (0.1cm);
\fill[red] (-1-4  ,-4)  circle (0.1*1.5cm); \fill[white] (-1-4  ,-4) circle (0.1cm);
\fill[red] (-2.5-4,-4)  circle (0.1*1.5cm); \fill[white] (-2.5-4,-4) circle (0.1cm);
\fill[red] (-3.5-4,-4)  circle (0.1*1.5cm); \fill[white] (-3.5-4,-4) circle (0.1cm);
\fill[red] (-1-8  ,-4)  circle (0.1*1.5cm); \fill[white] (-1-8  ,-4) circle (0.1cm);
\fill[red] (-2.5-8,-4)  circle (0.1*1.5cm); \fill[white] (-2.5-8,-4) circle (0.1cm);
\fill[red] (-3.5-8,-4)  circle (0.1*1.5cm); \fill[white] (-3.5-8,-4) circle (0.1cm);
\fill[red] (-1-12  ,-4)  circle (0.1*1.5cm); \fill[white] (-1-12  ,-4) circle (0.1cm);
\fill[red] (-2.5-12,-4)  circle (0.1*1.5cm); \fill[white] (-2.5-12,-4) circle (0.1cm);
\fill[red] (-3.5-12,-4)  circle (0.1*1.5cm); \fill[white] (-3.5-12,-4) circle (0.1cm);
\draw[orange][line width=1pt] (-4,-4) to[out=90,in=180] (0,0);
\draw[orange][line width=1pt] (0,0) -- (4,0) [dotted];
\draw[orange] (2,0) node[above]{$c$};
\draw[cyan][line width=1pt] (-1,-4) to[out=90,in=-90] (0,4);
\draw[cyan] (-0.25,1) node[left]{$\omega(a^2_L)$};
\draw[cyan][line width=1pt] (-1-4,-4) arc(180:0:2.5);
\draw[cyan] (-2,-1.70) node[above]{$\omega(a^3_L)$};
\draw[cyan][line width=1pt] (-1-8,-4) arc(180:0:2.5) [dotted];
\draw[cyan][line width=1pt] (-1-12,-4) arc(180:0:2.5);
\draw[cyan] (-12,-1.7) node[above]{$\omega(a^{N_L}_L)$};
\draw[violet][line width=1pt] (-16,-4) to[out=90,in=180] (0,0.2);
\draw[violet][line width=1pt] (0,0.2) -- (3,0.2) [dotted];
\draw[violet] (3,0.2) node[above]{$\tau = c^{\circlearrowleft}$};
\end{tikzpicture}
\caption{If $a^1_L\in\top(c)$, then we have $a_{\rbullet}^1=(a^1_L)^{\star}$.
Otherwise, $\tau$ crosses no $(a^1_L)^{\star}$.}
\label{fig:syzygy}
\end{figure}
Note that $P_{\theta} = \MM(c_{\proj}(a_L^{\theta+1})) \oplus \MM(c_{\proj}(a_R^{\theta+1}))$.
Thus, for $\theta \ge \max\{N_L-1, N_R-1\}$, we have $K_{\theta}=0$ and $P_\theta =0$. Let $\tilde{\tau}$ be the admissible curve corresponding to the complex
\[\cdots \mathop{\longrightarrow}\limits^{p_3} P_2
          \mathop{\longrightarrow}\limits^{p_2} P_1
          \mathop{\longrightarrow}\limits^{p_1} P_0
          \longrightarrow 0. \]
Denote by $(a^t_*)^{\star}$ the dual arc of $a^{t}_*$. Then $\widetilde{(a^t_*)^{\star}}$ is a graded $\rbullet$-arc and $\tau$ consecutively crosses
\begin{center}
$(a^{N_L}_L)^{\star}$, $\ldots$, $(a^{3}_L)^{\star}$, $(a^{2}_L)^{\star}$,
$a_{\rbullet}^1$, $a_{\rbullet}^2$, $\ldots$, $a_{\rbullet}^M$,
$(a^{2}_R)^{\star}$, $(a^{3}_R)^{\star}$, $\ldots$, $(a^{N_R}_R)^{\star}$.
\end{center}
Thus, $\tau = c^{\circlearrowleft}$.

Case 1.2. $\PP_L$ is an $\infty$-elementary $\gbullet$-polygon (see \Pic \ref{fig:infty pdim}).
In this subcase, we have  $\MM(\omega(a^{t}_{L}))\ne 0$ for all $t\ge 1$
and $\omega(a^{t}_{L}) = \omega(a^{t+n_L}_{L})$, where $\sharp\Egreen(\PP_L)=n$.
We can consider $\PP_R$ by the same way. Thus, $\tau$ consecutively crosses
\begin{center}
  $\ldots$,
  $(a^{n}_L)^{\star}$, $\ldots$, $(a^{2}_L)^{\star}$, $(a^{1}_L)^{\star}$,
  $(a^{n}_L)^{\star}$, $\ldots$, $(a^{2}_L)^{\star}$,
  $a_{\rbullet}^1$, $a_{\rbullet}^2$, $\ldots$, $a_{\rbullet}^M$,
  $(a^{2}_R)^{\star}$, $(a^{3}_R)^{\star}$, $\ldots$.
\end{center}

\begin{figure}[H]
\centering
\definecolor{ffqqqq}{rgb}{1,0,0}
\definecolor{bluearc}{rgb}{0,0,1}
\begin{tikzpicture}[scale=0.75] \small
\filldraw[black!20] (0,0) circle (0.3cm);
\draw[line width=1.2pt] (0,0) circle (0.3cm);
\draw[bluearc][line width=1.2pt] (1.73,1) -- (0, 2) -- (-1.73, 1);
\draw[bluearc][line width=1.2pt, dotted] (-1.73, 1) -- (-1.73,-1);
\draw[bluearc][line width=1.2pt] (-1.73,-1) -- (0,-2) -- (1.73,-1) -- (1.73,1);
\draw[bluearc] (-0.89, 1.45) node[above]{$a^n_L$};
\draw[bluearc] ( 0.89, 1.45) node[above]{$a^1_L$};
\draw[bluearc] ( 1.73, 0.00) node[right]{$a^2_L$};
\draw[bluearc] ( 0.89,-1.45) node[below]{$a^3_L$};
\draw[bluearc] (-0.89,-1.45) node[below]{$a^4_L$};
\draw[bluearc] (-1.73, 0.00) node[left]{$\vdots$};
\fill [bluearc] ( 1.73, 1) circle (2.8pt);
\fill [bluearc] ( 0.00, 2) circle (2.8pt);
\fill [bluearc] (-1.73, 1) circle (2.8pt);
\fill [bluearc] (-1.73,-1) circle (2.8pt);
\fill [bluearc] ( 0.00,-2) circle (2.8pt);
\fill [bluearc] ( 1.73,-1) circle (2.8pt);
\fill [bluearc] ( 1.73, 1) circle (2.8pt);
\draw [red][line width=1.2pt] (0,0.3) -- (-0.89, 1.45);
\draw [red][line width=1.2pt] (0,0.3) -- ( 0.89, 1.45);
\draw [red][line width=1.2pt] (0,0.3) to[out= 40,in=140] ( 1.73, 0.00);
\draw [red][line width=1.2pt] (0,0.3) to[out= 20,in= 70] ( 0.89,-1.45);
\draw [red][line width=1.2pt] (0,0.3) to[out=160,in=110] (-0.89,-1.45);
\draw [red][line width=1.2pt] (0,0.3) to[out=140,in= 40] (-1.73, 0.00) [dotted];
\fill [red] (0,0.3) circle (2.8pt); \fill [white] (0,0.3) circle (2pt);
\draw (0,0) node{$b$};
\draw[orange][line width=1.2pt] (0, 2.5) node[left]{$c$} to[out=0, in=165] (1,2.4) [dotted];
\draw[orange][line width=1.2pt] (1, 2.4) to[out=-15, in=180] (1.73,-1);
\draw[violet][line width=1.2pt] (0, 2.3) node[left]{$\tau$} to[out=0, in=165] (1,2.2) [dotted];
\draw[violet][line width=1.2pt] (1, 2.2)
  to[out=-15, in= 90] ( 1,0) to[out=-90,in=  0] (0,-1)
  to[out=180, in=-90] (-1,0) to[out= 90,in=180] (0, 1)
  to[out=  0, in= 90] ( 0.8,0) to[out=-90,in=  0] (0,-0.8)
  to[out=180, in=-90] (-0.8,0) to[out= 90,in=180] (0, 0.8)
  to[out=  0, in= 90] ( 0.6,0) to[out=-90,in=  0] (0,-0.6)
  to[out=180, in=-90] (-0.6,0) to[out= 90,in=180] (0, 0.6);
\draw[violet][line width=1.2pt][dotted] (0, 0.6)
  to[out=  0, in= 90] ( 0.45,0) to[out=-90,in=  0] (0,-0.45)
  to[out=180, in=-90] (-0.45,0) to[out= 90,in=180] (0, 0.45);
\end{tikzpicture}
\caption{String module with infinite projective dimension.}
\label{fig:infty pdim}
\end{figure}
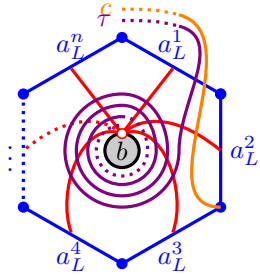

\begin{figure}[H]
\begin{center}
\definecolor{ffqqqq}{rgb}{1,0,0}
\definecolor{bluearc}{rgb}{0,0,1}
\begin{tikzpicture}[yscale=1.3]
\draw[black] (0,2) to[out=180,in=90] (-2,0) [line width=1pt];
\draw[black] (-1.5,0.5) node{$\PP_L$};
\draw[bluearc] ( 0, 2) to[out=-90,in=0] (-2, 0) [line width=1pt];
\draw[bluearc] ( 0, 2) -- (0,-0.5) [line width=1pt];
\fill[bluearc] ( 0, 2) circle (0.1cm);
\fill[bluearc] (-2, 0) circle (0.1cm);
\draw[red] (-1.41, 1.41) circle (0.1cm) [line width=1pt];
\draw[orange] (-1.41, 1.41) to[out=-45,in=180] (0.5,0) [line width=1pt];
\draw (-1,-0.5) node{ (I)};
\end{tikzpicture}
\ \ \ \
\begin{tikzpicture}[yscale=1.3]
\draw[black] (0,2) to[out=180,in=90] (-2,0) [line width=1pt];
\draw[black] (-1.5,0.5) node{$\PP_L$};
\draw[bluearc] ( 0, 2) to[out=-90,in=0] (-2, 0) [line width=1pt];
\draw[bluearc] (-2, 0) -- (0.5,0) [line width=1pt];
\fill[bluearc] ( 0, 2) circle (0.1cm);
\fill[bluearc] (-2, 0) circle (0.1cm);
\draw[red] (-1.41, 1.41) circle (0.1cm) [line width=1pt];
\draw[orange] (-1.41, 1.41) to[out=-45,in=90] (0,-0.5) [line width=1pt];
\draw (-1,-0.5) node{ (II)};
\end{tikzpicture}
\caption{Two forms of $\Dblue$-segment $c_{(0,1)}$ of $c$}
\label{fig:c(0,1)}
\end{center}
\end{figure}

Case 2. At least one of $c_{(0,1)}$ and $c_{(m(c),m(c)+1)}$ belongs to
Case (3) shown in \Pic \ref{fig:arc segment I}.
We only prove this case for $c_{(0,1)}$, the proof of $c_{(m(c),m(c)+1)}$ is similar. In this case, $c_{(0,1)}$ has two forms as shown in \Pic \ref{fig:c(0,1)}.

For (I), we have $\agreen{c}{1}\in\top(c)$ and $\MM(c_{\proj}(\agreen{c}{1}))$ is an indecomposable projective direct summand of $P_0$.
It follows that $Q_L=0$ for $\ker p_0$. Thus, $\tau$ crosses the dual $\rbullet$-arc $(\agreen{c}{1})^{\star}$ of $\agreen{c}{1}$ (see \Pic \ref{fig:c(0,1) satuation I}).

\begin{figure}[htbp]
\begin{center}
\definecolor{ffqqqq}{rgb}{1,0,0}
\definecolor{bluearc}{rgb}{0,0,1}
\begin{tikzpicture}
\fill[black!20] (-2,2)  --  (-2, 0) to[out=-90, in=180] (0, -2)
  -- (0,-2.2) to[out=180,in=-90] (-2.2,0) -- (-2.2,2);
\draw (-2, 0) to[out=-90, in=180] (0, -2)  [line width=1pt];
\draw (-2,2)  --  (-2, 0) [line width=1pt][dotted];
\fill[black!20] (3.5,-0.5) to[out=180,in=-90] (3,0) -- (3,1) to[out=90,in=180] (3.5,1.5);
\fill[black!20] (3.5,1.5) arc(90:-90:1);
\draw[line width=1pt] (3.5,-0.5) to[out=180,in=-90] (3,0) -- (3,1) to[out=90,in=180] (3.5,1.5);
\draw[line width=1pt][dotted] (3.5,1.5) arc(90:-90:1);
\draw[bluearc] (-2, 0) to[out=0,in=90] (0, -2) to[out=45,in=135] (2,-2) [line width=1pt];
\draw[bluearc] (2,-2) to[out=45,in=135] (4,-2)[line width=1pt][dotted];
\draw[bluearc] (4,-2) to[out=45,in=135] (6,-2)[line width=1pt];
\fill[bluearc] (-2,0) circle (0.1);
\fill[bluearc] (-2,2) circle (0.1);
\fill[bluearc] (0,-2) circle (0.1);
\fill[bluearc] (2,-2) circle (0.1);
\fill[bluearc] (4,-2) circle (0.1);
\fill[bluearc] (6,-2) circle (0.1);
\draw[bluearc][line width=1pt] (-2,2) -- (3,1) -- (-2,0);
\draw[bluearc][line width=1pt] (-2,0) -- (3,0.5) [dotted];
\draw[bluearc][line width=1pt] (-2,0) -- (3,0);
\fill[bluearc] (3,1) circle (0.1);
\fill[bluearc] (3,0.5) circle (0.1);
\fill[bluearc] (3,0) circle (0.1);
\fill[bluearc] (3.5,-0.5) circle (0.1);
\draw[bluearc][line width=1pt] (3.5,-0.5) to[out=-10, in=90] (6,-2);
\draw[red][line width=0.75pt] (-1.41,-1.41) to[out=10,in=170] (3.15,-0.35);
\draw[red][line width=0.75pt] (3.15,-0.35) to[out=-90,in=-90] (5.5,0);
\draw[red][line width=0.75pt] (3.15,-0.35) to[out=-160,in=90] (1,-2);
\draw[red][line width=0.75pt] (3.15,-0.35) to[out=-140,in=120] (3,-2) [dotted];
\draw[red][line width=0.75pt] (3.15,-0.35) to[out=-120,in=140] (5,-2);
\draw[red][line width=0.75pt] (3.15,-0.35) to[out=175,in=180] (3,0.25);
\draw[red][line width=0.75pt] (3,0.25) to[out=180,in=180] (3,0.75) [dotted];
\draw[red][line width=0.75pt] (3,0.75) -- (-2,1);
\draw[red][line width=0.75pt] (-2,1) -- (0,2);
\fill[red] (-1.41,-1.41) circle (0.11); \fill[white] (-1.41,-1.41) circle (0.08);
\fill[red] (3.15,-0.35) circle (0.11); \fill[white] (3.15,-0.35) circle (0.08);
\fill[red] (-2,1) circle (0.11); \fill[white] (-2,1) circle (0.08);
\draw[orange][line width=1pt] (-1.41,-1.41) to[out=20,in=180] (4,2.5);
\draw[orange][line width=1pt] (4,2.5) to[out=0,in=90] (6,0.5) [dotted];
\draw[orange][line width=1pt] (6,0.5) to[out=-90,in=90] (4,-2);
\draw[orange] (6,0.5) node[left]{$c$};
\draw[violet][line width=1pt] (0,-2) to[out=90,in=180] (4,2.7);
\draw[violet][line width=1pt] (4,2.7) to[out=0,in=90] (6.2,0.5) [dotted];
\draw[violet][line width=1pt] (6.2,0.5) to[out=-90,in=0] (3.5,-1)
  to[out=180,in=-90] (2.5,-0.5) to[out=90,in=180] (3,0);
\draw[violet] (6.2,0.5) node[right]{$\tau=c^{\circlearrowleft}$};
\draw[bluearc] (1.414-2,1.414-2) node[left]{$\agreen{c}{1}$};
\draw (-1,-1.73) node[below]{$\bSurf$};
\draw (4.5,0.5) node[left]{$\bSurf$};
\end{tikzpicture}
\caption{The case that $\agreen{c}{1}\in\top(c)$ }
\label{fig:c(0,1) satuation I}
\end{center}
\end{figure}

For (II), we have $\agreen{c}{1}\in\soc(c)$.
Assume that the string corresponding to $c$ is of the form:
\[ \mathfrak{v}(\agreen{c}{1}) \longleftarrow \cdots \longleftarrow \mathfrak{v}(a_{\gbullet}) \longrightarrow \cdots . \]
Let $c_{\simp}(\agreen{c}{1})$ be the permissible curve crosses only $\agreen{c}{1}$.
Then $\MM(c_{\simp}(\agreen{c}{1}))$ is a simple module corresponding to $\mathfrak{v}(\agreen{c}{1})$
and it is a direct summand of the socle of $\MM(c_{\proj}(a_{\gbullet}))$.
Thus we also have $Q_L=0$ for $\ker p_0$. So $\tilde{\tau}$ is the admissible curve such that $\tau$ crosses $(a_{\gbullet})^{\star}$.
\end{proof}

\subsection{The admissible curves corresponding to band modules}

Let $B(n,\lambda)$ be the band module corresponding to the permissible curve $c\in\PC^{\oslash}_{\oslash}(\SURF(A)^{\F_A}) \ne \varnothing$ with Jordan block $\pmb{J}_{n}(\lambda\ne 0)$.
By \cite{BR1987}, we have a short exact sequence
\begin{align}\label{s.e.s. of band}
0 \longrightarrow B(1,\lambda) \longrightarrow B(n+1,\lambda)
  \longrightarrow B(n,\lambda) \longrightarrow 0
\end{align}
for any $n\in\NN^+$. As we all know $\pdim B(n,\lambda) = 1$ (see \cite[Corollary 2.26]{Cpre}). Thus, we have

\begin{lemma} \label{lemm:1-band}
Let $(c,\pmb{J}_{n}(\lambda))\in\PC^{\oslash}_{\oslash}(\SURF(A)^{\F_A})\times \mathscr{J}$ be the permissible curve with Jordan block $\pmb{J}_{n}(\lambda)$ $(\lambda\ne 0)$.
Then the admissible curve corresponding to the projective resolution of $\MM(c,\pmb{J}_{n}(\lambda))\cong B(n,\lambda)$ is a graded curve
$\tilde{\tau} \in \AC_{\oslash}(\SURF(A)^{\F_A})$ such that $\tau$ is homotopic to $c$.
\end{lemma}

%
%
%
We immediately obtain the following result by Lemma \ref{lemm:1-band}.

\begin{corollary} \label{thm:band embedding}
For any $n\in \NN^+$, there exists an isomorphism in $\per A$
\[\X(\tilde{\tau}, \pmb{J}_{n}(\lambda)) \cong \comp{P}_n\]
between the band complex $\X(\tilde{\tau}, \pmb{J}_{n}(\lambda))$
and the complex $\comp{P}_n$ induced by the projective resolution of $B(n,\lambda)$,
where $\pmb{J}_{n}(\lambda)$ is the $n\times n$ Jordan block with non-zero eigenvalue $\lambda$.
\end{corollary}

\section{The cohomologies of complexes} \label{sec:truncations}

Let $A$ be a gentle algebra. In this section, we give a description of cohomologies of complexes in $\per A$.
For any path $\wp=\alpha_1\alpha_2\cdots\alpha_t$ on the quiver $\Q$ of $A$, recall that the formal inverse path of $\wp$ is $\wp^{-1}=\alpha_t^{-1}\cdots\alpha_2^{-1}\alpha_1^{-1}$. We denote by $\Q^{-1} = \bigcup_{\ell\in\NN}\Q^{-1}_{\ell}$ the set of all formal inverse paths, where $\Q^{-1}_{\ell}$ is the set of all formal inverse paths of length $\ell$.
For any $\wp\in\Q$, we have $(\wp^{-1})^{-1} = \wp \in \Q$. In particular, if the length of $\wp$ is zero, then $\wp=\wp^{-1}\in \Q_0 = \Q^{-1}_0$.

Let $\tc$ be an admissible curve with $c=\{c_{[i,i+1]}\}_{0\le i\le n(\tc)}$ (see \Pic \ref{fig:tc}).
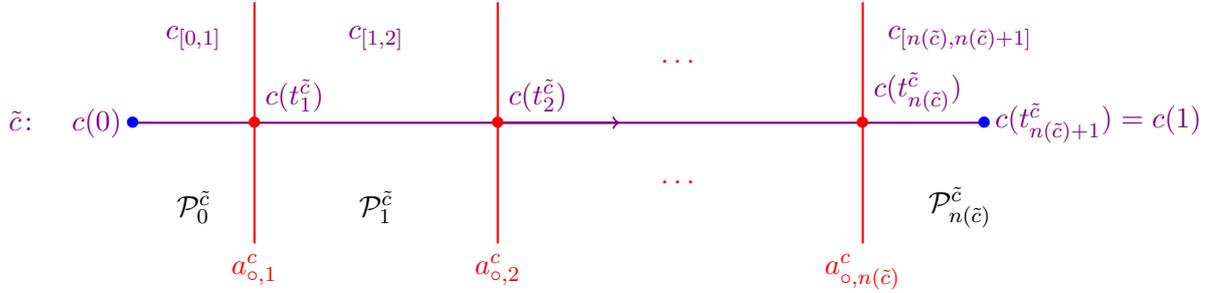
\begin{figure}[H]
\centering
\begin{tikzpicture}[scale=1.6]
\draw[violet,thick] (0,0)node[violet,left]{$\tc\colon\quad c(0)$} to
                    (7,0)node[violet,right]{$c(t^{\tc}_{n(\tc)+1})=c(1)$};
                    \draw[violet, thick, ->] (3.0,0) to (4.0,0);
\draw[violet]  (.5, .7)node{$c_{[0,1]}$} (2, .7)node{$c_{[1,2]}$} (6.8, .7)node{$c_{[n(\tc), n(\tc)+1]}$};
\draw[black]   (.5,-.7)node{$\PP_0^{\tc}$}     (2,-.7)node{$\PP_1^{\tc}$}     (6.8,-.7)node{$\PP_{n(\tc)}^{\tc}$};
\draw[red, thick](1,1)to(1,-1)node[below]{$a^{c}_{\rbullet, 1}$}
                 (1,0) node{$\bullet$} node[violet,above right]{$c(t^{\tc}_1)$};
\draw[red, thick](3,1)to(3,-1)node[below]{$a^{c}_{\rbullet, 2}$}
                 (3,0) node{$\bullet$} node[violet,above right]{$c(t^{\tc}_2)$};
\draw[red, thick](6,1)to(6,-1)node[below]{$a^{c}_{\rbullet, n(\tc)}$}
                 (6,0) node{$\bullet$} node[violet,above right]{$c(t^{\tc}_{n(\tc)})$}
                 (4.5,.5)node{$\cdots$}(4.5,-.5)node{$\cdots$};
\draw[blue] (0,0) node{$\gbullet$} (7,0) node{$\gbullet$};
\end{tikzpicture}
\caption{Arc segments of graded curve $\tc$}
\label{fig:tc}
\end{figure}
\noindent
It is clear that any $\Dred$-arc segment $c_{[i,i+1]}$ ($1\le i< n(\tc)$) of $c$ lies in the inner of an elementary $\rbullet$-polygon $\PP^{\tc}_i$,
which corresponding to a path $\wp(c_{[i,i+1]}): =\alpha_{i, 1}\alpha_{i, 2}\cdots\alpha_{i,t_i}$
(or a formal inverse path $\wp(c_{[i,i+1]}): =\alpha_{i,t_i}^{-1}\cdots\alpha_{i, 2}^{-1}\alpha_{i, 1}^{-1}$)
on the quiver of $A$.
Therefore, $c$ provides a path $\wp(c)=\wp_1\cdots\wp_{n(\tc)-1}$ on the underlying graph $\overline{\Q}$ of $\Q$, which is called {\defines the homotopy string of $\X(\tc)$}.
Let $\comp{P} = (P^n, d^n)_{n\in \ZZ} = $
\[ \cdots \mathop{\longrightarrow}\limits^{d^{-2}} P^{-1}
          \mathop{\longrightarrow}\limits^{d^{-1}} P^{0}
          \mathop{\longrightarrow}\limits^{d^{0}}  P^{1}
          \mathop{\longrightarrow}\limits^{d^{1}}  P^{2}
          \mathop{\longrightarrow}\limits^{d^{2}}  \cdots
\]
be an indecomposable object in $\per A$.
By \cite{BM2003} and \cite[Theorem 2.12]{OPS2018}, $\X(\tc) = (P^n, d^n)$ is the indecomposable projective complex in $\per A$ such that the following conditions hold:
\begin{itemize}
  \item The $n$-th component of $\X(\tc)$ is
    \[P^n = \bigoplus\limits_{\varsigma^{\tc}_i=-n}P(\bfv(\dualgreen{c}{i})),\]
    where $\bfv(\dualgreen{c}{i})$ is the starting point of $\wp_i$ and the ending point of $\wp_{i-1}$, $\dualgreen{c}{i}$ is the dual arc of $a^{c}_{\rbullet, i}$ and $\varsigma^{\tc}_i$ is the intersection index $\ii_{c(i)}(\tc, \ta_{\rbullet,i})$.

  \item $d^n$ is naturally induced by the homotopy string of projective complex.
\end{itemize}
Thus, for any $i$, $\MM(c_{\proj}(\dualgreen{c}{i}))$ is a direct summand of $(-\varsigma^{\tc}_i)$-th component $P^{-\varsigma^{\tc}_i}$ of $\X(\tc)$.
In this case, we say that the projective permissible curve $c_{\proj}(\dualgreen{c}{i})$ is a {\defines $(-\varsigma^{\tc}_i)$-th component} of $\tc$. Notice that we have the following four cases.
\begin{itemize}
  \item[(A)] (resp. (B)) The $\gbullet$-marked point in $\PP_{i-1}^{\tc}$ is to the right (resp. left) of $c_{[i-1,i]}$
    and the $\gbullet$-marked point in $\PP_{i}^{\tc}$ is to the left (resp. right) of $c_{[i,i+1]}$ for any
    $1\le i\le n(\tc)$.

  \item[(C)] (resp. (D)) The $\gbullet$-marked points in $\PP_{i-1}^{\tc}$ and $\PP_{i}^{\tc}$ are to the right (resp. left) of $c_{[i-1,i]}$ and $c_{[i,i+1]}$ for any $1\le i\le n(\tc)$, respectively.

\end{itemize}

Next, we introduce the truncations of projective permissible curves, which will be used to describe the cohomology of indecomposable object in $\per A$.

\begin{definition} \rm \label{def:truncation}
Let $p = \{p_{(i,i+1)}\}_{0\le i\le m(p)} \in \PC_{\m}(\SURF^{\F})$ be a permissible curve. For any $0\le t\le t+u\le m(p)+1$,  we call $\tru_{t}^{t+u}(p)$ a {\defines truncation} of $p$ given by $\agreen{p}{t}$ and $\agreen{p}{t+u}$ cutting $p$ (=truncation for short)
if it is a subsequence $\{p_{(i,i+1)}\}_{t\le i\le t+u-1}$ of $p$, where $\agreen{p}{0}=p(0)$ and $\agreen{p}{m(p)+1}=p(1)$.
Moreover, we say $\tru_{t_1}^{t_2}(p)$ is a {\defines trivial truncation of $p$} if $t_1>t_2$ and denote $\tru_{t_1}^{t_2}(p)=\varnothing$;
otherwise, we denote $\tru_{i}^{j}(p) \sqsubseteq p$.
\end{definition}

Let $c_1, \ldots, c_n$ be a family of permissible curves. Assume that $M =\bigoplus\limits_{i=1}^n \MM(c_i)$. We say that $q \sqsubseteq \MM^{-1}(M)=\{c_1, \ldots, c_n\}$ if $q \sqsubseteq c_i$ for some $\Dblue$-arc segments $q$ and some $1\le i\le n$.

\begin{proposition} \label{prop:HomologiesCaseI}
Let $\comp{P}$ be a string complex in $\per A$ and $\tc$ be an admissible curve in $\X^{-1}(\comp{P})$. Assume that $n(\tc) \ge 3$, that is, $\tc$ crosses at least three $\Dred$-arcs $\ared{c}{i-1}$, $\ared{c}{i}$ and $\ared{c}{i+1}$.
\begin{itemize}
\item[{\rm(1)}]
If the $\gbullet$-marked points $\m_1$ and $\m_2$ satisfy one of {Cases \rm(A)}, {\rm(C)} and {\rm(D)} {\rm(}see {\Pic}\!\!s \ref{fig:HomologiesCaseI1}--\ref{fig:HomologiesCaseI3}{\rm)},
then there exist two integers $1\le r, s \le m(c_{\proj}(\dualgreen{c}{i}))$ {\rm(}$=m$ for short{\rm)} such that $(c_{\proj}(\dualgreen{c}{i}))_{(0,1)}$, $\ldots$, $(c_{\proj}(\dualgreen{c}{i}))_{(r-1,r)}$ are $\Dblue$-arc segments of $c_{\proj}(\dualgreen{c}{i-1})$ and $(c_{\proj}(\dualgreen{c}{i}))_{(s,s+1)}$, $\ldots$, $(c_{\proj}(\dualgreen{c}{i}))_{(m, m+1)}$ are $\Dblue$-arc segments of $c_{\proj}(\dualgreen{c}{i+1})$, respectively.
%
In particular, if $s\ge r+2$, then
\[\tru_{r+1}^{s-1}(c_{\proj}(\dualgreen{c}{i})) \sqsubseteq \MM^{-1}(\rmH^{-\varsigma_i^{\tc}}(\comp{P})).\]

\item[{\rm(2)}] If $\m_1$ and $\m_2$ satisfy {\rm Case (B)} $($see \Pic \ref{fig:HomologiesCaseIII}$)$, then there exists an integer $1\le r\le m$ such that
\[\tru_{r}^{r}(c_{\proj}(\dualgreen{c}{i})) \sqsubseteq \MM^{-1}(\rmH^{-\varsigma^{\tc}_1}(\comp{P})). \]
That is, $\rmH^{-\varsigma^{\tc}_1}(\comp{P})e_{\mathfrak{v}(\agreen{p}{r})}\ne 0$. In this case, $\dualgreen{c}{i} = \agreen{p}{r}$.
\end{itemize}
\end{proposition}

\begin{figure}[H]
\centering
\definecolor{ffxfqq}{rgb}{1,0.5,0}
\definecolor{qqqqff}{rgb}{0,0,1}
\definecolor{ttqqqq}{rgb}{0.2,0,0}
\definecolor{ffqqqq}{rgb}{1,0,0}
\begin{tikzpicture}[scale=1.4]
\draw [line width=1pt,color=ffqqqq] (0,1)-- (-1,2);
\draw [line width=1pt,color=ffqqqq] (-1,2)-- (-3,2);
\draw [line width=1pt,dash pattern=on 2pt off 2pt,color=ffqqqq] (-3,2)-- (-4,1);
\draw [line width=1pt,color=ffqqqq] (-4,1)-- (-4,-1);
\draw [line width=1pt,dash pattern=on 2pt off 2pt,color=ffqqqq] (-4,-1)-- (-3,-2);
\draw [line width=1pt,color=ttqqqq] (-3,-2)-- (-1,-2);
\draw [line width=1pt,dash pattern=on 2pt off 2pt,color=ffqqqq] (-1,-2)-- (0,-1);
\draw [line width=1pt,color=ffqqqq] (0,-1)-- (0,1);
\draw [line width=1pt,dash pattern=on 2pt off 2pt,color=ffqqqq] (1,2)-- (0,1);
\draw [line width=1pt] (1,2)-- (3,2);
\draw [line width=1pt,dash pattern=on 2pt off 2pt,color=ffqqqq] (3,2)-- (4,1);
\draw [line width=1pt,color=ffqqqq] (4,1)-- (4,-1);
\draw [line width=1pt,dash pattern=on 2pt off 2pt,color=ffqqqq] (4,-1)-- (3,-2);
\draw [line width=1pt,color=ffqqqq] (3,-2)-- (1,-2);
\draw [line width=1pt,color=ffqqqq] (1,-2)-- (0,-1);
\draw [line width=1pt,dash pattern=on 2pt off 2pt,color=qqqqff] (-2.03,-2)-- (-3.49,-1.51);
\draw [line width=1pt,color=qqqqff] (-2.03,-2)-- (-4,-0.08);
\draw [line width=1pt,dash pattern=on 2pt off 2pt,color=qqqqff] (-2.03,-2)-- (-3.52,1.48);
\draw [line width=1pt,color=qqqqff] (-2.03,-2)-- (-2.02,2);
\draw [line width=1pt,color=qqqqff] (-2.03,-2)-- (-0.53,1.53);
\draw [line width=1pt,color=qqqqff] (-2.03,-2)-- (0,-0.06);
\draw [line width=1pt,dash pattern=on 2pt off 2pt,color=qqqqff] (-2.03,-2)-- (-0.55,-1.55);
\draw [line width=1pt,dash pattern=on 2pt off 2pt,color=qqqqff] (1.99,2)-- (0.5,1.5);
\draw [line width=1pt,color=qqqqff] (1.99,2)-- (0,-0.06);
\draw [line width=1pt,color=qqqqff] (1.99,2)-- (0.52,-1.51);
\draw [line width=1pt,color=qqqqff] (1.99,2)-- (1.97,-2);
\draw [line width=1pt,dash pattern=on 2pt off 2pt,color=qqqqff] (1.99,2)-- (3.54,-1.46);
\draw [line width=1pt,color=qqqqff] (1.99,2)-- (4,-0.03);
\draw [line width=1pt,dash pattern=on 2pt off 2pt,color=qqqqff] (1.99,2)-- (3.43,1.57);
\draw [line width=1pt][->] (-4   ,-0.56) -- (-3.62,-1.38);
\draw [line width=1pt][->] (-3.70, 1.3 ) -- (-4.00, 0.5 ); \draw (-3.70, 1.3 ) node[ left]{$\alpha_u$};
\draw [line width=1pt][->] (-2.48, 2   ) -- (-3.27, 1.73);
\draw [line width=1pt][->] (-0.7 , 1.7 ) -- (-1.44, 2   ); \draw (-1.44, 2   ) node[above]{$\alpha_2$};
\draw [line width=1pt][->] ( 0.  , 0.3 ) -- (-0.33, 1.33); \draw (-0.33, 1.33) node[right]{$\alpha_1$};
\draw [line width=1pt][->] ( 0.  ,-0.34) -- ( 0.24,-1.24); \draw ( 0.24,-1.24) node[below]{$\beta_1$};
\draw [line width=1pt][->] ( 0.68,-1.68) -- ( 1.48,-2   ); \draw ( 1.48,-2   ) node[below]{$\beta_2$};
\draw [line width=1pt][->] ( 2.4 ,-2   ) -- ( 3.32,-1.68);
\draw [line width=1pt][->] ( 3.77,-1.23) -- ( 4   ,-0.44); \draw ( 3.77,-1.23) node[right]{$\beta_v$};
\draw [line width=1pt][->] ( 4   , 0.43) -- ( 3.67, 1.33);
\draw [line width=1pt][->,line width=1.5pt,violet] (-5,0) -- (5,0) node[above]{$\tc$};
\begin{scriptsize}
\fill [red] ( 0.00,-1.00) circle (0.1cm); \fill [white] ( 0.00,-1.00) circle (0.07cm);
\fill [red] ( 0.00, 1.00) circle (0.1cm); \fill [white] ( 0.00, 1.00) circle (0.07cm);
\fill [red] (-1.00, 2.00) circle (0.1cm); \fill [white] (-1.00, 2.00) circle (0.07cm);
\fill [red] (-3.00, 2.00) circle (0.1cm); \fill [white] (-3.00, 2.00) circle (0.07cm);
\fill [red] (-4.00, 1.00) circle (0.1cm); \fill [white] (-4.00, 1.00) circle (0.07cm);
\fill [red] (-4.00,-1.00) circle (0.1cm); \fill [white] (-4.00,-1.00) circle (0.07cm);
\fill [red] (-3.00,-2.00) circle (0.1cm); \fill [white] (-3.00,-2.00) circle (0.07cm);
\fill [red] (-1.00,-2.00) circle (0.1cm); \fill [white] (-1.00,-2.00) circle (0.07cm);
\fill [red] ( 1.00, 2.00) circle (0.1cm); \fill [white] ( 1.00, 2.00) circle (0.07cm);
\fill [red] ( 3.00, 2.00) circle (0.1cm); \fill [white] ( 3.00, 2.00) circle (0.07cm);
\fill [red] ( 4.00, 1.00) circle (0.1cm); \fill [white] ( 4.00, 1.00) circle (0.07cm);
\fill [red] ( 4.00,-1.00) circle (0.1cm); \fill [white] ( 4.00,-1.00) circle (0.07cm);
\fill [red] ( 3.00,-2.00) circle (0.1cm); \fill [white] ( 3.00,-2.00) circle (0.07cm);
\fill [red] ( 1.00,-2.00) circle (0.1cm); \fill [white] ( 1.00,-2.00) circle (0.07cm);
\draw [red] ( 0.  , 0.75) node[right]{$\ta^c_{\rbullet, i}$};
\draw [red] (-4.  , 0.75) node[ left]{$\ta^c_{\rbullet, i-1}$};
\draw [red] ( 4.  , 0.75) node[right]{$\ta^c_{\rbullet, i+1}$};
\fill [blue] (-2.00,-2.00) circle (0.1cm);
\fill [blue] ( 2.00, 2.00) circle (0.1cm);
\draw [orange][line width=1.5pt][opacity=0.75]
      (-3.4,-2. ) to[out= 120, in= -90] (-3.8, 0. ) to[out=  90, in=-180]
      (-2. , 1.8) to[out=   0, in= 100] ( 0. , 0. ) to[out= -80, in= 180]
      ( 2. ,-1.8) to[out=   0, in= -90] ( 3.8, 0. ) to[out=  90, in= -60]
      ( 3.4, 2. );
\draw [orange] (-1.8, 2.20) node[above]{$p=c_{\proj}(\dualgreen{c}{i})$};
\draw [orange][->] (-1.8, 1.80) to (-1.8, 2.20);
\draw [green][line width=1.5pt]
      (-3.2,-2. ) to[out=  90, in=   0] (-4.5, 0.25) node[left]{$c_{\proj}(\dualgreen{c}{i-1})$};
\draw [green][line width=1.5pt]
      ( 3.2, 2. ) to[out= -90, in= 180] ( 4.5,-0.25) node[right]{$c_{\proj}(\dualgreen{c}{i+1})$};
\draw [cyan][line width=1.5pt]
      (-3.3, 0.9) to[out=  45, in= 180] (-2. , 1.6) to[out=   0, in= 110]
      ( 0. , 0. ) to[out= -70, in= 180] ( 2. ,-1.6) to[out=   0, in=-135]
      ( 3.3,-0.9);
\draw [cyan][->] (-3,1.25) -- (-3,2.4) node[above]{$\tru_{r+1}^{s-1}(c_{\proj}(\dualgreen{c}{i}))$};
\end{scriptsize}
\draw [blue] (-3. ,-1. ) node[above]{$\agreen{p}{r}$};
\draw [blue] ( 3. , 1. ) node[below]{$\agreen{p}{s}$};
\draw [blue] (-2,-2) node[below]{$\m_1$};
\draw [blue] ( 2, 2) node[above]{$\m_2$};
\end{tikzpicture}
\caption{Case (A)}
\label{fig:HomologiesCaseI1}
\end{figure}

\begin{figure}[H]
\centering
\definecolor{ffxfqq}{rgb}{1,0.5,0}
\definecolor{qqqqff}{rgb}{0,0,1}
\definecolor{ttqqqq}{rgb}{0.2,0,0}
\definecolor{ffqqqq}{rgb}{1,0,0}
\begin{tikzpicture}[scale=1.4]
\draw [line width=1pt,dash pattern=on 2pt off 2pt,color=ffqqqq] (0,1)-- (-1,2);
\draw [line width=1pt,color=ffqqqq] (-1,2)-- (-3,2);
\draw [line width=1pt,dash pattern=on 2pt off 2pt,color=ffqqqq] (-3,2)-- (-4,1);
\draw [line width=1pt,color=ffqqqq] (-4,1)-- (-4,-1);
\draw [line width=1pt,dash pattern=on 2pt off 2pt,color=ffqqqq] (-4,-1)-- (-3,-2);
\draw [line width=1pt,color=ttqqqq] (-3,-2)-- (-1,-2);
\draw [line width=1pt,dash pattern=on 2pt off 2pt,color=ffqqqq] (-1,-2)-- (0,-1);
\draw [line width=1pt,color=ffqqqq] (0,-1)-- (0,1);
\draw [line width=1pt,dash pattern=on 2pt off 2pt,color=ffqqqq] (1,2)-- (0,1);
\draw [line width=1pt,color=ffqqqq] (1,2)-- (3,2);
\draw [line width=1pt,dash pattern=on 2pt off 2pt,color=ffqqqq] (3,2)-- (4,1);
\draw [line width=1pt,color=ffqqqq] (4,1)-- (4,-1);
\draw [line width=1pt,dash pattern=on 2pt off 2pt,color=ffqqqq] (4,-1)-- (3,-2);
\draw [line width=1pt,color=ttqqqq] (3,-2)-- (1,-2);
\draw [line width=1pt,dash pattern=on 2pt off 2pt,color=ffqqqq] (1,-2)-- (0,-1);
\draw [line width=1pt,dash pattern=on 2pt off 2pt,color=qqqqff] (-2.03,-2)-- (-3.49,-1.51);
\draw [line width=1pt,color=qqqqff] (-2.03,-2)-- (-4,-0.08);
\draw [line width=1pt,dash pattern=on 2pt off 2pt,color=qqqqff] (-2.03,-2)-- (-3.52,1.48);
\draw [line width=1pt,color=qqqqff] (-2.03,-2)-- (-2.02,2);
\draw [line width=1pt,dash pattern=on 2pt off 2pt,color=qqqqff] (-2.03,-2)-- (-0.53,1.53);
\draw [line width=1pt,dash pattern=on 2pt off 2pt,color=qqqqff] (-2.03,-2)-- (-0.55,-1.55);
\draw [line width=1pt][->] (-3.7 , 1.3 ) -- (-4.  , 0.5 );
\draw [line width=1pt][->] (-2.48, 2.  ) -- (-3.27, 1.73);
\draw [line width=1pt][->] (-0.7 , 1.7 ) -- (-1.44, 2.  );
\draw [line width=1pt][->] ( 0   , 0.3 ) -- (-0.33, 1.33);
\draw [line width=1pt][->] (-4   ,-0.56) -- (-3.62,-1.38);
\draw [line width=1pt][->] ( 0.  ,-0.43) -- ( 0.33,-1.33);
\draw [line width=1pt][->,line width=1.5pt,violet] (-5,0) -- (5,0) node[above]{$\tc$};
\draw [line width=1pt,dash pattern=on 2pt off 2pt,color=qqqqff] (1.96,-2)-- (0.53,-1.53);
\draw [line width=1pt,dash pattern=on 2pt off 2pt,color=qqqqff] (1.96,-2)-- (0.53,1.53);
\draw [line width=1pt,color=qqqqff] (1.96,-2)-- (1.96,2);
\draw [line width=1pt,dash pattern=on 2pt off 2pt,color=qqqqff] (1.96,-2)-- (3.49,1.51);
\draw [line width=1pt,color=qqqqff] (1.96,-2)-- (4,0);
\draw [line width=1pt,dash pattern=on 2pt off 2pt,color=qqqqff] (1.96,-2)-- (3.5,-1.5);
\draw [shift={(-0.04,-2.93)},line width=1pt,color=qqqqff]  plot[domain=0.44:2.71,variable=\t]({1*2.2*cos(\t r)+0*2.2*sin(\t r)},{0*2.2*cos(\t r)+1*2.2*sin(\t r)});
\draw [orange][opacity=0.75][line width=1.5pt]
      (0.5,-2. ) to[out=90, in=0] (-1.6,0.8) to[out=180,in=90] (-3.4,-2. );
\begin{scriptsize}
\draw [orange] (-1.5,0.75) node[below]{$c_{\proj}(\dualgreen{c}{i})$};
\end{scriptsize}
\begin{scriptsize}
\fill [red] ( 0.00,-1.00) circle (0.1cm); \fill [white] ( 0.00,-1.00) circle (0.07cm);
\fill [red] ( 0.00, 1.00) circle (0.1cm); \fill [white] ( 0.00, 1.00) circle (0.07cm);
\fill [red] (-1.00, 2.00) circle (0.1cm); \fill [white] (-1.00, 2.00) circle (0.07cm);
\fill [red] (-3.00, 2.00) circle (0.1cm); \fill [white] (-3.00, 2.00) circle (0.07cm);
\fill [red] (-4.00, 1.00) circle (0.1cm); \fill [white] (-4.00, 1.00) circle (0.07cm);
\fill [red] (-4.00,-1.00) circle (0.1cm); \fill [white] (-4.00,-1.00) circle (0.07cm);
\fill [red] (-3.00,-2.00) circle (0.1cm); \fill [white] (-3.00,-2.00) circle (0.07cm);
\fill [red] (-1.00,-2.00) circle (0.1cm); \fill [white] (-1.00,-2.00) circle (0.07cm);
\fill [red] ( 1.00, 2.00) circle (0.1cm); \fill [white] ( 1.00, 2.00) circle (0.07cm);
\fill [red] ( 3.00, 2.00) circle (0.1cm); \fill [white] ( 3.00, 2.00) circle (0.07cm);
\fill [red] ( 4.00, 1.00) circle (0.1cm); \fill [white] ( 4.00, 1.00) circle (0.07cm);
\fill [red] ( 4.00,-1.00) circle (0.1cm); \fill [white] ( 4.00,-1.00) circle (0.07cm);
\fill [red] ( 3.00,-2.00) circle (0.1cm); \fill [white] ( 3.00,-2.00) circle (0.07cm);
\fill [red] ( 1.00,-2.00) circle (0.1cm); \fill [white] ( 1.00,-2.00) circle (0.07cm);
\draw [red] ( 0.  , 0.75) node[right]{$\ta^c_{\rbullet, i}$};
\draw [red] (-4.  , 0.75) node[ left]{$\ta^c_{\rbullet, i-1}$};
\draw [red] ( 4.  , 0.75) node[right]{$\ta^c_{\rbullet, i+1}$};
\fill [blue] (-2.00,-2.00) circle (0.1cm);
\fill [blue] ( 2.00,-2.00) circle (0.1cm);
\draw [blue] (-3. ,-1. ) node[above]{$\agreen{p}{r}$};
\draw [blue] (-1. ,-1. ) node[above]{$\agreen{p}{s}$};
\end{scriptsize}
\draw [green][line width=1.5pt]
      (-3.2,-2. ) to[out=  90, in=   0] (-4.5, 0.25) node[left]{$c_{\proj}(\dualgreen{c}{i-1})$};
\draw [green][line width=1.5pt]
      (4.5,-0.25) node[right]{$c_{\proj}(\dualgreen{c}{i+1})$} to[out= 180, in=  90] ( 0.8, -2. );
\draw [cyan][line width=1.5pt]
      (-3.3, 0.9) to[out=  45, in= 135] (-0.8, 0.9);
\draw [cyan][->] (-3,1.2) -- (-3,2.2) node[above]{\tiny$\tru_{r+1}^{s-1}(c_{\proj}(\dualgreen{c}{i}))$};
\end{tikzpicture}
\caption{Case (C)}
\label{fig:HomologiesCaseI2}
\end{figure}

\begin{figure}[H]
\centering
\definecolor{ffxfqq}{rgb}{1,0.5,0}
\definecolor{qqqqff}{rgb}{0,0,1}
\definecolor{ttqqqq}{rgb}{0.2,0,0}
\definecolor{ffqqqq}{rgb}{1,0,0}
\begin{tikzpicture}[scale=1.4]
\draw [line width=1pt,dash pattern=on 2pt off 2pt,color=ffqqqq] (0,-1)-- (-1,-2);
\draw [line width=1pt,color=ffqqqq] (-1,-2)-- (-3,-2);
\draw [line width=1pt,dash pattern=on 2pt off 2pt,color=ffqqqq] (-3,-2)-- (-4,-1);
\draw [line width=1pt,color=ffqqqq] (-4,-1)-- (-4,1);
\draw [line width=1pt,dash pattern=on 2pt off 2pt,color=ffqqqq] (-4,1)-- (-3,2);
\draw [line width=1pt,color=ttqqqq] (-3,2)-- (-1,2);
\draw [line width=1pt,dash pattern=on 2pt off 2pt,color=ffqqqq] (-1,2)-- (0,1);
\draw [line width=1pt,color=ffqqqq] (0,1)-- (0,-1);
\draw [line width=1pt,dash pattern=on 2pt off 2pt,color=ffqqqq] (1,-2)-- (0,-1);
\draw [line width=1pt,color=ffqqqq] (1,-2)-- (3,-2);
\draw [line width=1pt,dash pattern=on 2pt off 2pt,color=ffqqqq] (3,2)-- (4,1);
\draw [line width=1pt,color=ffqqqq] (4,-1)-- (4,1);
\draw [line width=1pt,dash pattern=on 2pt off 2pt,color=ffqqqq] (4,-1)-- (3,-2);
\draw [line width=1pt,color=ttqqqq] (3,2)-- (1,2);
\draw [line width=1pt,dash pattern=on 2pt off 2pt,color=ffqqqq] (1,2)-- (0,1);
\draw [line width=1pt,dash pattern=on 2pt off 2pt,color=qqqqff] (-2.03,2)-- (-3.49,1.51);
\draw [line width=1pt,color=qqqqff] (-2.03,2)-- (-4,0.08);
\draw [line width=1pt,dash pattern=on 2pt off 2pt,color=qqqqff] (-2.03,2)-- (-3.52,-1.48);
\draw [line width=1pt,color=qqqqff] (-2.03,2)-- (-2.02,-2);
\draw [line width=1pt,dash pattern=on 2pt off 2pt,color=qqqqff] (-2.03,2)-- (-0.53,-1.53);
\draw [line width=1pt,dash pattern=on 2pt off 2pt,color=qqqqff] (-2.03,2)-- (-0.55,1.55);
\draw [line width=1pt][<-] (-3.7 ,-1.3 ) -- (-4.  ,-0.5 );
\draw [line width=1pt][<-] (-2.48,-2.  ) -- (-3.27,-1.73);
\draw [line width=1pt][<-] (-0.7 ,-1.7 ) -- (-1.44,-2.  );
\draw [line width=1pt][<-] ( 0   ,-0.3 ) -- (-0.33,-1.33);
\draw [line width=1pt][<-] (-4   , 0.56) -- (-3.62, 1.38);
\draw [line width=1pt][<-] ( 0.  , 0.43) -- ( 0.33, 1.33);
\draw [line width=1pt][->,line width=1.5pt,violet] (-5,0) -- (5,0) node[above]{$\tc$};
\draw [line width=1pt,dash pattern=on 2pt off 2pt,color=qqqqff] (1.96, 2)-- (0.53, 1.53);
\draw [line width=1pt,dash pattern=on 2pt off 2pt,color=qqqqff] (1.96, 2)-- (0.53,-1.53);
\draw [line width=1pt,color=qqqqff] (1.96, 2)-- (1.96,-2);
\draw [line width=1pt,dash pattern=on 2pt off 2pt,color=qqqqff] (1.96, 2)-- (3.49,-1.51);
\draw [line width=1pt,color=qqqqff] (1.96, 2)-- (4,0);
\draw [line width=1pt,dash pattern=on 2pt off 2pt,color=qqqqff] (1.96, 2)-- (3.5, 1.5);
\draw [line width=1pt,color=qqqqff] (-1.96, 2) to[out=-45,in=180] (0,0.8) to[out=0,in=-135] (1.96, 2);
\begin{scriptsize}%
\fill [red] ( 0.00, 1.00) circle (0.1cm); \fill [white] ( 0.00, 1.00) circle (0.07cm);
\fill [red] ( 0.00,-1.00) circle (0.1cm); \fill [white] ( 0.00,-1.00) circle (0.07cm);
\fill [red] (-1.00,-2.00) circle (0.1cm); \fill [white] (-1.00,-2.00) circle (0.07cm);
\fill [red] (-3.00,-2.00) circle (0.1cm); \fill [white] (-3.00,-2.00) circle (0.07cm);
\fill [red] (-4.00,-1.00) circle (0.1cm); \fill [white] (-4.00,-1.00) circle (0.07cm);
\fill [red] (-4.00, 1.00) circle (0.1cm); \fill [white] (-4.00, 1.00) circle (0.07cm);
\fill [red] (-3.00, 2.00) circle (0.1cm); \fill [white] (-3.00, 2.00) circle (0.07cm);
\fill [red] (-1.00, 2.00) circle (0.1cm); \fill [white] (-1.00, 2.00) circle (0.07cm);
\fill [red] ( 1.00,-2.00) circle (0.1cm); \fill [white] ( 1.00,-2.00) circle (0.07cm);
\fill [red] ( 3.00,-2.00) circle (0.1cm); \fill [white] ( 3.00,-2.00) circle (0.07cm);
\fill [red] ( 4.00,-1.00) circle (0.1cm); \fill [white] ( 4.00,-1.00) circle (0.07cm);
\fill [red] ( 4.00, 1.00) circle (0.1cm); \fill [white] ( 4.00, 1.00) circle (0.07cm);
\fill [red] ( 3.00, 2.00) circle (0.1cm); \fill [white] ( 3.00, 2.00) circle (0.07cm);
\fill [red] ( 1.00, 2.00) circle (0.1cm); \fill [white] ( 1.00, 2.00) circle (0.07cm);
\draw [red] ( 0.  ,-0.75) node[right]{$\ta^c_{\rbullet, i}$};
\draw [red] (-4.  ,-0.75) node[ left]{$\ta^c_{\rbullet, i-1}$};
\draw [red] ( 4.  ,-0.75) node[right]{$\ta^c_{\rbullet, i+1}$};
\fill [blue] (-2.00, 2.00) circle (0.1cm);
\fill [blue] ( 2.00, 2.00) circle (0.1cm);
\draw [blue] ( 0.8, 1. ) node[above]{$\agreen{p}{r}$};
\draw [blue] ( 3.1, 1.1) node[below left]{$\agreen{p}{s}$};
\end{scriptsize}
\draw [orange][opacity=0.75][line width=1.5pt]
      (3.4,2. ) to[out=-90,in=0] (1.6,-0.8) to[out=180,in=-90] (-0.5,2. )
      node[above right]{$c_{\proj}(\dualgreen{c}{i})$};
\draw [green][line width=1.5pt]
      (3.2,2.) to[out=-90, in=180] (4.5,-0.25) node[right]{$c_{\proj}(\dualgreen{c}{i+1})$};
\draw [green][line width=1.5pt]
      (-4.5,0.25) node[left]{$c_{\proj}(\dualgreen{c}{i+1})$} to[out=0, in=-90] (-0.8,2.);
\draw [cyan][line width=1.5pt]
      (0.8,-0.9) to[out=-45, in=225] (3.3,-0.9);
\draw [cyan][->] (3,-1.2) -- (3,-2.2) node[below]{\tiny$\tru_{r+1}^{s-1}(c_{\proj}(\dualgreen{c}{i}))$};
\end{tikzpicture}
\caption{Case (D)}
\label{fig:HomologiesCaseI3}
\end{figure}

\begin{figure}[H]
\centering
\definecolor{ffxfqq}{rgb}{1,0.5,0}
\definecolor{qqqqff}{rgb}{0,0,1}
\definecolor{ttqqqq}{rgb}{0.2,0,0}
\definecolor{ffqqqq}{rgb}{1,0,0}
\begin{tikzpicture}[scale=1.4]
\draw [line width=1pt,color=ffqqqq] (0,-1)-- (-1,-2);
\draw [line width=1pt,color=ffqqqq] (-1,-2)-- (-3,-2);
\draw [line width=1pt,dash pattern=on 2pt off 2pt,color=ffqqqq] (-3,-2)-- (-4,-1);
\draw [line width=1pt,color=ffqqqq] (-4,-1)-- (-4, 1);
\draw [line width=1pt,dash pattern=on 2pt off 2pt,color=ffqqqq] (-4, 1)-- (-3, 2);
\draw [line width=1pt,color=ttqqqq] (-3, 2)-- (-1, 2);
\draw [line width=1pt,dash pattern=on 2pt off 2pt,color=ffqqqq] (-1, 2)-- ( 0, 1);
\draw [line width=1pt,color=ffqqqq] ( 0, 1)-- ( 0,-1);
\draw [line width=1pt,dash pattern=on 2pt off 2pt,color=ffqqqq] ( 1,-2)-- ( 0,-1);
\draw [line width=1pt] (1,-2)-- (3,-2);
\draw [line width=1pt,dash pattern=on 2pt off 2pt,color=ffqqqq] ( 3,-2)-- ( 4,-1);
\draw [line width=1pt,color=ffqqqq] ( 4,-1)-- (4, 1);
\draw [line width=1pt,dash pattern=on 2pt off 2pt,color=ffqqqq] ( 4, 1)-- (3, 2);
\draw [line width=1pt,color=ffqqqq] ( 3, 2)-- (1, 2);
\draw [line width=1pt,color=ffqqqq] ( 1, 2)-- (0, 1);
\draw [line width=1pt,dash pattern=on 2pt off 2pt,color=qqqqff] (-2.03, 2)-- (-3.49, 1.51);
\draw [line width=1pt,color=qqqqff] (-2.03, 2)-- (-4, 0.08);
\draw [line width=1pt,dash pattern=on 2pt off 2pt,color=qqqqff] (-2.03, 2)-- (-3.52,-1.48);
\draw [line width=1pt,color=qqqqff] (-2.03, 2)-- (-2.02,-2);
\draw [line width=1pt,color=qqqqff] (-2.03, 2)-- (-0.53,-1.53);
\draw [line width=1pt,color=qqqqff] (-2.03, 2)-- (0, 0.06);
\draw [line width=1pt,dash pattern=on 2pt off 2pt,color=qqqqff] (-2.03, 2)-- (-0.55, 1.55);
\draw [line width=1pt,dash pattern=on 2pt off 2pt,color=qqqqff] (1.99,-2)-- (0.5,-1.5);
\draw [line width=1pt,color=qqqqff] (1.99,-2)-- (0, 0.06);
\draw [line width=1pt,color=qqqqff] (1.99,-2)-- (0.52, 1.51);
\draw [line width=1pt,color=qqqqff] (1.99,-2)-- (1.97, 2);
\draw [line width=1pt,dash pattern=on 2pt off 2pt,color=qqqqff] (1.99,-2)-- (3.54, 1.46);
\draw [line width=1pt,color=qqqqff] (1.99,-2)-- (4, 0.03);
\draw [line width=1pt,dash pattern=on 2pt off 2pt,color=qqqqff] (1.99,-2)-- (3.43,-1.57);
\begin{scriptsize}
\fill [red] ( 0.00, 1.00) circle (0.1cm); \fill [white] ( 0.00, 1.00) circle (0.07cm);
\fill [red] ( 0.00,-1.00) circle (0.1cm); \fill [white] ( 0.00,-1.00) circle (0.07cm);
\fill [red] (-1.00,-2.00) circle (0.1cm); \fill [white] (-1.00,-2.00) circle (0.07cm);
\fill [red] (-3.00,-2.00) circle (0.1cm); \fill [white] (-3.00,-2.00) circle (0.07cm);
\fill [red] (-4.00,-1.00) circle (0.1cm); \fill [white] (-4.00,-1.00) circle (0.07cm);
\fill [red] (-4.00, 1.00) circle (0.1cm); \fill [white] (-4.00, 1.00) circle (0.07cm);
\fill [red] (-3.00, 2.00) circle (0.1cm); \fill [white] (-3.00, 2.00) circle (0.07cm);
\fill [red] (-1.00, 2.00) circle (0.1cm); \fill [white] (-1.00, 2.00) circle (0.07cm);
\fill [red] ( 1.00,-2.00) circle (0.1cm); \fill [white] ( 1.00,-2.00) circle (0.07cm);
\fill [red] ( 3.00,-2.00) circle (0.1cm); \fill [white] ( 3.00,-2.00) circle (0.07cm);
\fill [red] ( 4.00,-1.00) circle (0.1cm); \fill [white] ( 4.00,-1.00) circle (0.07cm);
\fill [red] ( 4.00, 1.00) circle (0.1cm); \fill [white] ( 4.00, 1.00) circle (0.07cm);
\fill [red] ( 3.00, 2.00) circle (0.1cm); \fill [white] ( 3.00, 2.00) circle (0.07cm);
\fill [red] ( 1.00, 2.00) circle (0.1cm); \fill [white] ( 1.00, 2.00) circle (0.07cm);
\draw [red] ( 0.  ,-0.75) node[left]{$\ta^c_{\rbullet, i}$};
\draw [red] (-4.  ,-0.75) node[ left]{$\ta^c_{\rbullet, i-1}$};
\draw [red] ( 4.  ,-0.75) node[right]{$\ta^c_{\rbullet, i+1}$};
\fill [blue] (-2.00, 2.00) circle (0.1cm);
\fill [blue] ( 2.00,-2.00) circle (0.1cm);
\draw [line width=1pt][->,line width=1.5pt,violet] (-5,-0.25) to[out=0,in=180] (5,0.25) node[above]{$\tc$};
\draw [orange][line width=1.5pt]
      (-0.5, 2.0) to[out=-90,in=130] (0,0) to[out=-50,in= 90] ( 0.5,-2.0)
      node[below]{$p=c_{\proj}(\dualgreen{c}{i})$};
\draw [green][line width=1.5pt] (-4.5, 0.0) node[ left]{$c_{\proj}(\dualgreen{c}{i-1})$}
      to[out=0, in=-95] (-0.6, 2.0);
\draw [green][line width=1.5pt] ( 4.5, 0.0) node[right]{$c_{\proj}(\dualgreen{c}{i+1})$}
      to[out=180,in=85] ( 0.6,-2.0);
\draw [cyan][line width=1.5pt] (-0.2,-0.2) to ( 0.2,0.2)
      node[right]{$\tru_{r}^{r}(c_{\proj}(\dualgreen{c}{i}))$};
\end{scriptsize}
\draw [blue] (0.5,-0.5) node[right]{$\agreen{p}{r}$};
\draw [blue] (-2, 2) node[above]{$\m_1$};
\draw [blue] ( 2,-2) node[below]{$\m_2$};
\end{tikzpicture}
\caption{Case (B)}
\label{fig:HomologiesCaseIII}
\end{figure}

\begin{proof}
We only prove Case (A), the proofs of the other Cases are similar.
By \cite[Lemma 2.16]{OPS2018}, $\tc$ corresponds to the following homotopy string $\omega$
\[ \cdots  \longline  \bfv(\dualgreen{c}{i-1})
\mathop{\longleftarrow}\limits^{\wp_1^{-1}} \bfv(\dualgreen{c}{i})
\mathop{\longleftarrow}\limits^{\wp_2} \bfv(\dualgreen{c}{i+1}) \longline \cdots, \]
where $\wp_1$ and $\wp_2$ are induced by the paths $\alpha_1\cdots\alpha_u$ and $\beta_1\cdots\beta_v$, respectively.
We only need to show that $\rmH^{-\varsigma_{i}^{\tc}}(\comp{P}) e_t \ne 0$ for any $e_t=\bfv(\dualgreen{c}{t})$ ($r+1\le t\le s-1$),
$\rmH^{-\varsigma_{i}^{\tc}}(\comp{P}) \alpha_l \ne 0$ ($1\le l<u$) and
$\rmH^{-\varsigma_{i}^{\tc}}(\comp{P}) \beta_k \ne 0$ ($1\le k<v$).

The indecomposable projective complex corresponded by $\omega$ is of the form
\begin{align}
 \cdots & \longrightarrow
  Q_1\oplus P(\bfv(\dualgreen{c}{i-1})) \oplus P(\bfv(\dualgreen{c}{i+1})) \oplus Q_2 \nonumber \\
& \mathop{\longline\!\!\!\longline\!\!\!\longline\!\!\!\longline\!\!\!\longrightarrow}\limits^{
  d^{-\varsigma_{i-1}^{\tc}} = \left(
  \begin{smallmatrix}
  f_1 & f_2 & 0 & 0 \\
   0  & \wp_1^{-1} & \wp_2 & 0\\
   0  & 0   & f_3 & f_4
  \end{smallmatrix}
  \right)}
  P_1 \oplus P(\bfv(\dualgreen{c}{i})) \oplus P_2
  \mathop{\longline\!\!\!\longline\!\!\!\longrightarrow}\limits^{
  d^{-\varsigma_{i}^{\tc}}
  = \left(\begin{smallmatrix}
  g_1& 0 &  0  \\
  0  & 0 & g_2
  \end{smallmatrix}\right)}
  R_1\oplus R_2 \longrightarrow \cdots \nonumber
\end{align}
where $P_1$, $P_2$, $Q_1$, $Q_2$, $R_1$ and $R_2$ are projective modules.
Thus,
\begin{align}
  \rmH^{-\varsigma_{i}^{\tc}}(\comp{P})
& = \ker d^{-\varsigma_{i}^{\tc}}/\im d^{-\varsigma_{i-1}^{\tc}} \nonumber \\
& \cong
\left(\begin{smallmatrix}
\ker g_1 \\
P(\bfv(\dualgreen{c}{i}) \\
\ker g_2
\end{smallmatrix}\right)
{\Big/}
\left(\begin{smallmatrix}
f_1(Q_1) + f_2(P(\bfv(\dualgreen{c}{i-1}))) \\
P(\bfv(\dualgreen{c}{i-1})) + P(\bfv(\dualgreen{c}{i+1})) \\
f_3(P(\bfv(\dualgreen{c}{i+1}))) + f_4(Q_2)
\end{smallmatrix}\right).
\nonumber
\end{align}
For any $e_t=\bfv(\dualgreen{c}{t})$ ($r+1\le t\le s-1$), we have  $P(\bfv(\dualgreen{c}{i})) e_t \ne 0$ and
\[(P(\bfv(\dualgreen{c}{i-1})) + P(\bfv(\dualgreen{c}{i+1})))e_t
= P(\bfv(\dualgreen{c}{i-1}))e_t+ P(\bfv(\dualgreen{c}{i+1}))e_t =0. \]
Therefore, as $\kk$-vector spaces, we have
\begin{align}
  \rmH^{-\varsigma_{i}^{\tc}}(\comp{P}) e_t &
  \ge \big(P(\bfv(\dualgreen{c}{i}))/(P(\bfv(\dualgreen{c}{i-1}))+P(\bfv(\dualgreen{c}{i+1})))\big)e_t \nonumber \\
& \ge P(\bfv(\dualgreen{c}{i}))e_t/(P(\bfv(\dualgreen{c}{i-1}))+P(\bfv(\dualgreen{c}{i+1})))e_t
  = P(\bfv(\dualgreen{c}{i}))e_t \ne 0 \nonumber
\end{align}
Similarly, for two arrows $\alpha_l$ ($1\le l<u$) and $\beta_k$ ($1\le k<v$), we have
\begin{center}
$\rmH^{-\varsigma_{i}^{\tc}}(\comp{P})\alpha_l\ne 0$ and $\rmH^{-\varsigma_{i}^{\tc}}(\comp{P})\beta_k\ne 0$.
\end{center}
Then it is clearly that any $\Dblue$-arc segment $(c_{\proj}(\dualgreen{c}{i}))^c_{[j,j+1]}$
($r+1\le j\le s-2$) is also a $\Dblue$-arc segment of a permissible curve corresponding to
some direct summand of $\rmH^{-\varsigma_{i}^{\tc}}(\comp{P})$.
\end{proof}

Keep the notations from Proposition \ref{prop:HomologiesCaseI}. Similarly, we can obtain the following result.

\begin{proposition} \label{prop:HomologiesCaseII}
Assume $n(\tc)\ge 2$ {\rm(}see \Pic \ref{fig:HomologiesCaseII}{\rm)}. Then there exists an integer $1\le r \le m(c_{\proj}(\dualgreen{c}{2}))$
    such that $(c_{\proj}(\dualgreen{c}{2}))_{(0,1)}$, $\ldots$, $(c_{\proj}(\dualgreen{c}{2}))_{(r-1,r)}$ are $\Dblue$-arc segments of $c_{\proj}(\dualgreen{c}{1})$.
Furthermore,
\[\tru_{r+1}^{m+1}(c_{\proj}(\dualgreen{c}{1})) \sqsubseteq \MM^{-1}(\rmH^{-\varsigma^{\tc}_1}(\comp{P})).\]
In particular, if $n(\tc)=2$, then $\rmH^{-\varsigma^{\tc}_1}(\comp{P})$ is indecomposable. Denote by  $c'=\MM^{-1}(\rmH^{-\varsigma^{\tc}_1}(\comp{P}))$, in this case,
\[\tru_{r+1}^{m+1}(c_{\proj}(\dualgreen{c}{1})) = \tru_1^{m(c')}(c').\]
\end{proposition}

\begin{figure}[H]
\centering
\definecolor{ffxfqq}{rgb}{1,0.5,0}
\definecolor{qqqqff}{rgb}{0,0,1}
\definecolor{ttqqqq}{rgb}{0.2,0,0}
\definecolor{ffqqqq}{rgb}{1,0,0}
\begin{tikzpicture}[scale=1.4]
\draw [line width=1.6pt,color=ffqqqq] ( 0,1)-- (-1,2);
\draw [line width=1.6pt,color=ffqqqq] (-1,2)-- (-3,2);
\draw [line width=1.6pt,dash pattern=on 2pt off 2pt,color=ffqqqq] (-3,2)-- (-4,1);
\draw [line width=1.6pt,color=ffqqqq] (-4,1)-- (-4,-1);
\draw [line width=1.6pt,dash pattern=on 2pt off 2pt,color=ffqqqq] (-4,-1)-- (-3,-2);
\draw [line width=1.6pt,color=ttqqqq] (-3,-2)-- (-1,-2);
\draw [line width=1.6pt,dash pattern=on 2pt off 2pt,color=ffqqqq] (-1,-2)-- (0,-1);
\draw [line width=1.6pt,color=ffqqqq] (0,-1)-- (0,1);
\draw [line width=1pt,dash pattern=on 2pt off 2pt,color=qqqqff] (-2.03,-2)-- (-3.49,-1.51);
\draw [line width=1pt,color=qqqqff] (-2.03,-2)-- (-4,-0.08);
\draw [line width=1pt,dash pattern=on 2pt off 2pt,color=qqqqff] (-2.03,-2)-- (-3.52,1.48);
\draw [line width=1pt,color=qqqqff] (-2.03,-2)-- (-2.02,2);
\draw [line width=1pt,color=qqqqff] (-2.03,-2)-- (-0.53,1.53);
\draw [line width=1pt,color=qqqqff] (-2.03,-2)-- (0,-0.06);
\draw [line width=1pt,dash pattern=on 2pt off 2pt,color=qqqqff] (-2.03,-2)-- (-0.55,-1.55);
\draw [line width=1pt,dash pattern=on 2pt off 2pt,color=qqqqff] (1.99,2)-- (0.5,1.5);
\draw [line width=1pt,color=qqqqff] (1.99,2)-- (0,-0.06);
\draw [line width=1pt,dash pattern=on 2pt off 2pt,color=qqqqff] (1.99,2)-- (0.52,-1.51);
\draw [line width=1pt,color=qqqqff] (1.99,2)-- (1.97,-2);
\draw [ red] (-4.  ,-0.75) node[ left]{$\ta^c_{\rbullet, 1}$};
\draw [ red] ( 0.  , 0.75) node[right]{$\ta^c_{\rbullet, 2}$};
\draw [blue] (-3.  ,-1.  ) node[above]{$\agreen{p}{r}$};
\draw [line width=1pt][->] (-4   ,-0.56) -- (-3.62,-1.38);
\draw [line width=1pt][->] (-3.70, 1.3 ) -- (-4.00, 0.5 ); \draw (-3.70, 1.3 ) node[ left]{$\alpha_u$};
\draw [line width=1pt][->] (-2.48, 2   ) -- (-3.27, 1.73);
\draw [line width=1pt][->] (-0.7 , 1.7 ) -- (-1.44, 2   ); \draw (-1.44, 2   ) node[above]{$\alpha_2$};
\draw [line width=1pt][->] ( 0.  , 0.3 ) -- (-0.33, 1.33); \draw (-0.33, 1.33) node[right]{$\alpha_1$};
\draw [line width=1.6pt,violet][->] (-4.87,-0.  ) -- (-0.45,-0.  ) node[above]{$\tc$};
\draw [line width=1.6pt,violet]     (-0.5 ,-0.  ) to[out=  0,in=-90] ( 2.00, 2.00);
\fill [blue] (-4.87, 0.  ) circle (0.1cm);
\fill [blue] ( 2.  , 2.  ) circle (0.1cm);
\begin{scriptsize}
\fill [red] ( 0.00,-1.00) circle (0.1cm); \fill [white] ( 0.00,-1.00) circle (0.07cm);
\fill [red] ( 0.00, 1.00) circle (0.1cm); \fill [white] ( 0.00, 1.00) circle (0.07cm);
\fill [red] (-1.00, 2.00) circle (0.1cm); \fill [white] (-1.00, 2.00) circle (0.07cm);
\fill [red] (-3.00, 2.00) circle (0.1cm); \fill [white] (-3.00, 2.00) circle (0.07cm);
\fill [red] (-4.00, 1.00) circle (0.1cm); \fill [white] (-4.00, 1.00) circle (0.07cm);
\fill [red] (-4.00,-1.00) circle (0.1cm); \fill [white] (-4.00,-1.00) circle (0.07cm);
\fill [red] (-3.00,-2.00) circle (0.1cm); \fill [white] (-3.00,-2.00) circle (0.07cm);
\fill [red] (-1.00,-2.00) circle (0.1cm); \fill [white] (-1.00,-2.00) circle (0.07cm);
\fill [blue] (-2.00,-2.00) circle (0.1cm);
\end{scriptsize}
\draw [orange][line width=1.5pt][opacity=0.75]
      (-3.4,-2. ) to[out= 120, in= -90] (-3.8, 0. ) to[out=  90, in=-180]
      (-2. , 1.8) to[out=   0, in= 100] ( 0. , 0. ) to[out= -80, in= 180]
      ( 2.3,-1.8) node[right]{$p=c_{\proj}(\dualgreen{c}{i})$};
\draw [green][line width=1.5pt]
      (-3.2,-2. ) to[out=  90, in=   0] (-4.5, 0.25) node[left]{$c_{\proj}(\dualgreen{c}{1})$};
\draw [cyan][line width=1.5pt]
      (-3.3, 0.9) to[out=  45, in= 180] (-2. , 1.6) to[out=   0, in= 110]
      ( 0. , 0. ) to[out= -70, in= 180] ( 2.3,-1.4)
      node[right]{$\tru_{r+1}^{m+1}(c_{\proj}(\dualgreen{c}{i}))$};
\draw [blue] (-2,-2) node[below]{$\m$};
\end{tikzpicture}
\caption{Admissible curve $\tc$ crosses at least two $\Dred$-arcs, where $\m$ in $\PP_1^{\tc}$ is to the right of $c_{[1,2]}$.\emph{}}
\label{fig:HomologiesCaseII}
\end{figure}


Recall that any $\Dblue$-arc segment $s$ is a function from $[t_1,t_2]$ ($t_1<t_2$) to $\Surf$, where $s(t_1)$, $s(t_2) \in \partial\PP$ for some elementary $\gbullet$-polygon $\PP$.
Let $s=\{s_{(i,i+1)}:[t^s_i, t^s_{i+1}]\to\Surf\}_{1\le i\le m-1}$ be a sequence of $\Dblue$-arc segments.
If $s$ is a truncation, then
$$s_{(j-1,j)}(t^s_j) = s_{(j,j+1)}(t^s_j),$$
for all $1<j<m-1$.
Furthermore, its {\defines completion}, say $\hs$, is the permissible curve such that
\[ m=m(s) \text{ and }  \ \hs_{(i,i+1)} = s_{(i,i+1)}\  (1\le i\le m-1). \]

Let $\tc$ be an admissible curve, $\ared{\tc}{i-1}$, $\ared{\tc}{i}$ and $\ared{\tc}{i+1}$ be three graded $\Dred$-arcs crossed by $c$,
and $\tau_1$ and $\tau_2$ be two truncations of projective permissible curve.
We define the {\defines supplemental union} of $\tau_1$ and $\tau_2$ with respect to $\tc$ to be
\[ \tau_1 \overrightarrow{\cup}_{\tc}\ \tau_2
   \ (=\tau_1 \overrightarrow{\cup}_{\tc}\ \tau_2 \text{ for simplicity})\
:= \tau_1\cup s_{*} \cup \tau_2, \]
if
$\tau_1$ and $\tau_2$ respectively are two truncations of $c_{\proj}(\dualgreen{c}{i-1})$ and $c_{\proj}(\dualgreen{c}{i+1})$ such that
the the positional relationship of $\tau_1=s_1$ (resp. $={\color{cyan}s}$),
$\tau_2={\color{cyan}s'}$ (resp. $= s_2$)
and $s_{*} = {\color{blue!50}s_{\rmII}}$ (resp. $={\color{blue!50}s_{\rmI}}$)
are shown in \Pic \ref{fig:supp.union}.
Otherwise, $\tau_1 \overrightarrow{\cup}_{\tc}\ \tau_2 := \{\tau_1, \tau_2\}$.

Assume that $s_1$, $s_2$, ${\color{cyan}s}$, ${\color{cyan}s'}$, ${\color{blue!50}s_{\rmI}}$, ${\color{blue!50}s_{\rmII}}$ and ${\color{cyan}s^{\star}}$ are curves shown in \Pic \ref{fig:supp.union}.
It is natural that $s^{\star}$ is a point which can be seen as a trivial curve. We have
\begin{center}
$s\overrightarrow{\cup}s^{\star}=s\cup s_{\rmI}\cup s^{\star}$,
$s^{\star}\overrightarrow{\cup}s' = s^{\star} \cup s_{\rmII} \cup s'$,
and
$s\overrightarrow{\cup}s^{\star}\overrightarrow{\cup}s' =s\cup s_{\rmI}\cup s^{\star}\cup s_{\rmII} \cup s'$.
\end{center}
Let $s_1$, $s_2$, $\ldots$, $s_n$ be truncations of projective permissible curves. By induction, we have
${\mathop{\overrightarrow{\bigcup}}\limits_{i=1}^n} s_i = \Big({\mathop{\overrightarrow{\bigcup}}\limits_{i=1}^{n-1}}s_i\Big) \overrightarrow{\cup} s_n$.
Furthermore, we define
\[\Ocup\limits_{i=1}^n s_i = \widehat{{\mathop{\overrightarrow{\bigcup}}_{i=1}^n} s_i}. \]


\begin{figure}[H]
\centering
\definecolor{ffxfqq}{rgb}{1,0.5,0}
\definecolor{qqqqff}{rgb}{0,0,1}
\definecolor{ttqqqq}{rgb}{0.2,0,0}
\definecolor{ffqqqq}{rgb}{1,0,0}
\begin{tikzpicture}[scale=1.4]
\draw [line width=1pt,color=ffqqqq] (0,-1)-- (-1,-2);
\draw [line width=1pt,color=ffqqqq] (-1,-2)-- (-3,-2);
\draw [line width=1pt,dash pattern=on 2pt off 2pt,color=ffqqqq] (-3,-2)-- (-4,-1);
\draw [line width=1pt,color=ffqqqq] (-4,-1)-- (-4, 1);
\draw [line width=1pt,dash pattern=on 2pt off 2pt,color=ffqqqq] (-4, 1)-- (-3, 2);
\draw [line width=1pt,color=ttqqqq] (-3, 2)-- (-1, 2);
\draw [line width=1pt,dash pattern=on 2pt off 2pt,color=ffqqqq] (-1, 2)-- ( 0, 1);
\draw [line width=1pt,color=ffqqqq] ( 0, 1)-- ( 0,-1);
\draw [line width=1pt,dash pattern=on 2pt off 2pt,color=ffqqqq] ( 1,-2)-- ( 0,-1);
\draw [line width=1pt] (1,-2)-- (3,-2);
\draw [line width=1pt,dash pattern=on 2pt off 2pt,color=ffqqqq] ( 3,-2)-- ( 4,-1);
\draw [line width=1pt,color=ffqqqq] ( 4,-1)-- (4, 1);
\draw [line width=1pt,dash pattern=on 2pt off 2pt,color=ffqqqq] ( 4, 1)-- (3, 2);
\draw [line width=1pt,color=ffqqqq] ( 3, 2)-- (1, 2);
\draw [line width=1pt,color=ffqqqq] ( 1, 2)-- (0, 1);
\draw [line width=1pt,dash pattern=on 2pt off 2pt,color=qqqqff] (-2.03, 2)-- (-3.49, 1.51);
\draw [line width=1pt,color=qqqqff] (-2.03, 2)-- (-4, 0.08);
\draw [line width=1pt,dash pattern=on 2pt off 2pt,color=qqqqff] (-2.03, 2)-- (-3.52,-1.48);
\draw [line width=1pt,color=qqqqff] (-2.03, 2)-- (-2.02,-2);
\draw [line width=1pt,color=qqqqff] (-2.03, 2)-- (-0.53,-1.53);
\draw [line width=1pt,color=qqqqff] (-2.03, 2)-- (0, 0.06);
\draw [line width=1pt,dash pattern=on 2pt off 2pt,color=qqqqff] (-2.03, 2)-- (-0.55, 1.55);
\draw [line width=1pt,dash pattern=on 2pt off 2pt,color=qqqqff] (1.99,-2)-- (0.5,-1.5);
\draw [line width=1pt,color=qqqqff] (1.99,-2)-- (0, 0.06);
\draw [line width=1pt,color=qqqqff] (1.99,-2)-- (0.52, 1.51);
\draw [line width=1pt,color=qqqqff] (1.99,-2)-- (1.97, 2);
\draw [line width=1pt,dash pattern=on 2pt off 2pt,color=qqqqff] (1.99,-2)-- (3.54, 1.46);
\draw [line width=1pt,color=qqqqff] (1.99,-2)-- (4, 0.03);
\draw [line width=1pt,dash pattern=on 2pt off 2pt,color=qqqqff] (1.99,-2)-- (3.43,-1.57);
\begin{scriptsize}
\fill [red] ( 0.00, 1.00) circle (0.1cm); \fill [white] ( 0.00, 1.00) circle (0.07cm);
\fill [red] ( 0.00,-1.00) circle (0.1cm); \fill [white] ( 0.00,-1.00) circle (0.07cm);
\fill [red] (-1.00,-2.00) circle (0.1cm); \fill [white] (-1.00,-2.00) circle (0.07cm);
\fill [red] (-3.00,-2.00) circle (0.1cm); \fill [white] (-3.00,-2.00) circle (0.07cm);
\fill [red] (-4.00,-1.00) circle (0.1cm); \fill [white] (-4.00,-1.00) circle (0.07cm);
\fill [red] (-4.00, 1.00) circle (0.1cm); \fill [white] (-4.00, 1.00) circle (0.07cm);
\fill [red] (-3.00, 2.00) circle (0.1cm); \fill [white] (-3.00, 2.00) circle (0.07cm);
\fill [red] (-1.00, 2.00) circle (0.1cm); \fill [white] (-1.00, 2.00) circle (0.07cm);
\fill [red] ( 1.00,-2.00) circle (0.1cm); \fill [white] ( 1.00,-2.00) circle (0.07cm);
\fill [red] ( 3.00,-2.00) circle (0.1cm); \fill [white] ( 3.00,-2.00) circle (0.07cm);
\fill [red] ( 4.00,-1.00) circle (0.1cm); \fill [white] ( 4.00,-1.00) circle (0.07cm);
\fill [red] ( 4.00, 1.00) circle (0.1cm); \fill [white] ( 4.00, 1.00) circle (0.07cm);
\fill [red] ( 3.00, 2.00) circle (0.1cm); \fill [white] ( 3.00, 2.00) circle (0.07cm);
\fill [red] ( 1.00, 2.00) circle (0.1cm); \fill [white] ( 1.00, 2.00) circle (0.07cm);
\draw [red] ( 0.  ,-0.75) node[left]{$\ta^c_{\rbullet, i}$};
\draw [red] (-4.  ,-0.75) node[ left]{$\ta^c_{\rbullet, i-1}$};
\draw [red] ( 4.  ,-0.75) node[right]{$\ta^c_{\rbullet, i+1}$};
\fill [blue] (-2.00, 2.00) circle (0.1cm);
\fill [blue] ( 2.00,-2.00) circle (0.1cm);
\draw [line width=1pt][->,line width=1.5pt,violet] (-5,0.5) to[out=0,in=180] (5,-0.5) node[right]{$\tc$};
\draw [green][line width=1.5pt]
      (-0.5, 2.0) to[out=-90,in=130] (0,0) to[out=-50,in= 90] ( 0.5,-2.2)
      node[below]{$p=c_{\proj}(\dualgreen{c}{i})$};
\draw [orange][line width=1.5pt] (-4.5, 0.0) node[ left]{$c_{\proj}(\dualgreen{c}{i-1})$}
      to[out=0, in=-95] (-0.6, 2.0);
\draw [orange][line width=1.5pt] ( 4.5, 0.0) node[right]{$c_{\proj}(\dualgreen{c}{i+1})$}
      to[out=180,in=85] ( 0.6,-2.0);
\draw [cyan][line width=1.5pt] (-4.5,-0.3) node[ left]{$s$}
      to[out=   0,in=-160] (-1.3, 0.3);
\draw [cyan][line width=1.5pt] ( 4.5, 0.3) node[right]{$s'$}
      to[out= 180,in=  20] ( 1.3,-0.3);
\draw [cyan][line width=1.5pt] (-0.2,-0.2) to ( 0.2,0.2) node[right]{$s^{\star}$};
\draw [blue!50][line width=1.5pt] (-1.35, 0.40) to[out=  30,in=-120] (-0.85, 0.85);
\draw [blue!50] (-1.10, 0.63) node[below right]{$s_{\rmI}$};
\draw [blue!50][line width=1.5pt] ( 1.35,-0.40) to[out= 210,in=  60] ( 0.85,-0.85);
\draw [blue!50] ( 1.10,-0.63) node[above left]{$s_{\rmII}$};
\draw [black][opacity=0.75] [line width=1.5pt]
  (-4.5, 0.3) node[left]{$s_1$} to[out=   0,in=-135] (-1.3, 1.3);
\draw [black][opacity=0.75] [line width=1.5pt]
  ( 4.5,-0.3) node[right]{$s_2$} to[out= 180,in=  45] ( 1.3,-1.3);
\end{scriptsize}
\draw [blue] (-3. , 1. ) node[below]{$\agreen{p}{r}$};
\draw [blue] ( 3. ,-1. ) node[above]{$\agreen{p}{s}$};
\draw [blue] (-2, 2) node[above]{$\m_1$};
\draw [blue] ( 2,-2) node[below]{$\m_2$};
\end{tikzpicture}
\caption{The supplemental union of $s$ and $s'$}
\label{fig:supp.union}
\end{figure}
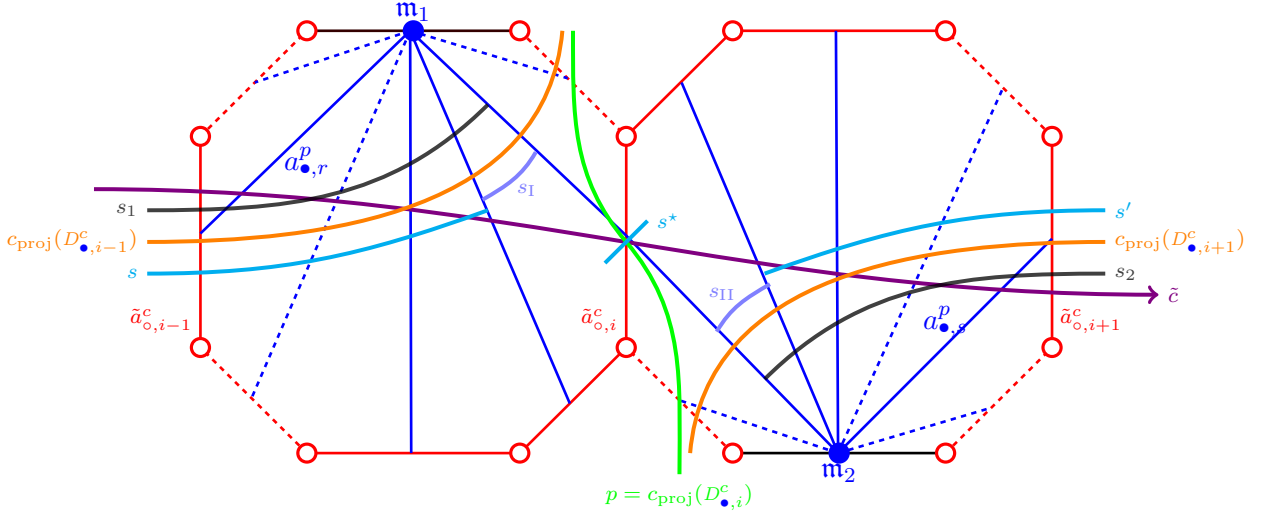

\begin{theorem} \label{thm:homologies}
Let $\tc$ be an admissible curve in $\AC_{\m}(\SURF^{\F})$.
Then, for any $n\in\ZZ$, there are integers $r_i$ and $s_i$ with $0 \le r_i \le s_i \le m(c_{\proj}(\dualgreen{c}{i}))+1$ and $1\le i\le n(\tc)$
such that
\begin{align}\label{homologies}
  \MM^{-1}(\rmH^n(\X(\tc))) = \Ocup\limits_{n=-\varsigma^{\tc}_{i}}\tru_{r_i}^{s_i}(c_{\proj}(\dualgreen{c}{i})).
\end{align}
\end{theorem}

\begin{proof}
By Propositions \ref{prop:HomologiesCaseI} and \ref{prop:HomologiesCaseII},
for any $1\le i\le n(\tc)$ satisfying $-\varsigma^{\tc}_{i} = n$,
there exist two integers $r_i$ and $s_i$ such that $\tru_{r_i}^{s_i}(c_{\proj}(\dualgreen{c}{i})) \sqsubseteq \MM^{-1}(\rmH^{n}(\X(\tc)))$.
Since the $n$-th homology of $\X(\tc)$ is given by all intersections of $\tc$ and $\rbullet$-grFFAS whose intersection index is $-n$,
we obtain (\ref{homologies}).
\end{proof}

\begin{remark} \rm
We call the truncation $\tru_{r_i}^{s_i}(c_{\proj}(\dualgreen{c}{i}))$ given in Theorem \ref{thm:homologies} the {\defines $i$-th cohomological curve} of $\tc$ at the intersection $c(i)$.
\end{remark}

The following  result is a direct corollary of Theorem \ref{thm:homologies}.

\begin{corollary}
Let $\X(\tc)$ be a string complex corresponding to $\tc \in \AC_{\m}(\SURF^{\F})$.
Then $\rmH^n(\X(\tc)) = 0$ if and only if for all $1\le i\le n(\tc)$, $i$-th cohomological curves of $\tc$  with $\varsigma^{\tc}_i = -n$ are trivial.
\end{corollary}


\section{The relation between rotations and truncations} \label{sec:inverse}
In this section, we study the relation between taking cohomology and embedding. Recall that $(-)^{\rota}$ sends any permissible curve $c$ in $\PC_{\m}(\SURF^{\F})$
to a graded curve $c^{\rota}$. Note that $c^{\rota}$ has two endpoints lying $\M$ or $c^{\rota}$ is a half-line with endpoint lying $\M$,
i.e., $c^{\rota}$ has only one endpoint lying $\M$.

\begin{lemma} \label{lemm:exact}
Let $c\in\PC_{\m}(\SURF^{\F})$. Then for any $\tc^{\rota}\in \AC_{\m}(\SURF^{\F})$,
$\rmH^{t}(\X(\tc^{\rota})) = 0$ for all $t\neq\max\{-\varsigma^{\tc^{\rota}}_j \mid 1\le j\le n(\tc^{\rota})\}$.
\end{lemma}

\begin{proof}
As shown in \Pic \ref{fig:threeparts}, we divide the permissible curve $c$ into three parts by $\rbullet$-FFAS as follows:
\begin{itemize}
  \item[(1)] the middle segment $c_{\rmM}$;
  \item[(2)] the first end segment $c_{\rmF}$;
  \item[(3)] the last end segment $c_{\rmL}$.
\end{itemize}
We only prove this lemma for the middle segment $c_{\rmM}$, the proofs for the other segments are similar.
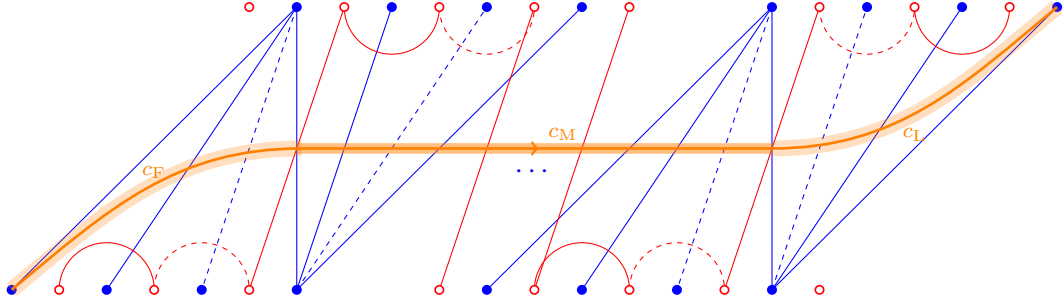
\begin{figure}[H]
\centering
\definecolor{ffqqqq}{rgb}{1,0,0}
\definecolor{qqqqff}{rgb}{0,0,1}
\begin{tikzpicture}[scale=1.25]
\draw [color=qqqqff] (-3, 3)-- (-6, 0);
\draw [color=qqqqff] (-3, 3)-- (-5, 0);
\draw [color=qqqqff] (-3, 3)-- (-4, 0) [dash pattern=on 2pt off 2pt];
\draw [color=qqqqff] (-3, 3)-- (-3, 0);
\draw [color=qqqqff] (-3, 0)-- ( 0, 3);
\draw [color=qqqqff] (-3, 0)-- (-1, 3) [dash pattern=on 2pt off 2pt];
\draw [color=qqqqff] (-3, 0)-- (-2, 3);
\draw [color=qqqqff] (-1, 0)-- ( 2, 3);
\draw [color=qqqqff] ( 2, 3)-- ( 2, 0);
\draw [color=qqqqff] ( 2, 3)-- ( 1, 0) [dash pattern=on 2pt off 2pt];
\draw [color=qqqqff] ( 2, 3)-- ( 0, 0);
\draw [color=qqqqff] ( 2, 0)-- ( 3, 3) [dash pattern=on 2pt off 2pt];
\draw [color=qqqqff] ( 2, 0)-- ( 4, 3);
\draw [color=qqqqff] ( 2, 0)-- ( 5, 3);
\draw [red] (-5.5, 0) arc(180:0:0.5);
\draw [red] (-4.5, 0) arc(180:0:0.5) [dash pattern=on 2pt off 2pt];
\draw [red] (-3.5, 0)-- (-2.5, 3);
\draw [red] (-2.5, 3) arc(-180:0:0.5);
\draw [red] (-1.5, 3) arc(-180:0:0.5) [dash pattern=on 2pt off 2pt];
\draw [red] (-0.5, 3)-- (-1.5, 0);
\draw [red] ( 0.5, 3)-- (-0.5, 0);
\draw [red] (-0.5, 0) arc( 180:0:0.5);
\draw [red] ( 0.5, 0) arc( 180:0:0.5) [dash pattern=on 2pt off 2pt];
\draw [red] ( 1.5, 0)-- ( 2.5, 3);
\draw [red] ( 2.5, 3) arc(-180:0:0.5) [dash pattern=on 2pt off 2pt];
\draw [red] ( 3.5, 3) arc(-180:0:0.5);
\begin{scriptsize}
\fill [color=qqqqff] (-3, 3) circle (1.5pt);
\fill [color=qqqqff] (-6, 0) circle (1.5pt);
\fill [color=qqqqff] (-3, 0) circle (1.5pt);
\fill [color=qqqqff] ( 0, 3) circle (1.5pt);
\fill [color=qqqqff] (-1, 0) circle (1.5pt);
\fill [color=qqqqff] ( 2, 3) circle (1.5pt);
\fill [color=qqqqff] ( 2, 0) circle (1.5pt);
\fill [color=qqqqff] ( 5, 3) circle (1.5pt);
\fill [color=qqqqff] (-5, 0) circle (1.5pt);
\fill [color=qqqqff] (-4, 0) circle (1.5pt);
\fill [color=qqqqff] (-2, 3) circle (1.5pt);
\fill [color=qqqqff] (-1, 3) circle (1.5pt);
\fill [color=qqqqff] ( 0, 0) circle (1.5pt);
\fill [color=qqqqff] ( 1, 0) circle (1.5pt);
\fill [color=qqqqff] ( 3, 3) circle (1.5pt);
\fill [color=qqqqff] ( 4, 3) circle (1.5pt);
\fill [red] (-5.5, 0) circle (1.5pt); \fill [white] (-5.5, 0) circle (1. pt);
\fill [red] (-4.5, 0) circle (1.5pt); \fill [white] (-4.5, 0) circle (1. pt);
\fill [red] (-3.5, 0) circle (1.5pt); \fill [white] (-3.5, 0) circle (1. pt);
\fill [red] (-2.5, 3) circle (1.5pt); \fill [white] (-2.5, 3) circle (1. pt);
\fill [red] (-1.5, 3) circle (1.5pt); \fill [white] (-1.5, 3) circle (1. pt);
\fill [red] (-0.5, 3) circle (1.5pt); \fill [white] (-0.5, 3) circle (1. pt);
\fill [red] ( 0.5, 3) circle (1.5pt); \fill [white] ( 0.5, 3) circle (1. pt);
\fill [red] (-1.5, 0) circle (1.5pt); \fill [white] (-1.5, 0) circle (1. pt);
\fill [red] (-0.5, 0) circle (1.5pt); \fill [white] (-0.5, 0) circle (1. pt);
\fill [red] ( 0.5, 0) circle (1.5pt); \fill [white] ( 0.5, 0) circle (1. pt);
\fill [red] ( 1.5, 0) circle (1.5pt); \fill [white] ( 1.5, 0) circle (1. pt);
\fill [red] ( 2.5, 3) circle (1.5pt); \fill [white] ( 2.5, 3) circle (1. pt);
\fill [red] ( 3.5, 3) circle (1.5pt); \fill [white] ( 3.5, 3) circle (1. pt);
\fill [red] ( 4.5, 3) circle (1.5pt); \fill [white] ( 4.5, 3) circle (1. pt);
\fill [red] (-3.5, 3) circle (1.5pt); \fill [white] (-3.5, 3) circle (1. pt);
\fill [red] ( 2.5, 0) circle (1.5pt); \fill [white] ( 2.5, 0) circle (1. pt);
\draw [orange][line width=6pt][opacity=0.25] (-6,0. ) to[out= 40,in=180] (-2.95,1.5);
      \draw[orange] (-4.5,1.25) node{$c_{\rmF}$};  
\draw [orange][line width=4pt][opacity=0.50] (-3,1.5) -- (2,1.5);
      \draw[orange] (-0.2,1.5) node[above]{$c_{\rmM}$};  
\draw [orange][line width=6pt][opacity=0.25] ( 2,1.5) to[out=  0,in=220] (5,3);
      \draw[orange] ( 3.5,1.65) node{$c_{\rmL}$};  
\draw [orange][line width=1pt] (-6,0. ) to[out= 40,in=180] (-2.95,1.5);
\draw [orange][line width=1pt] (-3,1.5) -- (2,1.5);
\draw [orange][line width=1pt][->] (-0.5,1.5) -- (-0.45,1.5);
\draw [orange][line width=1pt] ( 2,1.5) to[out=  0,in=220] (5,3);
\end{scriptsize}
\draw[blue] (-0.5, 1.25) node{$\cdots$};
\end{tikzpicture}
\caption{Dividing the permissible curve $c$ to three parts $c_{\rmF}$, $c_{\rmM}$ and $c_{\rmL}$}
\label{fig:threeparts}
\end{figure}
\noindent By Definition \ref{def:rot}, $c_{\rmM}$ is the common part of $c$ and $c^{\rota}$.
Assume that $c_{\rmM}$ crosses at least three $\rbullet$-arcs $\ared{c^{\rota}}{i-1}$, $\ared{c^{\rota}}{i}$ and $\ared{c^{\rota}}{i+1}$
such that the $\gbullet$-marked points $\m$ in $\PP_{i-1}^{\tc}$ is to the left of $c_{\rmM}$ and
the $\gbullet$-marked points $\m'$ in $\PP_i^{\tc}$ is to the right of $c_{\rmM}$ (see \Pic \ref{fig:middleparts}).
Indeed, $\m$ and $\m'$ cannot be on the same side of $c_{\rmM}$ because $c$ is a permissible curve. Moreover, if $\m$ and $\m'$ are the right and left of $c_{\rmM}$, then $\dualgreen{c^{\rota}}{i}$ is a top $\gbullet$-arc of $c$, then we have
\begin{center}
$-\varsigma^{\tc^{\rota}_i} = \max\{-\varsigma^{\tc^{\rota}}_j \mid 1\le j\le n(\tc^{\rota})\}$.
\end{center}
Thus, $c_{\proj}(\dualgreen{c^{\rota}}{i-1})$, $c_{\proj}(\dualgreen{c^{\rota}}{i})$ and $c_{\proj}(\dualgreen{c^{\rota}}{i+1})$
are projective permissible curves satisfying the following conditions.
\begin{figure}[htbp]
\centering
\definecolor{ffqqqq}{rgb}{1,0,0}
\definecolor{qqqqff}{rgb}{0,0,1}
\begin{tikzpicture}[scale=1.5]
\draw [blue] ( 0, 2)-- ( 0  ,-2  );
\draw [blue] ( 0,-2)-- ( 3.5, 1.5) [dash pattern=on 1pt off 1pt];
\draw [blue] ( 0,-2)-- ( 4  , 1  );
\draw [blue] ( 0,-2)-- ( 4.5, 0.5);
\draw [blue] ( 0, 2)-- (-3.5,-1.5) [dash pattern=on 1pt off 1pt];
\draw [blue] ( 0, 2)-- (-4  ,-1  );
\draw [blue] ( 0, 2)-- (-4.5,-0.5);
\draw [blue] ( 0, 2) arc(-180:0:0.5) [dash pattern=on 1pt off 1pt];
\draw [blue] ( 0, 2) arc(-180:0:1.0);
\draw [blue] ( 0,-2) arc(0: 180:0.5) [dash pattern=on 1pt off 1pt];
\draw [blue] ( 0,-2) arc(0: 180:1.0);
\draw [red] (-2.5, 2  )-- (-4  ,-0.5);
\draw [red] (-4  ,-0.5)-- (-3.5,-1  );
\draw [red] (-3.5,-1  )-- (-2.5,-2  );
\draw [red] ( 2.5,-2  )-- ( 4  , 0.5);
\draw [red] ( 4  , 0.5)-- ( 3.5, 1  );
\draw [red] ( 3.5, 1  )-- ( 2.5, 2  );
\draw [red] (-2.5,-2) arc(180:0:0.5);
\draw [red] (-1.5,-2) arc(180:0:0.5) [dash pattern=on 1pt off 1pt];
\draw [red] ( 2.5, 2) arc(0:-180:0.5);
\draw [red] ( 1.5, 2) arc(0:-180:0.5) [dash pattern=on 1pt off 1pt];
\draw [red] (-2.5,-2) to[out=90, in=-90] ( 2.5, 2);
\begin{scriptsize}
\fill [blue] ( 0  , 2  ) circle (1.5pt);
\fill [blue] ( 0  ,-2  ) circle (1.5pt);
\fill [blue] ( 3.5, 1.5) circle (1.5pt);
\fill [blue] ( 4  , 1  ) circle (1.5pt);
\fill [blue] ( 4.5, 0.5) circle (1.5pt);
\fill [blue] (-3.5,-1.5) circle (1.5pt);
\fill [blue] (-4  ,-1  ) circle (1.5pt);
\fill [blue] (-4.5,-0.5) circle (1.5pt);
\fill [blue] ( 2  , 2  ) circle (1.5pt);
\fill [blue] ( 1  , 2  ) circle (1.5pt);
\fill [blue] (-2  ,-2  ) circle (1.5pt);
\fill [blue] (-1  ,-2  ) circle (1.5pt);
\fill [red] (-2.5, 2) circle (1.5pt); \fill [white] (-2.5, 2) circle (1. pt);
\fill [red] (-4,-0.5) circle (1.5pt); \fill [white] (-4,-0.5) circle (1. pt);
\fill [red] (-3.5,-1) circle (1.5pt); \fill [white] (-3.5,-1) circle (1. pt);
\fill [red] (-2.5,-2) circle (1.5pt); \fill [white] (-2.5,-2) circle (1. pt);
\fill [red] (-1.5,-2) circle (1.5pt); \fill [white] (-1.5,-2) circle (1. pt);
\fill [red] (-0.5,-2) circle (1.5pt); \fill [white] (-0.5,-2) circle (1. pt);
\fill [red] ( 2.5,-2) circle (1.5pt); \fill [white] ( 2.5,-2) circle (1. pt);
\fill [red] ( 4,0. 5) circle (1.5pt); \fill [white] ( 4,0. 5) circle (1. pt);
\fill [red] ( 3.5, 1) circle (1.5pt); \fill [white] ( 3.5, 1) circle (1. pt);
\fill [red] ( 2.5, 2) circle (1.5pt); \fill [white] ( 2.5, 2) circle (1. pt);
\fill [red] ( 1.5, 2) circle (1.5pt); \fill [white] ( 1.5, 2) circle (1. pt);
\fill [red] ( 0.5, 2) circle (1.5pt); \fill [white] ( 0.5, 2) circle (1. pt);
\end{scriptsize}
\draw [orange][line width=1pt][->] (-4.5,0) -- (4.5,0);
\draw [orange] (1,0) node[above]{$c$};
\draw [violet][line width=1pt][->] (-4.5,0.5) to[out=0,in=180] (4.5,-0.5);
\draw [violet] (1,-0.2) node[below]{$c^{\rota}$};
\draw [cyan][line width=1pt] (-0.5,-1.75) to[out=90,in=-90] (0.5,1.75);
\draw [cyan][line width=1pt][->] (0,0) -- (0.09,0.18);
\draw [cyan] (-0.2,-0.4) node[left]{\tiny$p=c_{\proj}(\dualgreen{c^{\rota}}{i})$};
\draw [cyan][line width=1pt] (-0.5,-1.75) to[out=90,in=-90] (0.5,1.75);
\draw [green][line width=1pt][->] (-0.5,-1.75) to[out=45,in=135] (4,-1.75)
      node[below]{$c_{\proj}(\dualgreen{c^{\rota}}{i+1})$};
\draw [green][line width=1pt] ( 0.5, 1.75) to[out=-135,in=-45] (-4,1.75)
      node[above]{$c_{\proj}(\dualgreen{c^{\rota}}{i-1})$};
      \draw [green][line width=1pt][->] (-4,1.75) -- (-3.9,1.65);
\begin{scriptsize}
\draw (0,0.5) node[left]{$\agreen{p}{r}\!=\!\dualgreen{c^{\rota}}{i}$};
\draw ( 1, 1) node[below]{$\agreen{p}{r+1}$};
\draw (-1,-1) node[above]{$\agreen{p}{r-1}$};
\draw (0.95,-1.5) node[right]{$\dualgreen{c^{\rota}}{i+1}$};
\draw (2.8 ,-1.5) node[ left]{$\ared{c^{\rota}}{i+1}$};
\draw (-0.95,1.5) node[ left]{$\dualgreen{c^{\rota}}{i-1}$};
\draw (-2.8 ,1.5) node[right]{$\ared{c^{\rota}}{i-1}$};
\end{scriptsize}
\end{tikzpicture}
\caption{The $-\varsigma^{\tc^{\rota}}_{i}$-th cohomological curve of $\tc^{\rota}$ is zero. }
\label{fig:middleparts}
\end{figure}
\begin{itemize}
  \item $(c_{\proj}(\dualgreen{c^{\rota}}{i}))_{(0,1)}$, $(c_{\proj}(\dualgreen{c^{\rota}}{i}))_{(1,2)}$,
    $\ldots$, $(c_{\proj}(\dualgreen{c^{\rota}}{i}))_{(r-1,r)}$ are $\Dblue$-arc segments of $c_{\proj}(\dualgreen{c^{\rota}}{i+1})$,
    and \[(c_{\proj}(\dualgreen{c^{\rota}}{i}))_{(0,1)} = (c_{\proj}(\dualgreen{c^{\rota}}{i+1}))_{(0,1)}.\]
  \item $(c_{\proj}(\dualgreen{c^{\rota}}{i}))_{(r,r+1)}$, $(c_{\proj}(\dualgreen{c^{\rota}}{i}))_{(r+1,r+2)}$,
    $\ldots$, $(c_{\proj}(\dualgreen{c^{\rota}}{i}))_{(m,m+1)}$ are $\Dblue$-arc segments of $c_{\proj}(\dualgreen{c^{\rota}}{i-1})$
    and \[(c_{\proj}(\dualgreen{c^{\rota}}{i}))_{(m,m+1)} = (c_{\proj}(\dualgreen{c^{\rota}}{i-1}))_{(m',m'+1)},\]
    where $m=m(c_{\proj}(\dualgreen{c^{\rota}}{i}))$, $m'=m(c_{\proj}(\dualgreen{c^{\rota}}{i-1}))$.
  \item $-\varsigma_i^{\tc^{\rota}}\neq\max\{-\varsigma^{\tc^{\rota}}_j \mid 1\le j\le n(\tc^{\rota})\}$.
\end{itemize}
By Proposition \ref{prop:HomologiesCaseI}, $\tru^{r-1}_{r+1}(c_{\proj}(\dualgreen{c^{\rota}}{i}))$ is trivial.
Then, by Theorem \ref{thm:homologies}, we have
\[ \MM^{-1}(\rmH^{t}(\X(\tc))) = \Ocup_{t=-\varsigma_i^{\tc^{\rota}}}\tru_{r_i}^{s_i}(c_{\proj}(\dualgreen{c}{i})) = 0 \]
when $t\neq\max\{-\varsigma^{\tc^{\rota}}_t \mid 1\le t\le n(\tc^{\rota})\}$.
\end{proof}

\begin{theorem} \label{thm:string embedding 2}
Any permissible curve $c\in\PC_{\m}(\SURF^{\F})$ is the completion of the supplemental union of all cohomological curves of $\tc^{\rota}$.
\end{theorem}

\begin{proof}
By Lemma \ref{lemm:exact}, the permissible curve $c$ has three parts by $\rbullet$-FFAS: the middle segment $c_{\rmM}$, the first end segment $c_{\rmF}$
and the last end segment $c_{\rmL}$. For the middle segment $c_{\rmM}$, it is clear that the supplemental union of two adjacent cohomology curves of $\tc^{\rota}$ is a truncation of $c^{\rota}$
which contains $c^{\rota}_{[i-1,i]}\cup c^{\rota}_{[i,i+1]}$.
\end{proof}

\begin{remark}
Given a grading of $\tc^{\rota}$ such that the minimal intersection index given by $\tc^{\rota}$ intersecting with $\rbullet$-grFFAS equals to zero, then we obtain the following embedding.
\[(-)^\rota: \PC_{\m}(\SURF(A)^{\F_A}) \to \AC_{\m}(\SURF(A)^{\F_A}),
c\mapsto \tc^{\rota}.\]
\end{remark}

\section{Applications} \label{sec:applications}

\subsection{The cohomological length of complexes in derived category of gentle algebras}
In this subsection, in terms of the geometric characterization of the cohomology of complexes, we give an alternative proof ``no gaps" theorem as to cohomological length for the bounded derived categories of gentle algebras \cite{Z2019}.
Let $A$ be a gentle algebra and $\comp{X}$ be a complex in the bounded derived category $\Dcat^b(A)$. Recall that the cohomological length of $\comp{X}$ is
\begin{align}
 & \hl(\comp{X}) = \max\{\dim_{\kk}\H^i(\comp{X}) \mid i\in \ZZ\},  \nonumber  
\end{align}
which was introduced by the second author and Han in \cite{ZH2016} to study of Brauer-Thrall type theorems for the bounded derived category.
Recall that the ``no gaps" theorem of the module category with respect to the length of indecomposable modules were studied by Bongartz \cite{Bon2013} and Ringel\cite{Rin2011}. They proved that if there is an indecomposable module of length $n>1$, then there exists an indecomposable module of length $n-1$ for a finite-dimensional algebra. Motivated by this result, the second author and Han in \cite{ZH2016} studied the Brauer-Thrall type theorems for derived category of a gentle algebra and obtained that if there is an indecomposable object in $\Dcat^b(A)$ of length $n>1$, then there exists an indecomposable with cohomological length $n-1$.

As we know, there is a full embedding from $\modcat A$ to $\Dcat^b(A)$, it is obvious that the dimension of an $A$-module $M$ is equal to the cohomological length of the stalk complex $M$. In terms of the embedding given in Theorem \ref{thm:string embedding} and Corollary \ref{thm:band embedding}, we revisit and give an alternative proof for the Brauer-Thrall type theorems for the bounded derived category of a gentle algebra.
We begin with the following useful lemma.

\begin{lemma} \label{lemm:hl invariant}
Let $\SURF^{\F}$ be a graded marked ribbon surface and $\tc \in \AC_{\m}(\SURF^{\F})$ be an admissible curve with $c= \{c_{[i,i+1]}\}_{0\le i \le n(\tc)}$ $(n(\tc)\ge 2)$.
Assume $\hl(\X(\tc))=\dim_{\kk}\H^{-\varsigma}(\X(\tc))$ for some integer $\varsigma$
and $k = \min\{1\le i\le n(\tc) \mid \varsigma^{\tc}_i = \varsigma\}$.
If $k\ge 2$, then there is an admissible curve $\tc' \in \AC_{\m}(\SURF^{\F})$ with $c'= \{c'_{[j,j+1]}\}_{0\le j \le n(\tc')}$ such that
\begin{itemize}
  \item $\max\limits_{\varsigma^{\tc'}_{1}\ne t\ne \varsigma^{\tc'}_{n(\tc')}}\dim_{\kk}\H^{-t}(\X(\tc')) = \hl(\X(\tc))$;
  \item $\min\{1<j<n(\tc') \mid \varsigma^{\tc'}_j = \varsigma'\}=2$;
  \item $\ared{\tc'}{1} = \ared{\tc}{k-1}$ for any $k\geq2$.
\end{itemize}
\end{lemma}

\begin{proof}
If $k=2$ or $k= n(\tc)-1$, then $\tc'=\tc$ or $\tc'=\tc^{-1}$. Next, we construct $\tc'$ for $2<k<n(\tc)-1$.
There are four Cases (A), (B), (C) and (D) given in subsection \ref{sec:truncations}.
We only prove Case (A), the proofs of (B), (C) and (D) are similar.
In this case (see \Pic \ref{fig:hl invariant}), by Proposition \ref{prop:HomologiesCaseI}, there are $r,s\in\NN^+$ such that
\[\tru_{r+1}^{s-1}(c_{\proj}(\dualgreen{c}{k})) \sqsubseteq \MM^{-1}(\rmH^{-\varsigma}(\X(\tc)))
= \MM^{-1}\left(\bigoplus_{1\le\imath\le M} \MM(c_{\imath})\right)\]
for some $c_1, \ldots, c_M$. Since $k = \min\{1\le i\le n(\tc) \mid \varsigma^{\tc}_i = \varsigma\}$, by Theorem \ref{thm:homologies}, we can assume that
\begin{align}\label{tru}
\rmH^{-\varsigma}(\X(\tc)) \cong \MM(T_0 \ocup \Ocup\limits_{1\le\jmath\le N} T_{\jmath}),
\end{align}
where $T_0=\tru_{r+1}^{s-1}(c_{\proj}(\dualgreen{c}{k}))$ and $T_1$, $\ldots$, $T_N$ are some truncations of projective permissible curves.
Then the curve $T_{0}^{\star}:=\tru_{0}^{m(\widehat{T_0})-1}(\widehat{T_0})$ is a truncation of some $c_{\imath}$ ($1\le \imath\le M$).
\begin{figure}[htbp]
\centering
\definecolor{ffxfqq}{rgb}{1,0.5,0}
\definecolor{qqqqff}{rgb}{0,0,1}
\definecolor{ttqqqq}{rgb}{0.2,0,0}
\definecolor{ffqqqq}{rgb}{1,0,0}
\begin{tikzpicture}[scale=1.4]
\draw [line width=1pt,color=ffqqqq] (0,1)-- (-1,2);
\draw [line width=1pt,dash pattern=on 2pt off 2pt,color=ffqqqq] (-1,2)-- (-3,2);
\draw [line width=1pt,color=ffqqqq] (-3,2)-- (-4,1);
\draw [line width=1pt,color=ffqqqq] (-4,1)-- (-4,-1);
\draw [line width=1pt,dash pattern=on 2pt off 2pt,color=ffqqqq] (-4,-1)-- (-3,-2);
\draw [line width=1pt,color=ttqqqq] (-3,-2)-- (-1,-2);
\draw [line width=1pt,dash pattern=on 2pt off 2pt,color=ffqqqq] (-1,-2)-- (0,-1);
\draw [line width=1pt,color=ffqqqq] (0,-1)-- (0,1);
\draw [line width=1pt,dash pattern=on 2pt off 2pt,color=ffqqqq] (1,2)-- (0,1);
\draw [line width=1pt] (1,2)-- (3,2);
\draw [line width=1pt,dash pattern=on 2pt off 2pt,color=ffqqqq] (3,2)-- (4,1);
\draw [line width=1pt,color=ffqqqq] (4,1)-- (4,-1);
\draw [line width=1pt,dash pattern=on 2pt off 2pt,color=ffqqqq] (4,-1)-- (3,-2);
\draw [line width=1pt,color=ffqqqq] (3,-2)-- (1,-2);
\draw [line width=1pt,color=ffqqqq] (1,-2)-- (0,-1);
\draw [line width=1pt,dash pattern=on 2pt off 2pt,color=qqqqff] (-2.03,-2)-- (-3.49,-1.51);
\draw [line width=1pt,color=qqqqff] (-2.03,-2)-- (-4,-0.08) -- (-5,0.9);
\draw [line width=1pt,color=qqqqff] (-2.03,-2)-- (-3.52,1.48)-- (-3.8,2.13);
\draw [line width=1pt,dash pattern=on 2pt off 2pt,color=qqqqff] (-2.03,-2)-- (-2.02,2);
\draw [line width=1pt,color=qqqqff] (-2.03,-2)-- (-0.53,1.53);
\draw [line width=1pt,color=qqqqff] (-2.03,-2)-- (0,-0.06);
\draw [line width=1pt,dash pattern=on 2pt off 2pt,color=qqqqff] (-2.03,-2)-- (-0.55,-1.55);
\draw [line width=1pt,dash pattern=on 2pt off 2pt,color=qqqqff] (1.99,2)-- (0.5,1.5);
\draw [line width=1pt,color=qqqqff] (1.99,2)-- (0,-0.06);
\draw [line width=1pt,color=qqqqff] (1.99,2)-- (0.52,-1.51);
\draw [line width=1pt,color=qqqqff] (1.99,2)-- (1.97,-2);
\draw [line width=1pt,dash pattern=on 2pt off 2pt,color=qqqqff] (1.99,2)-- (3.54,-1.46);
\draw [line width=1pt,color=qqqqff] (1.99,2)-- (4,-0.03);
\draw [line width=1pt,dash pattern=on 2pt off 2pt,color=qqqqff] (1.99,2)-- (3.43,1.57);
\draw [line width=1pt][->] (-4   ,-0.56) -- (-3.62,-1.38);
\draw [line width=1pt][->] (-3.70, 1.3 ) -- (-4.00, 0.5 ); \draw (-3.70, 1.3 ) node[ left]{$\alpha_u$};
\draw [line width=1pt][->] (-2.48, 2   ) -- (-3.27, 1.73);
\draw [line width=1pt][->] (-0.7 , 1.7 ) -- (-1.44, 2   ); \draw (-1.44, 2   ) node[above]{$\alpha_2$};
\draw [line width=1pt][->] ( 0.  , 0.3 ) -- (-0.33, 1.33); \draw (-0.33, 1.33) node[right]{$\alpha_1$};
\draw [line width=1pt][->] ( 0.  ,-0.34) -- ( 0.24,-1.24); \draw ( 0.24,-1.24) node[below]{$\beta_1$};
\draw [line width=1pt][->] ( 0.68,-1.68) -- ( 1.48,-2   ); \draw ( 1.48,-2   ) node[below]{$\beta_2$};
\draw [line width=1pt][->] ( 2.4 ,-2   ) -- ( 3.32,-1.68);
\draw [line width=1pt][->] ( 3.77,-1.23) -- ( 4   ,-0.44); \draw ( 3.77,-1.23) node[right]{$\beta_v$};
\draw [line width=1pt][->] ( 4   , 0.43) -- ( 3.67, 1.33);
\draw [line width=1pt][->,line width=1.5pt,violet] (-5,0) -- (5,0) node[right]{$\tc$};
\draw [line width=1pt][->,line width=1.5pt,black] (-5,1) to[out=-45,in=180] (-2.5,0.3) -- (5,0.3) node[right]{$\tc'$}; 
\draw [line width=1pt][->,line width=1.5pt,gray] (-3.8,2) to[out=-45,in=180] (5,0.6) node[right]{$\tc''$}; 
\begin{scriptsize}
\fill [red] ( 0.00,-1.00) circle (0.1cm); \fill [white] ( 0.00,-1.00) circle (0.07cm);
\fill [red] ( 0.00, 1.00) circle (0.1cm); \fill [white] ( 0.00, 1.00) circle (0.07cm);
\fill [red] (-1.00, 2.00) circle (0.1cm); \fill [white] (-1.00, 2.00) circle (0.07cm);
\fill [red] (-3.00, 2.00) circle (0.1cm); \fill [white] (-3.00, 2.00) circle (0.07cm);
\fill [red] (-4.00, 1.00) circle (0.1cm); \fill [white] (-4.00, 1.00) circle (0.07cm);
\fill [red] (-4.00,-1.00) circle (0.1cm); \fill [white] (-4.00,-1.00) circle (0.07cm);
\fill [red] (-3.00,-2.00) circle (0.1cm); \fill [white] (-3.00,-2.00) circle (0.07cm);
\fill [red] (-1.00,-2.00) circle (0.1cm); \fill [white] (-1.00,-2.00) circle (0.07cm);
\fill [red] ( 1.00, 2.00) circle (0.1cm); \fill [white] ( 1.00, 2.00) circle (0.07cm);
\fill [red] ( 3.00, 2.00) circle (0.1cm); \fill [white] ( 3.00, 2.00) circle (0.07cm);
\fill [red] ( 4.00, 1.00) circle (0.1cm); \fill [white] ( 4.00, 1.00) circle (0.07cm);
\fill [red] ( 4.00,-1.00) circle (0.1cm); \fill [white] ( 4.00,-1.00) circle (0.07cm);
\fill [red] ( 3.00,-2.00) circle (0.1cm); \fill [white] ( 3.00,-2.00) circle (0.07cm);
\fill [red] ( 1.00,-2.00) circle (0.1cm); \fill [white] ( 1.00,-2.00) circle (0.07cm);
\draw [red] ( 0.  , 0.75) node[right]{$\ta^c_{\rbullet, k}$};
\draw [red] (-4.  ,-0.75) node[ left]{$\ta^c_{\rbullet, k-1}$};
\draw [red] ( 4.  ,-0.75) node[right]{$\ta^c_{\rbullet, k+1}$};
\fill [blue] (-5.  , 0.9 ) circle (0.1cm); 
\fill [blue] (-3.8 , 2.13) circle (0.1cm); 
\fill [blue] (-2.00,-2.00) circle (0.1cm);
\fill [blue] ( 2.00, 2.00) circle (0.1cm);
\draw [orange][line width=1.5pt][opacity=0.75]
      (-3.4,-2. ) to[out= 120, in= -90] (-3.8, 0. ) to[out=  90, in=-180]
      (-2. , 1.8) to[out=   0, in= 100] ( 0. , 0. ) to[out= -80, in= 180]
      ( 2. ,-1.8) to[out=   0, in= -90] ( 3.8, 0. ) to[out=  90, in= -60]
      ( 3.4, 2. );
\draw [orange] (-1.8, 2.20) node[above]{$p=c_{\proj}(\dualgreen{c}{k})$};
\draw [orange][->] (-1.8, 1.80) to (-1.8, 2.20);
\draw [green][line width=1.5pt]
      (-3.2,-2. ) to[out=  90, in=   0] (-4.5, 0.25) node[left]{$c_{\proj}(\dualgreen{c}{k-1})$};
\draw [green][line width=1.5pt]
      ( 3.2, 2. ) to[out= -90, in= 180] ( 4.5,-0.25) node[right]{$c_{\proj}(\dualgreen{c}{k+1})$};
\draw [cyan][line width=1.5pt]
      (-3.3, 0.9) to[out=  45, in= 180] (-2. , 1.6) to[out=   0, in= 110]
      ( 0. , 0. ) to[out= -70, in= 180] ( 2. ,-1.6) to[out=   0, in=-135]
      ( 3.3,-0.9);
\draw [cyan][->] (-2.7,1.4) -- (-2.7,2.5) node[above]{$T_0=\tru_{r+1}^{s-1}(c_{\proj}(\dualgreen{c}{k}))$};
\draw [cyan][line width=4pt][opacity=0.5]
      (-5. , 0.9) to[out= -45, in= 225]
      (-3.2, 0.7) to[out=  45, in= 180] (-2. , 1.4) to[out=   0, in= 110]
      ( 0. , 0. ) to[out= -70, in= 180] ( 2. ,-1.4) to[out=   0, in=-135]
      ( 3.2,-0.7);
\draw [cyan][->] (3,-0.9) -- (3,-2.4) node[below]{$T_0^{\star}=\tru_0^{m(\widehat{T_0})-1}(\widehat{T_0})$};
\draw [purple][line width=1pt] (-2. , 1.2) to[out=   0, in= 120]
      ( 0. , 0. ) to[out= -60, in= 180] ( 2. ,-1.2) to[out=   0, in=-135]
      (3.1,-0.45);
\draw [purple][->] (-0.7,0.83) -- (-0.7,2.5);
\draw [purple] (-0.4,2.5)node[above]{$U_0=\tru_{r+2}^{s-1}(c_{\proj}(\dualgreen{c}{k}))$};
\draw [purple][line width=3pt][opacity=0.5]
      (-3.8 , 2.13) to[out= -65, in= 180] (-2. , 1) to[out=   0, in= 135]
      ( 0.  ,  0. ) to[out= -45, in= 180] ( 2. ,-1) to[out=   0, in=-135]
      ( 3. ,-0.25);
\draw [purple][->] (0.5,-0.5) -- (0.5,-2.4) node[below]{$U_0^{\star}=\tru_0^{m(\widehat{U_0})-1}(\widehat{U_0})$};
\end{scriptsize}
\draw [blue] (-3. ,-1. ) node[above]{$\agreen{p}{r}$};
\draw [blue] ( 3. , 1. ) node[below]{$\agreen{p}{s}$};
\draw [blue] (-2,-2) node[below]{$\m_1$}; \draw (-1.8,-1) node{$\PP_{k-1}^{\tc}$};
\draw [blue] ( 2, 2) node[above]{$\m_2$}; \draw ( 1.6,-0.5) node{$\PP_{k}^{\tc}$};
\end{tikzpicture}
\caption{The case of $k\ge 2$}
\label{fig:hl invariant}
\end{figure}
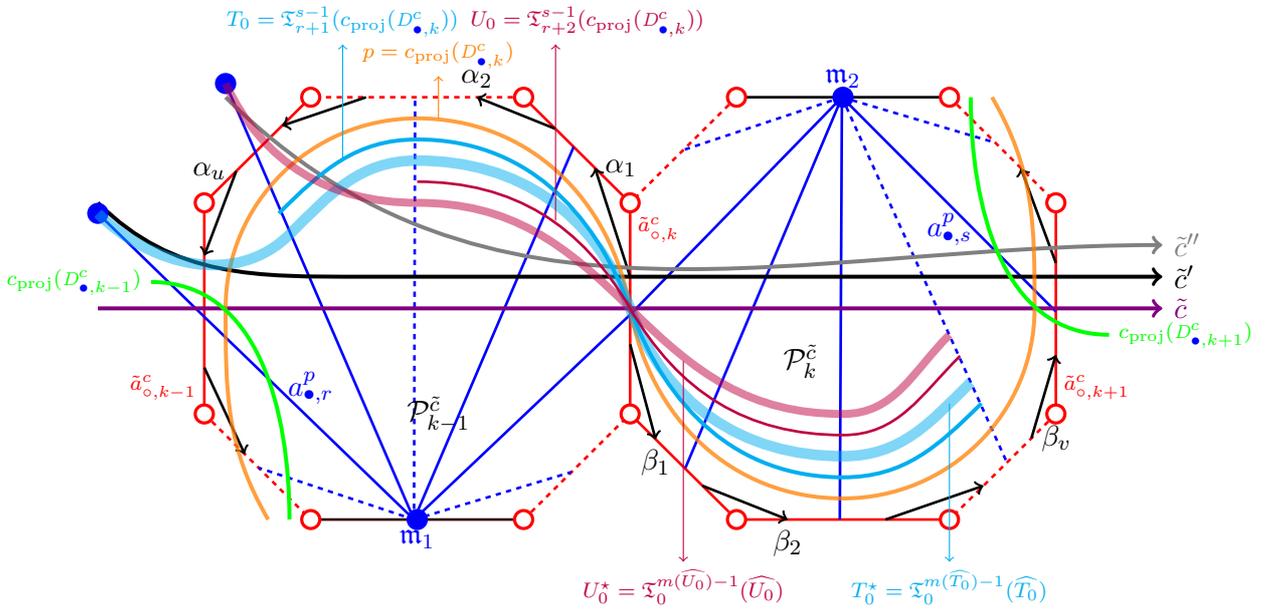
Now assume that $c^{*}$ is a curve with $c^{*} = \{c^{*}_{[j,j+1]}\}_{1\le j\le n(\tc^{*})}$ which satisfy
$c^{*}_{[j,j+1]}=c_{[j+k-2,j+k-1]}$ for all $1\le j\le n(\tc^{*})$ and $\varsigma^{\tc^{*}}_2=\varsigma^{\tc}_k=\varsigma$.
Let $\tc' =c^{*}_{[0,1]}\cup c^{*}$, where $c^{*}_{[0,1]}$ is the $\Dred$-arc segment lying in $\PP^{\tc}_{k-2}$.
We claim that $\tc'$ is the curve we need.

By Theorem \ref{thm:homologies}, we have
\[\H^{-\varsigma}(\X(\tc'))=\H^{-\varsigma}(\X(\tc))\] and
\[\dim_{\kk}\H^{-\varpi}(\X(\tc'))\le\dim_{\kk}\H^{-\varsigma}(\X(\tc))\ (\text{for any $\varpi\ne\varsigma, \varsigma^{\tc'}_1, \varsigma^{\tc'}_{n(\tc')}$}).\]
Thus, $\max\limits_{\varsigma^{\tc'}_{1}\ne t\ne \varsigma^{\tc'}_{n(\tc')}}\dim_{\kk}\H^{-t}(\X(\tc')) = \hl(\X(\tc))$.
Taking $\varsigma^{\tc'}_1=\varsigma'$, then $\tc'$ is an admissible curve in $\AC_{\m}(\SURF^{\F})$ such that
$\min\{1\le j\le n(\tc') \mid \varsigma^{\tc'}_j = \varsigma'\}=2$ and $\ared{\tc'}{1} = \ared{\tc}{k-1}$.
\end{proof}

\begin{proposition} \label{prop:stringcase}
Let $A$ be a gentle algebra and $\comp{X}$ a string complex in $\Dcat^b(A)$.
Then there exists a complex $\comp{Y}$ such that $\hl(\comp{X})-1=\hl(\comp{Y})$.
\end{proposition}

\begin{proof}
Let $\tc \in \AC_{\m}(\SURF(A)^{\F_A})$ be the admissible curve with $\X(\tc) \cong \comp{X}$. If $n(\tc)=1$, then $\X(\tc)$ is a stalk complex which can be viewed as a projective module over $A$. In this case, the result follows by the Brauer-Thrall type theorems for module categories.

Now assume that $n(\tc)\ge 2$. Keep the notation from Lemma \ref{lemm:hl invariant}. If $k\geq 2$, then there exists an admissible curve $\tc' \in \AC_{\m}(\SURF(A)^{\F_A})$ such that
\[\hl(\comp{X}) = \hl(\X(\tc)) = \max\limits_{\varsigma^{\tc'}_{1}\ne t\ne \varsigma^{\tc'}_{n(\tc')}}\dim_{\kk}\H^{-t}(\X(\tc')) = \hl(\X(\tc)).\]
Let $\ta_{\rbullet}$ be the graded $\Dred$-arc satisfying $\ta_{\rbullet}$ is a graded $\Dred$-arc adjacent to $\ared{\tc}{k-1}$ and $\ta_{\rbullet}$ is a predecessor of $\ared{\tc}{k-1}$ along the positive direction of the boundary of $\PP^{\tc}_{k-1}$, (see \Pic \ref{fig:hl invariant}, $\ta_{\rbullet}$ is the graded $\Dred$-arc corresponding to the source point of $\alpha_u$).

We define $\tc''$ to be the admissible curve in $\AC_{\m}(\SURF(A)^{\F_A})$ such that
\[ \ared{\tc''}{1} = \ta_{\rbullet}, \ared{\tc''}{t} =  \ared{\tc}{k+(t-2)}
\  \text{for any}~2\le t\le n(\tc)-k+2
\text{ and }
\varsigma^{\tc''}_2=\varsigma^{\tc'}_2. \]
Then
\[\H^{-\varsigma^{\tc''}_2}(\X(\tc'')) \cong \MM\left({{U_0}} \ocup \Ocup\limits_{1\le\jmath\le N} T_{\jmath}\right), \]
where ${{U_0}} = \tru^{s-1}_{r+2}(c_{\proj}(\dualgreen{c}{k}))$
and the truncations $T_1$, $T_2$, $\ldots$, $T_N$ is determined by the formula (\ref{tru}).
By Theorem \ref{thm:OPS and BCS corresponding} (1), we have
\begin{align}\label{1}
\dim_{\kk}\H^{-\varsigma^{\tc''}_2}(\X(\tc'')) = \dim_{\kk}\H^{-\varsigma^{\tc'}_2}(\X(\tc')) - 1 = \hl(\comp{X}) - 1.
\end{align}
By the same way, we can deal with $t_r^{\tc''}$ for any $1< r< n(\tc'')$ with $\dim_{\kk}\H^{-\varsigma_r^{-\tc''}}(\X(\tc'')) = \hl(\comp{X})$.
Indeed, we only need to observe the last one and obtain an admissible curve $\tc'''$ such that
\begin{align}\label{2}
\max\limits_{2\le i\le n(\tc''')-1}\dim_{\kk}\H^{-\varsigma_i^{\tc'''}}(\X(\tc'''))
= \dim_{\kk}\H^{-\varsigma^{\tc''}_2}(\X(\tc'')) \mathop{=\!=\!=}\limits^{(\ref{1})} \hl(\comp{X}) - 1.
\end{align}

Let $\omega=\omega_1\cdots\omega_l$ $(l+1=n(\tc'''))$ be the homotopy string corresponding to $\X(\tc''')$.
Assume that the lengths of $\omega_1$ and $\omega_l$ are $\ell_{\Left}$ and $\ell_{\Right}$, respectively. Without loss of generality, we can assume that $\omega_1\in\Q_{\ell_{\Left}}$ and $\omega_l\in\Q_{\ell_{\Right}}$. Then
\[\omega := \xymatrix{v_1 \ar@{~>}[r]^{\omega_1} & v_2 \ar@{~}[r] & \cdots \ar@{~}[r] & v_l \ar@{~>}[r]^{\omega_l} & v_{l+1}}.\]
Thus the $\theta=(-\varsigma^{\tc'''}_{1})$-th component of the complex $\X(\tc''')$ is of the form
\begin{align}\label{comp:Left}
\cdots \longrightarrow
   P_{\theta-1}=P(v_2) \oplus \hat{P}_{\theta-1} \mathop{-\!\!\!-\!\!\!-\!\!\!\longrightarrow}\limits^{d_{\theta-1}=\big( {{}_{\varphi}^{f}}\ {{}_{\phi}^{0}} \big)}
   P_\theta = P(v_1) \oplus \hat{P}_{\theta} \mathop{-\!\!\!-\!\!\!-\!\!\!\longrightarrow}\limits^{d_\theta=(0\ \ast)}
   P_{\theta+1}  \longrightarrow \cdots,
\end{align}
where $f$ is induced by $\omega_1$ and $\hat{P}_{\theta-1}$, $\hat{P}_{\theta}$ are projective. On the other hand, the $\vartheta=(-\varsigma^{\tc'''}_{l+1})$-th component of the complex $\X(\tc''')$ is of the form
\begin{align}\label{comp:Right}
\cdots \longrightarrow P_{\vartheta-1}
   \mathop{-\!\!\!-\!\!\!-\!\!\!\longrightarrow}\limits^{d_{\vartheta-1}=\big({}^{\star}_{0}\big)} P_{\vartheta} = \hat{P}_{\vartheta}\oplus P(v_{l+1})
   \mathop{-\!\!\!-\!\!\!-\!\!\!\longrightarrow}\limits^{d_{\vartheta}=\big({^{\zeta}_{\eta}}\ {^0_g}\big)} P_{\vartheta+1} = \hat{P}_{\vartheta+1}\oplus P(v_l)
   \longrightarrow \cdots,
\end{align}
where $g$ is induced by $\omega_l$ and $\hat{P}_{\vartheta}$, $\hat{P}_{\vartheta+1}$ are projective.
It is obvious that,
\[\H^{\theta}(\X(\tc''')) =  \ker d_{\theta} /\im d_{\theta-1} = (P(v_1)/\im f \oplus \ker(*)/\im(\varphi+\phi))\]
and
\[\H^{\vartheta}(\X(\tc''')) =  \ker d_{\vartheta} /\im d_{\vartheta-1} = \ker\zeta\oplus\ker(\im(\eta+g)/\im(\star)). \]
We define the complex $\comp{Y}=(Y_n, d^Y_n)$ as follows:

\[Y_n=\left\{\begin{array}{ll}
X_n&\mathrm{if}\ n\ne \theta+1, \vartheta-1,\\
\mathrm{coker}f \oplus P_{\theta+1}&\mathrm{if}\ n= \theta+1,\\
P_{\vartheta-1}\oplus\ker\zeta&\mathrm{if}\ n= \vartheta-1.
\end{array}\right.\]
where $(X_n, d^X_n)=\X(\tc''')$ and $d^Y_n$ is induced by $d^X_n$ naturally.
Then it is clear that $\comp{Y}$ is indecomposable. By the construction of $\tc'''$, we have
\[\dim_{\kk}\H^{-\varsigma^{\tc'''}_{1}}(\comp{Y}) < \dim_{\kk}\H^{-\varsigma^{\tc}_{k}}(\X(\tc)) \le \hl(\comp{X});\]
and
\[\dim_{\kk}\H^{-\varsigma^{\tc'''}_{l+1}}(\comp{Y}) < \dim_{\kk}\H^{-\varsigma^{\tc}_{k+l}}(\X(\tc)) \le \hl(\comp{X}).\]
Thus, by (\ref{2}), we obtain that $\hl(\comp{Y})=\hl(\comp{X})-1$.

If $k=1$, let $\ta_{\rbullet}$ be the graded $\Dred$-arc adjacent to the edge $b$ of $\PP^{\tc}_0$ lying in $\bSurf$ such that
$\ta_{\rbullet}$ is a predecessor of $b$ along the positive direction of the boundary of $\PP^{\tc}_0$.
Now consider the admissible curve $\tc'$ shown in \Pic \ref{fig:k=1}. In this case, $n(\tc')=n(\tc)+1$, $t^{\tc'}_1$ lies in $\ta_{\rbullet}$
and $\varsigma^{\tc'}_2=\varsigma^{\tc}_1$. Moreover, for any $1\le i\le n(\tc)$, we have $c'_{[i+1,i+2]} = c_{[i,i+1]}$. In particular, $c'_{[1,2]}$ is an $\Dred$-arc segment lying in $\PP_0^{\tc}$ of Case (2) shown in \Pic \ref{fig:arc segment II}.
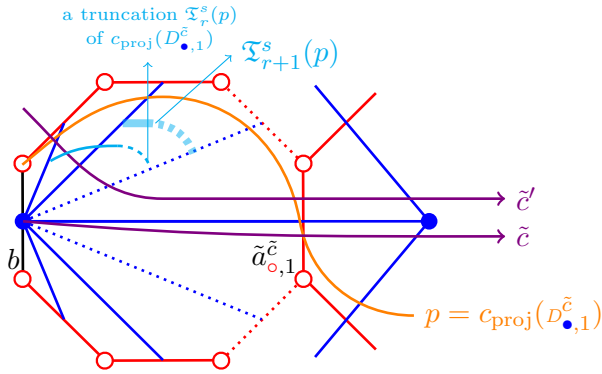
\begin{figure}[htbp]
\begin{center}
\definecolor{bluearc}{rgb}{0,0,1}
\begin{tikzpicture}
\draw (-1.75,-0.5) node[left]{$b$};
\draw[  red][rotate around={  0-22.5:(0,0)}] ( 2.00, 0.00) -- ( 1.41, 1.41) [line width=1pt];
     \draw (1.9,-0.5) node[left]{$\ared{\tc}{1}$};
\draw[  red][rotate around={ 45-22.5:(0,0)}] ( 2.00, 0.00) -- ( 1.41, 1.41) [line width=1pt][dotted];
\draw[  red][rotate around={ 90-22.5:(0,0)}] ( 2.00, 0.00) -- ( 1.41, 1.41) [line width=1pt];
\draw[  red][rotate around={135-22.5:(0,0)}] ( 2.00, 0.00) -- ( 1.41, 1.41) [line width=1pt];
\draw[black][rotate around={180-22.5:(0,0)}] ( 2.00, 0.00) -- ( 1.41, 1.41) [line width=1pt];
\draw[  red][rotate around={225-22.5:(0,0)}] ( 2.00, 0.00) -- ( 1.41, 1.41) [line width=1pt];
\draw[  red][rotate around={270-22.5:(0,0)}] ( 2.00, 0.00) -- ( 1.41, 1.41) [line width=1pt];
\draw[  red][rotate around={315-22.5:(0,0)}] ( 2.00, 0.00) -- ( 1.41, 1.41) [line width=1pt][dotted];
\draw[red][line width=1pt][shift={(1.8, 0.7)}] (0,0)--(1, 1);
\draw[red][line width=1pt][shift={(1.8,-0.7)}] (0,0)--(1,-1);
\fill[  red][rotate around={  0-22.5:(0,0)}] ( 2.00, 0.00) circle (0.12);
\fill[  red][rotate around={ 45-22.5:(0,0)}] ( 2.00, 0.00) circle (0.12);
\fill[  red][rotate around={ 90-22.5:(0,0)}] ( 2.00, 0.00) circle (0.12);
\fill[  red][rotate around={135-22.5:(0,0)}] ( 2.00, 0.00) circle (0.12);
\fill[  red][rotate around={180-22.5:(0,0)}] ( 2.00, 0.00) circle (0.12);
\fill[  red][rotate around={225-22.5:(0,0)}] ( 2.00, 0.00) circle (0.12);
\fill[  red][rotate around={270-22.5:(0,0)}] ( 2.00, 0.00) circle (0.12);
\fill[  red][rotate around={315-22.5:(0,0)}] ( 2.00, 0.00) circle (0.12);
\fill[white][rotate around={  0-22.5:(0,0)}] ( 2.00, 0.00) circle (0.09);
\fill[white][rotate around={ 45-22.5:(0,0)}] ( 2.00, 0.00) circle (0.09);
\fill[white][rotate around={ 90-22.5:(0,0)}] ( 2.00, 0.00) circle (0.09);
\fill[white][rotate around={135-22.5:(0,0)}] ( 2.00, 0.00) circle (0.09);
\fill[white][rotate around={180-22.5:(0,0)}] ( 2.00, 0.00) circle (0.09);
\fill[white][rotate around={225-22.5:(0,0)}] ( 2.00, 0.00) circle (0.09);
\fill[white][rotate around={270-22.5:(0,0)}] ( 2.00, 0.00) circle (0.09);
\fill[white][rotate around={315-22.5:(0,0)}] ( 2.00, 0.00) circle (0.09);
\fill[bluearc][rotate around={0:(0,0)}] (-1.84, 0.00) circle (0.11);
\draw[bluearc][line width=1pt] (-1.84, 0.00) to (  0:1.84);
\draw[bluearc][line width=1pt] (-1.84, 0.00) to ( 45:1.84)[dotted];
\draw[bluearc][line width=1pt] (-1.84, 0.00) to ( 90:1.84);
\draw[bluearc][line width=1pt] (-1.84, 0.00) to (135:1.84);
\draw[bluearc][line width=1pt] (-1.84, 0.00) to (180:1.84);
\draw[bluearc][line width=1pt] (-1.84, 0.00) to (225:1.84);
\draw[bluearc][line width=1pt] (-1.84, 0.00) to (270:1.84);
\draw[bluearc][line width=1pt] (-1.84, 0.00) to (315:1.84)[dotted];
\draw[bluearc][line width=1pt] (-1.84, 0.00) to (  0: 3.5);
\fill[bluearc] (3.5, 0) circle (0.11);
\draw[bluearc][line width=1pt] (2., 1.8)--(3.5, 0);
\draw[bluearc][line width=1pt] (2.,-1.8)--(3.5, 0);
\draw[orange][line width=1pt] (-1.85,0.75) to[out=40,in=180] (0,1.65) to[out=0,in=100] (1.8,0)
     to[out=-80,in=180] (3.3,-1.25) node[right]{$p=c_{\proj}(\dualgreen{\tc}{1})$};
\draw[violet][line width=1pt][->] (-1.84, 0.00) to[out= -5,in=180] (4.5,-0.2) node[right]{$\tc$};
\draw[violet][line width=1pt][->] (-1.84, 1.50) to[out=-45,in=180] (0, 0.30) -- (4.5, 0.30) node[right]{$\tc'$};
\draw[cyan][line width=1pt] (-1.48, 0.8 ) to[out= 30,in=180] (-0.6 , 1   );
\draw[cyan][line width=1pt] (-0.6 , 1   ) to[out=  0,in=110] (-0.17, 0.72) [dash pattern=on 2pt off 2pt];
\draw[cyan][->] (-0.2, 0.77) -- (-0.2, 2.11);
\draw[cyan] (-0.17, 2.46) node[above]{\tiny a truncation $\tru_r^s(p)$};
\draw[cyan] (-0.17, 2.11) node[above]{\tiny of $c_{\proj}(\dualgreen{\tc}{1})$};
\draw[cyan][opacity=0.5][line width=3pt] (-0.53, 1.3 ) -- (-0.2 , 1.3 );
\draw[cyan][opacity=0.5][line width=3pt] (-0.2 , 1.3 ) to[out=  0,in=110] ( 0.40, 0.90) [dash pattern=on 2pt off 2pt];
\draw[cyan][->] (-0.1,1.3) -- (0.9,2.2) node[right]{$\tru_{r+1}^{s}(p)$};
\end{tikzpicture}
\caption{The case of $k=1$}
\label{fig:k=1}
\end{center}
\end{figure}
Thus, the homotopy string $\varpi = \varpi_1\cdots\varpi_{\ell}$ corresponding to $\X(\tc')$ is of the form
\[ \xymatrix{v_1 \ar@{~>}[r]^{\varpi_1} & v_2 \ar@{~}[r] & \cdots \ar@{~}[r] & v_{\ell+1}. } \]
Then the $\kappa=-\varsigma^{\tc'}_1$-th component of $\X(\tc')$ is of the form
\[\label{comp:Left}
\cdots \longrightarrow
   P_{\kappa-1}=P(v_2) \oplus \hat{P}_{\kappa-1} \mathop{-\!\!\!-\!\!\!-\!\!\!\longrightarrow}\limits^{d_{\kappa-1}=\big( {{}_{\varphi}^{f}}\ {{}_{\phi}^{0}} \big)}
   P_\kappa = P(v_1) \oplus \hat{P}_{\kappa} \mathop{-\!\!\!-\!\!\!-\!\!\!\longrightarrow}\limits^{d_\kappa=(0\ \ast)}
   P_{\kappa+1}  \longrightarrow \cdots.
\]
By the same way, we can construct a complex $\comp{Y}$ with $\hl(\comp{Y})=\hl(\comp{X})-1$.
\end{proof}

\begin{lemma} \label{lemm:bandcase}
Let $A$ be a gentle algebra has band complex. Then, for any $l\in \NN$, there exists a string complex $\comp{X}$ such that $\hl(\comp{X})\ge l$.
\end{lemma}

\begin{proof}
By Theorem \ref{thm:OPS and BCS corresponding} (2), $\AC^{\oslash}_{\oslash}(\SURF(A)^{\F_A}) \ne \varnothing$,
thus there is a band $c=\{c_{[i,i+1]}\}_{0\le i\le n-1}$ ($n>1$) such that
$c_{[\overline{i},\overline{i+1}]}(1)=c_{[\overline{i+1},\overline{i+2}]}(0)$, where $\overline{x}\equiv x(\bmod n)$.
Let $c_{[n,n+1]}$ be an $\Dred$-arc segment of Case (2) in \Pic \ref{fig:arc segment II} with $c_{[n,n+1]} \neq c_{[n-1,n]}$ and $c_{[n,n+1]} \neq c_{[n-1,n]}^{-1}$. Then,  $c\cup\{c_{[n,n+1]}\}=\{c_{[i,i+1]}\}_{0\le i\le n}$ is a truncation of some admissible curve $c'$ in $\AC_{\m}(\SURF(A)^{\F_A})$ with $c_{[0,1]}(1)=c'(t_2)$. Moreover, for any grading $\tc'$ of $c'$, we have $\varsigma^{\tc'}_{2} = \varsigma^{\tc'}_{n+2} =: \varsigma$ (see \Pic \ref{fig:band}).

\begin{figure}[htbp]
\centering
\begin{tikzpicture} [scale=0.85]
\draw[red][line width=1pt] (0,0)--(0,4) node[above]{$\ared{}{1}$};
\draw[red][line width=1pt] (0,0)--(-4*0.707,4*0.707) node[above right]{$\ared{}{n-1}$};
\draw[red][line width=1pt] (0,0)--(-4,0);
\draw[red][line width=1pt] (0,0)--(-4*0.707,-4*0.707) [dash pattern=on 2pt off 2pt];
\draw[red][line width=1pt] (0,0)--(0,-4) [dash pattern=on 2pt off 2pt];
\draw[red][line width=1pt] (0,0)--(4*0.707,-4*0.707) [dash pattern=on 2pt off 2pt];
\draw[red][line width=1pt] (0,0)--(4,0) node[right]{$\ared{}{3}$};
\draw[red][line width=1pt] (0,0)--(4*0.707,4*0.707)  node[above right]{$\ared{}{2}$};
\fill[white] (0,0) circle (1);
\fill[gray] (0,0) circle (0.5);
\draw[black][line width=1pt][dash pattern=on 2pt off 2pt] (0,0) circle (0.5);
\draw[black] (0.5,0) node[right]{$b_0$};
\draw[violet][line width=1.5pt] (0,0) circle (1.5);
\draw[violet][line width=1.5pt][->] (-1.5*0.707, 1.5*0.707) to[out=45, in=135] (1.5*0.707,1.5*0.707);
\draw[violet] (1.5,0) node[above left]{$\tc$};
\draw[violet] (-1.48, 0.61) node{$c_{[n-2,n-1]}$};
\draw[violet] (-0.61, 1.48) node[above]{$c_{[n-1,n]}$};
\draw[violet] ( 0.61, 1.48) node[above]{$c_{[0,1]}$};
\draw[violet] ( 1.48, 0.61) node[right]{$c_{[1,2]}$};
\draw[black][line width=1.5pt]
     (0,3) node[below left]{$c(t^{\tc'}_1)$} to[out=0,in=90] (3,0) to[out=-90,in=0]
     (0,-3) to[out=180,in=-90] (-3,0) to[out=90,in=180] (0,3.5) node[above left]{$c'(t^{\tc'}_{n+1})$} to[out=0,in=90]
     (3.5,0) node[below]{$\tc'$} [->];
\draw [cyan] (3*0.707,3*0.707) node{$\bullet$} node[below]{$\varsigma_2^{\tc'}$};
\draw [cyan] (3.5*0.707,3.5*0.707) node{$\bullet$} node[right]{$\varsigma_{n+2}^{\tc'}$};
\end{tikzpicture}
\caption{The admissible curve $b$ corresponded by band complex.}
\label{fig:band}
\end{figure}
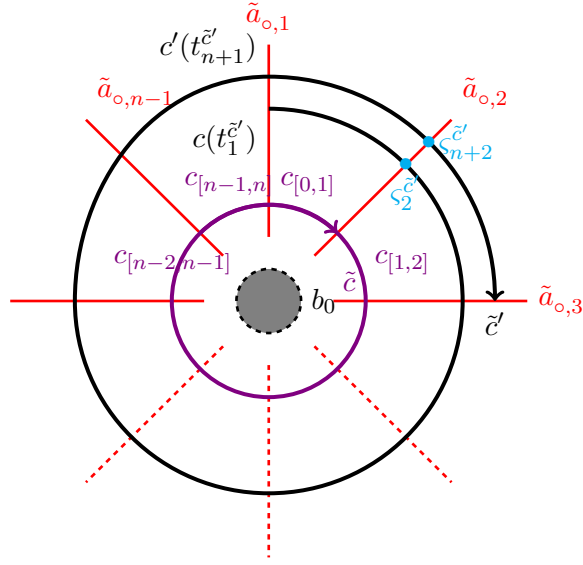
Notice that any band complex is non-exact. Thus, without loss of generality, we can assume that $\dim\H^{-\varsigma}(\X(\tc))\ne 0$. Then, by the intersection index of $c'(t^{\tc}_2)$ and $c'(t^{\tc}_{n+2})$, we obtain that $\dim\H^{-\varsigma}(\X(\tc'))\ne 0$.

Furthermore, assume that $c'\in \AC_{\m}(\SURF(A)^{\F_A})$ is an admissible curve in surrounding some boundary components of $\Surf(A)$ such that $c'$ passes through the $\Dred$-arc $\ared{}{2} (=\ared{c'}{2})$ $t$ times. Then
\[\varsigma^{\tc'}_{2} = \varsigma^{\tc'}_{n+2} = \varsigma^{\tc'}_{2n+2} = \cdots = \varsigma^{\tc'}_{tn+2} =: \varsigma. \]
By hypothesis, there exits an integer $m\in\NN^+$ such that $\dim_{\kk}\H^{-\varsigma}(\X(\tc')) \ge mt$.
Let $t = \lfloor l/m\rfloor+1$. Then we obtain $\dim_{\kk}\H^{-\varsigma}(\X(\tc')) \ge l$,
where $\X(\tc')$ is a string complex given by Theorem \ref{thm:OPS and BCS corresponding} (2).
\end{proof}

\begin{proposition} \label{prop:bandcase}
Let $A$ be a gentle algebra and $\comp{X}$ a band complex in $\Dcat^b(A)$.
Then there exists a complex $\comp{Y}$ such that $\hl(\comp{X})-1=\hl(\comp{Y})$.
\end{proposition}

\begin{proof}
Assume that $\hl(\comp{X})=l$. Then by Lemma \ref{lemm:bandcase}, there exists a string complex $\comp{Y}_1$ such that $\hl(\comp{Y}_1)\ge l$. Let $\hl(\comp{Y}_1)=L$.
By Proposition \ref{prop:stringcase}, there exists a string complex $\comp{Y}_2$ with $\hl(\comp{Y}_2)=L-1$.
Now again by Proposition \ref{prop:stringcase}, we obtain a family of string complexes $\comp{Y}_1$, $\comp{Y}_2$, $\ldots$, $\comp{Y}_L$ such that $\hl(\comp{Y}_i) = L-i+1$, $1\le i\le L$.
Then $\comp{Y}_{L-l+2}$ is the string complex with $\hl(\comp{Y}_{L-l+2})=L-(L-l+2)+1=l-1$.
\end{proof}

Now we can give the main result in this subsection.

\begin{theorem}
Let $A$ be a gentle algebra and $\comp{X}$ any indecomposable complex in $\Dcat^b(A)$ with cohomological length $l=\hl(\comp{X})$. 
Then there exists an indecomposable complex $\comp{Y}$ in $\Dcat^b(A)$ such that $\hl(\comp{Y})=l-1$.
\end{theorem}

\begin{proof}
It is obvious by Propositions \ref{prop:stringcase} and \ref{prop:bandcase}.
\end{proof}

\subsection{Strong Nakayama conjecture over gentle algebras}
In this subsection, we use the embedding to give a geometric proof for the Strong Nakayama conjecture for gentle algebras.
Let $\Lambda$ be an Artinian ring. Recall that the Strong Nakayama conjecture (=SNC) claim that
\begin{center}
	$\Ext_\Lambda^{\ge 0}(M, \Lambda)=0$ if and only if $M=0$.
\end{center}
In this subsection, we give a directly proof of the SNC for gentle algebras by using the embedding of geometric models.

\begin{lemma} \label{lemm:SNC-band case}
	Let $A$ be a gentle algebra and $B(n,\lambda\ne 0)\in \modcat A$ be a band module.
	Then $\Ext_A^1(B(n,\lambda), A)\ne 0$.
\end{lemma}

\begin{proof}
	We have a short exact sequence for band module $B=B(n,\lambda)$ as follows
	\[ 0 \longrightarrow P_1^n \longrightarrow P_0^n \longrightarrow B \longrightarrow 0. \]
	Since $\pdim B(n,\lambda) = 1$, we obtain a long exact sequence
	\begin{align}
		0 \longrightarrow
		& \Hom_A(B, P_1^n) \longrightarrow \Hom_A(B, P_0^n) \longrightarrow \Hom_A(B, B) \nonumber \\
		\longrightarrow
		& \Ext_A^1(B, P_1^n) \longrightarrow \Ext_A^1(B, P_0^n) \longrightarrow  \Ext_A^1(B, B) \longrightarrow 0 \nonumber
	\end{align}
	If $\Ext_A^1(B, A)=0$, then so is $\Ext_A^1(B, P_0^n)$. Thus,  $\Ext_A^1(B, B)$=0,
	this is a contradiction because any band module has non-trivial positive self-extension.
\end{proof}

\begin{lemma}\label{lemm:SNC-string case}
	Let $A$ be a gentle algebra and $M \in \modcat A$ be a string module.
	Then there is an integer $d\ge 0$ such that $\Ext_A^d(M, A)\ne 0$.
\end{lemma}

\begin{proof}
	Let $c\in \PC_{\m}(\SURF(A)^{\F_A})$ be a permissible curve and $M\cong \MM(c)$.
	By Theorem \ref{thm:string embedding}, we have
	\[\Ext_A^d(M, A) \cong \Hom_{\per A} (\X(\tc^{\rota}), A[d]). \]
	Next, we need to show that $\Hom_{\per A} (\X(\tc^{\rota}), A[d])\ne 0$ for some $d\ge 1$. First, if $n(\tc^{\rota})=1$, that is, $c^{\rota}$ is a concatenation of two $\Dred$-arc segments, then $\X(\tc^{\rota})$ is quasi-isomorphic to the stalk complex $eA$ for some primitive idempotent of $A$. In this case, we have $\Ext_A^1(eA, A[1])\ne 0$.
	
	Now if $c^{\rota}$ contains at least one $\Dred$-arc segment $c_{[i,i+1]}$ ($1\leq i\leq n(\tc^{\rota}){-}1$)
	with $\ii_{c(t_{i+1})}(\tc^{\rota}, \ared{c}{i+1})=0$
	(see \Pic \ref{fig:pf of lemm:SNC-string case}).
	\begin{figure}[htbp]
		\begin{center}
			\definecolor{bluearc}{rgb}{0,0,1}
			\begin{tikzpicture}
				\draw[  red][rotate around={  0-22.5:(0,0)}] ( 2.00, 0.00) -- ( 1.41, 1.41) [line width=1pt];
				\draw[  red][rotate around={ 45-22.5:(0,0)}] ( 2.00, 0.00) -- ( 1.41, 1.41) [line width=1pt][dotted];
				\draw[  red][rotate around={ 90-22.5:(0,0)}] ( 2.00, 0.00) -- ( 1.41, 1.41) [line width=1pt];
				\draw[  red][rotate around={135-22.5:(0,0)}] ( 2.00, 0.00) -- ( 1.41, 1.41) [line width=1pt];
				\draw[  red][rotate around={180-22.5:(0,0)}] ( 2.00, 0.00) -- ( 1.41, 1.41) [line width=1pt];
				\draw[  red][rotate around={225-22.5:(0,0)}] ( 2.00, 0.00) -- ( 1.41, 1.41) [line width=1pt][dotted];
				\draw[black][rotate around={270-22.5:(0,0)}] ( 2.00, 0.00) -- ( 1.41, 1.41) [line width=1pt];
				\draw[  red][rotate around={315-22.5:(0,0)}] ( 2.00, 0.00) -- ( 1.41, 1.41) [line width=1pt][dotted];
				\fill[  red][rotate around={  0-22.5:(0,0)}] ( 2.00, 0.00) circle (0.12);
				\fill[  red][rotate around={ 45-22.5:(0,0)}] ( 2.00, 0.00) circle (0.12);
				\fill[  red][rotate around={ 90-22.5:(0,0)}] ( 2.00, 0.00) circle (0.12);
				\fill[  red][rotate around={135-22.5:(0,0)}] ( 2.00, 0.00) circle (0.12);
				\fill[  red][rotate around={180-22.5:(0,0)}] ( 2.00, 0.00) circle (0.12);
				\fill[  red][rotate around={180-22.5:(0,0)}] ( 2.00, 0.00) node[left]{$q$};
				\fill[  red][rotate around={225-22.5:(0,0)}] ( 2.00, 0.00) circle (0.12);
				\fill[  red][rotate around={270-22.5:(0,0)}] ( 2.00, 0.00) circle (0.12);
				\fill[  red][rotate around={315-22.5:(0,0)}] ( 2.00, 0.00) circle (0.12);
				\fill[white][rotate around={  0-22.5:(0,0)}] ( 2.00, 0.00) circle (0.09);
				\fill[white][rotate around={ 45-22.5:(0,0)}] ( 2.00, 0.00) circle (0.09);
				\fill[white][rotate around={ 90-22.5:(0,0)}] ( 2.00, 0.00) circle (0.09);
				\fill[white][rotate around={135-22.5:(0,0)}] ( 2.00, 0.00) circle (0.09);
				\fill[white][rotate around={180-22.5:(0,0)}] ( 2.00, 0.00) circle (0.09);
				\fill[white][rotate around={225-22.5:(0,0)}] ( 2.00, 0.00) circle (0.09);
				\fill[white][rotate around={270-22.5:(0,0)}] ( 2.00, 0.00) circle (0.09);
				\fill[white][rotate around={315-22.5:(0,0)}] ( 2.00, 0.00) circle (0.09);
				\fill[rotate around={  0:(0,0)}] ( 1.8,-0.4) node[right]{$\ared{\tc^{\rota}}{i+1}$};
				\fill[rotate around={180:(0,0)}] ( 1.8, 0.4) node[ left]{$\ared{\tc^{\rota}}{i} = \ta_{\rbullet}^{t+1}$};
				\fill[rotate around={135:(0,0)}] ( 1.9, 0.2) node[ left]{$\ta_{\rbullet}^t$};
				\fill[rotate around={ 90:(0,0)}] ( 1.8,-0.4) node[above]{$\ta_{\rbullet}^{t-1}$};
				\fill[bluearc][rotate around={270:(0,0)}] ( 1.84, 0.00) circle (0.11);
				\draw[bluearc][line width=1pt] ( 0.00,-1.84) to (  0:1.84);
				\draw[bluearc][line width=1pt] ( 0.00,-1.84) to ( 45:1.84) [dotted];
				\draw[bluearc][line width=1pt] ( 0.00,-1.84) to ( 90:1.84);
				\draw[bluearc][line width=1pt] ( 0.00,-1.84) to (135:1.84);
				\draw[bluearc][line width=1pt] ( 0.00,-1.84) to (180:1.84);
				\draw[bluearc][line width=1pt] ( 0.00,-1.84) to (225:1.84) [dotted];
				\draw[bluearc][line width=1pt] ( 0.00,-1.84) to (315:1.84) [dotted];
				\draw (135:1.84) node{$\bullet$}; \draw (135:1.84) node[above]{$r$};
				\draw (-1.84,0.5) node{$\bullet$}; \draw (-1.84,0.5) node[left]{$s$};
				\draw[violet][->] (1.88,0.5) to[out=0,in=-90] (2.5,1) ;
				\draw[violet] (2.5,1) node[above]{$\ii_{c(t_{i+1})}(\tc^{\rota}, \ared{c}{i+1})=0$};
				\draw[violet][line width=1pt] (-1.84,0.5) -- (1.84,0.5);
				\draw[violet]( 1.60, 0.5 ) node[below]{$c^{\rota}$};
				\draw[black] (-0.97, 0.49) node{$\bullet$};
				\draw[black] (-0.97, 0.49) node[below]{$p$};
				\draw[black] (0.3,0) node{$\PP$};
			\end{tikzpicture}
			\caption{ }
			\label{fig:pf of lemm:SNC-string case}
		\end{center}
	\end{figure}
	Denote by $a_{\gbullet}^i$ the dual $\gbullet$-arc of $a_{\rbullet}^i$. Note that $\X(\ta_{\gbullet}^i)$ is quasi-isomorphic to the
	stalk complex of $e_i A$, where $e_i$ is  the primitive idempotent of $A$ corresponding to the vertex $\mathfrak{v}(a_{\gbullet}^i)$ of $\Q$. Then, we have
	\begin{itemize}
		\item $\ii_q(\ta_{\rbullet}^t, \ta_{\rbullet}^{t+1}) = 1$
		because the grading of $\alpha: \mathfrak{v}(a_{\gbullet}^t) \to \mathfrak{v}(a_{\gbullet}^{t+1})$ is zero;
		\item $\ii_{a_{\rbullet}\cap a_{\gbullet}}(\ta_{\gbullet}^i, \ta_{\rbullet}^i)=0$, for all $1\le i\le N$,
		where $N$ is the number of elements in $\Ered(\PP)$
		and $\ii_{a_{\rbullet}\cap a_{\gbullet}}(\ta_{\rbullet}^i, \ta_{\gbullet}^i)=1$;
		\item $\ii_{s}(\tc^{\rota}, \ta_{\rbullet}^{t+1})=1$
		because $\X(\tc^{\rota})$ is the complex in $\per A$ corresponding to
		the projective resolution of $\MM(c)$.
	\end{itemize}
	Furthermore, we obtain
	\begin{align}
		\ii_p(\tc^{\rota}, \ta_{\gbullet}^t)
		& = \ii_s(\tc^{\rota}, \ta_{\rbullet}^{t+1})
		\cdot \big(\ii_q(\ta_{\rbullet}^t, \ta_{\rbullet}^{t+1})\big)^{-1}
		\cdot \ii_r(\ta_{\rbullet}^t, \ta_{\gbullet}^t) \nonumber \\
		& = 1+(-1)+1 = 1 > 0.
	\end{align}
	Then $\Ext_A^1 (\MM(c), e_tA) \cong \Hom_{\per A}(\X(\tc^{\rota}), \X(\ta_{\gbullet}^t)[1])\ne 0$, and so $\Ext_A^1 (\MM(c), A) \ne 0$.
\end{proof}

Now we can prove SNC for gentle algebras.

\begin{theorem}
	Let $A$ be a gentle algebra and $M$ be an  $A$-module.
	Then $\Ext_A^{\ge 0} (M, A)=0$ if and only if $M=0$.
\end{theorem}

\begin{proof}
	If $M=0$, then it is trivial. Now let $M = \bigoplus_{i=1}^nM_i$, where all $M_i$ are indecomposable. Then
	\[\Ext_A^{\ge 0}(M, A) \cong \bigoplus_{i=1}^n\Ext_A^{\ge 0}(M_i, A), \]
	where $M_i$ is either a string module or a band module for any $1\le i\le n$.
	By Lemmas \ref{lemm:SNC-band case} and \ref{lemm:SNC-string case}, there is an integer $d\ge 0$ such that $\Ext_A^d(M_i, A)\ne 0$. We obtain a contradiction.
\end{proof}

\begin{remark}\rm
	In \cite{GR2005}, Gei\ss\ and Reiten provided a method to calculate the self-injective dimension of a gentle algebra.
	Furthermore, we can show that the finite dimension conjecture is true for gentle algebras.
	Note that this yields that SNC holds for gentle algebras.
\end{remark}

\section{Two examples} \label{sec: examples}
\begin{example} \rm \label{exp:homologies}
Let $A=\kk\Q/\I$ be the algebra given by the quiver $\Q=$
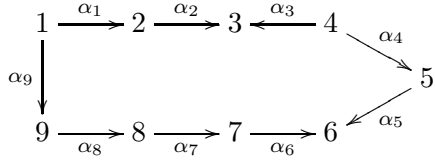
\[\xymatrix@R=0.25cm{
1\ar[r]^{\alpha_1}\ar[dd]_{\alpha_9} & 2\ar[r]^{\alpha_2} & 3  & 4 \ar[l]_{\alpha_3} \ar[rd]^{\alpha_4} & \\
& & & & 5 \ar[ld]^{\alpha_5} \\
9\ar[r]_{\alpha_8} & 8 \ar[r]_{\alpha_7} & 7 \ar[r]_{\alpha_6} & 6 &
}\]
with $\I=\langle \alpha_1\alpha_2, \alpha_4\alpha_5, \alpha_9\alpha_8, \alpha_8\alpha_7 \rangle.$ It is obvious that $A$ is a gentle algebra. The marked ribbon surface $\mathbf{S}^{\mathcal{F}}(A)$ of $A$ is shown in \Pic \ref{fig:homologies}.

(1) We take an admissible curve $\tc$. Then we have $n(\tc)=10$ and $\ared{\tc}{1}$ corresponds to the vertex $2\in \Q_0$. Assume that $\varsigma^{\tc}_1=\ii_{c(t^{\tc}_1)}(\tc, \ared{\tc}{1})=1$. Now we first calculate $\H^{-1}(\X(\tc))$.
\begin{figure}[htbp]
\centering
\begin{tikzpicture}[scale=2]
\draw[line width=1pt] (0,0) circle(2);
\fill[gray!50] (0,0) circle(0.5);
\draw[line width=1pt] (0,0) circle(0.5);
\fill[blue][line width=1pt] ( 0.  , 2.  ) circle(0.08);
\fill[blue][line width=1pt] ( 0.  , 0.5 ) circle(0.08);
\fill[blue][line width=1pt] ( 1.41, 1.41) circle(0.08);
\fill[blue][line width=1pt] ( 0.35, 0.35) circle(0.08);
\fill[blue][line width=1pt] ( 2.00, 0.00) circle(0.08);
\fill[blue][line width=1pt] ( 1.41,-1.41) circle(0.08);
\fill[blue][line width=1pt] ( 0.35,-0.35) circle(0.08);
\fill[blue][line width=1pt] ( 1. ,-1.732) circle(0.08);
\fill[blue][line width=1pt] (-0.35,-0.35) circle(0.08);
\draw[blue][dash pattern=on 2pt off 2pt] ( 0.  , 0.5 )--( 0.  , 2.  )--( 1.41, 1.41)--( 0.35, 0.35)--( 2.00, 0.  )
      --( 1.41,-1.41)--( 0.35,-0.35)--( 1. ,-1.732);
\draw[blue][dash pattern=on 2pt off 2pt] ( 0.35,-0.35) to[out=-80,in=0] (0,-1) to[out=180,in=-100] (-0.35,-0.35)
     to[out=170,in=-90] (-1.3,0.4) to[out=90,in=135] (0,0.5);
\fill[red] (-2.  , 0.  ) circle(0.08); \fill[white] (-2.  , 0.  ) circle(0.06);
\fill[red] ( 0.19, 0.46) circle(0.08); \fill[white] ( 0.19, 0.46) circle(0.06);
\fill[red] ( 0.76, 1.85) circle(0.08); \fill[white] ( 0.76, 1.85) circle(0.06);
\fill[red] ( 1.85, 0.76) circle(0.08); \fill[white] ( 1.85, 0.76) circle(0.06);
\fill[red] ( 0.5 , 0.  ) circle(0.08); \fill[white] ( 0.5 , 0.  ) circle(0.06);
\fill[red] ( 1.85,-0.76) circle(0.08); \fill[white] ( 1.85,-0.76) circle(0.06);
\fill[red] ( 1.23,-1.57) circle(0.08); \fill[white] ( 1.23,-1.57) circle(0.06);
\fill[red] ( 0.  ,-0.5 ) circle(0.08); \fill[white] ( 0.  ,-0.5 ) circle(0.06);
\fill[red] (-0.5 , 0.  ) circle(0.08); \fill[white] (-0.5 , 0.  ) circle(0.06);
\draw[red][dash pattern=on 2pt off 2pt] (-2.  , 0.  ) to[out=  85,in= 100] ( 0.19, 0.46) -- ( 0.76, 1.85);
\draw[red][dash pattern=on 2pt off 2pt] ( 0.19, 0.46) to[out=  50,in= 150] ( 1.85, 0.76);
\draw[red][dash pattern=on 2pt off 2pt] ( 0.5 , 0.  ) to[out=   0,in= 230] ( 1.85, 0.76);
\draw[red][dash pattern=on 2pt off 2pt] ( 0.5 , 0.  ) -- ( 1.85,-0.76);
\draw[red][dash pattern=on 2pt off 2pt] ( 0.5 , 0.  ) -- ( 1.23,-1.57) to[out= 200,in=  -80] (-2.  , 0.  ) -- (-0.5 , 0.  );
\draw[red][dash pattern=on 2pt off 2pt] (-2.  , 0.  ) to[out= -65,in= -90] ( 0.  ,-0.5 );
{\tiny
\draw[blue] ( 0.  , 0.85) node[right]{1} ( 0.8 , 1.65) node[below]{2} ( 1.06, 1.06) node[left]{3}
      ( 1.6 , 0.1 ) node[above]{4} ( 1.74,-0.56) node[ left]{5} ( 1.  ,-1.  ) node[right]{6}
      ( 0.7 ,-1.11) node[below]{7} ( 0.  ,-0.99) node[above]{8} (-1.3 , 0.4 ) node[ left]{9};}
\draw[violet][line width=1.5pt] ( 0.  , 0.5 ) to[out=  45, in=  90] ( 1.25, 0.  ) node[left]{$\tc$}
     to[out= -90, in=   0] ( 0.  ,-1.25) to[out= 180, in= -90] (-1.25, 0.  )
     to[out=  90, in= 180] ( 0.  , 1.25) to[out=   0, in= 195] ( 1.41, 1.41);
\draw[violet][line width=1.5pt]  ( 0.  ,-1.25) to[out= 180, in= -90] (-1.25, 0.  ) [<-];
{\tiny
\draw[black][line width=4pt][opacity=0.25]
     (0,2) to[out=-70,in=135] (0.85,0.77) to[out=-45,in=105]
     (1.42,0.00) to[out=-75,in=75](1.414,-1.414);
     \draw[->] (1.44,-1.2) -- (1,-1.2) -- (0.7,-1.5) node[left]{$\tru_1^1(p_4)\ocup\tru_2^2(p_4)$};
\draw[black][line width=4pt][opacity=0.25] (0,2) to[out=-40,in=150] ( 1.0, 1.3) to[out=-30,in=100](2,0); 
\draw[orange][line width=1.5pt] ( 0.76, 1.85)--( 2.  , 0.  ) ( 1.47, 0.77) node[left]{$p_2$};
\draw[green][line width=1.5pt]  ( 0.76, 1.85)to[out=180,in=90](-0.5 , 0.  ) (-0.12,1.45) node[left]{$p_1$};
\draw[cyan][line width=2pt][opacity=0.75] ( 1.31, 1.31)to[out=-45,in=110]( 2.  , 0.  );
\draw[cyan][->] ( 1.41, 1.21) -- ( 1.8 , 1.21) node[right]{$\tru_2^3(p_2)$};
\draw[orange][line width=1.5pt] ( 0.  , 2.  ) to[out=-80,in=100] ( 1.85,-0.76) ( 1.65,-0.26) node[right]{$p_4$};
\draw[green][line width=1.5pt] ( 0.  , 2.  ) to[out=-75,in=135] ( 2.  ,-0.  );
\draw[green][line width=1.5pt] ( 1.85,-0.76) to[out=220,in=65] ( 1. ,-1.732);
\draw[cyan][line width=2pt][opacity=0.75] (1.5,0.14) -- (1.62,0.00);
\draw[cyan] (1.58,0.02) node[below right]{$\tru_2^2(p_4)$};
\draw[cyan][line width=2pt][opacity=0.75] (0.68,0.77) -- (0.78,0.67);
\draw[cyan][->] (0.78,0.67) -- (1.98,0.67) node[right]{$\tru_1^1(p_4)$};
}
\end{tikzpicture}
\caption{The marked ribbon surface of the gentle algebra $A$
($1$, $2$, $\ldots$, $9$ are $\Dblue$-arcs corresponded by vertices of $\Q_0$)}
\label{fig:homologies}
\end{figure}
Notice that the indexes of intersections $c(t^{\tc}_1)$, $c(t^{\tc}_3)$, $c(t^{\tc}_7)$ and $c(t^{\tc}_9)$ given by $\tc$ intersecting with $\rbullet$-FFAS are $1$. Denote by $p_i$ the projective permissible curve corresponding to the indecomposable projective module $P(i)$. First, by Propositions \ref{prop:HomologiesCaseI} and \ref{prop:HomologiesCaseII}, we calculate the truncations of projective curves corresponding to these intersections.
\begin{itemize}
    \item For the intersection $c(t^{\tc}_1)$, the $1$-st cohomological curve is the truncation ${\color{cyan}\tru_2^3(p_2)}$ of ${\color{orange}p_2}$ which is induced by ${\color{green}p_1}$.
    \item For the intersection $c(t^{\tc}_3)$, the $3$-rd cohomological curve is trivial.
    \item For the intersection $c(t^{\tc}_7)$, the $7$-th cohomological curve is ${\color{cyan}\tru_2^2(p_4)}$ of ${\color{orange}p_4}$ which is induced by ${\color{green}p_5}$ and ${\color{green}p_3}$.
    \item For the intersection $c(t^{\tc}_9)$, the $9$-th cohomological curve is ${\color{cyan}\tru_1^1(p_4)}$.
 \end{itemize}
Then, by Theorem \ref{thm:homologies}, we have
\begin{align}
  \H^{-1}(\X(\tc)) & \cong \MM(\tru_2^3(p_2)\ocup\tru_2^2(p_4)\ocup\tru_1^1(p_4))
= \MM(\widehat{\tru_2^3(p_2)}\cup \widehat{\tru_2^2(p_4)\cup s\cup\tru_1^1(p_4)}), \nonumber
\end{align}
where $s$ is the $\Dblue$-arc segment of Case (2) in the $\gbullet$-elementary polygon with edges $\mathfrak{v}^{-}(3)$ and $\mathfrak{v}^{-}(4)$.
Thus,
\[ \H^{-1}(\X(\tc)) \cong \MM(\widehat{\tru_2^3(p_2)})\oplus \MM(\widehat{\tru_2^2(p_4)\cup s\cup\tru_1^1(p_4)})
     = 3 \oplus {^4_3}. \]

Next we construct a complex $\comp{Y}$ such that $\hl(\comp{Y}) = \hl(\X(\tc))-1$. Similar to the method above, we can obtain that
\begin{center}
$\H^{-3}(\X(\tc))=0$, $\H^{-2}(\X(\tc))={^7_6}$, $\H^{-1}(\X(\tc))=3\oplus {^4_3}$, $\H^{0}(\X(\tc))=1\oplus{^1_9}$,
\end{center}
it follows that $\hl(\X(\tc)) = \dim_{\kk}\H^{-1}(\X(\tc))$=3. In this case,
\begin{center}
$\varsigma=1$ and $ k = \min\{1\le i\le n(\tc)=10 \mid \varsigma^{\tc}_i = \varsigma = 1\} = 1.$
\end{center}
According to the proof of Proposition \ref{prop:stringcase}, the curve $c_{[1,2]}\cup\cdots\cup c_{[10,11]}$ induces a complex
\[\comp{X}= 6
\mathop{-\!\!\!-\!\!\!\longrightarrow}\limits^{\left(\begin{smallmatrix} 0 \\ \alpha_7\alpha_6 \\ \alpha_5 \\ 0 \end{smallmatrix}\right)}
\ker\alpha_1 \oplus \begin{smallmatrix} 8 \\ 7 \\ 6 \end{smallmatrix} \oplus {^5_6} \oplus 3
\mathop{-\!\!\!-\!\!\!-\!\!\!-\!\!\!-\!\!\!-\!\!\!\longrightarrow}\limits^{
\left(\begin{smallmatrix}
 \hbar & 0 & 0 & 0 \\
 0 & \alpha_8 & 0 & 0\\
 0 & 0 & \alpha_4 & \alpha_3\\
 0 & 0 & 0 & \alpha_2
\end{smallmatrix}\right)}
{^2_3}\oplus {^9_8} \oplus {{_5}{^4}{_3}} \oplus ^2_3
\mathop{-\!\!\!-\!\!\!-\!\!\!-\!\!\!-\!\!\!-\!\!\!\longrightarrow}\limits^{
\left(\begin{smallmatrix}
\alpha_1 & \alpha_9 & 0 & 0 \\
 0 & 0 & 0 & \alpha_1
\end{smallmatrix}\right) }
{{_9}{^1}{_2}}\oplus {{_9}{^1}{_2}},
\]
where $\hbar:\ker\alpha_1 \to {^2_3}$ is the kernel of $\alpha_1: {^2_3}\to {{_9}{^1}{_2}}$.
Then there is a curve $c'$ with
\[c' = c_{[0,1]}'\cup c_{[1,2]}'\cup\cdots\cup c_{[11, 12]}' \] such that $\comp{X}$ is quasi-isomorphic to $\X(\tc')$, where $n(\tc')=n(\tc)+1=11$ and for all $1\le i\le n(\tc)$, $c_{[i+1,i+2]}'=c_{[i,i+1]}$, see the left picture in \Pic \ref{fig:coholength-1}.
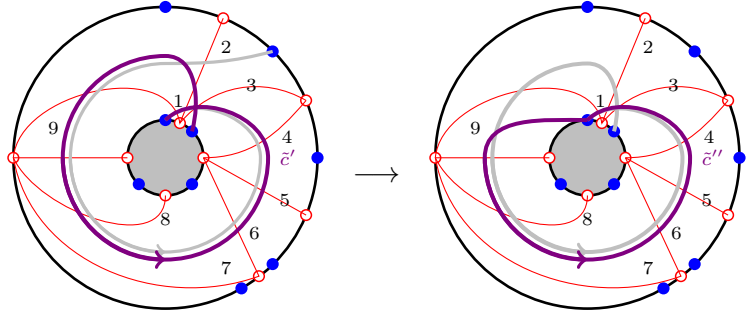
\begin{figure}[htbp]
\centering
\begin{tikzpicture}
\draw[line width=1pt] (0,0) circle(2);
\fill[gray!50] (0,0) circle(0.5);
\draw[line width=1pt] (0,0) circle(0.5);
\fill[blue][line width=1pt] ( 0.  , 2.  ) circle(0.08);
\fill[blue][line width=1pt] ( 0.  , 0.5 ) circle(0.08);
\fill[blue][line width=1pt] ( 1.41, 1.41) circle(0.08);
\fill[blue][line width=1pt] ( 0.35, 0.35) circle(0.08);
\fill[blue][line width=1pt] ( 2.00, 0.00) circle(0.08);
\fill[blue][line width=1pt] ( 1.41,-1.41) circle(0.08);
\fill[blue][line width=1pt] ( 0.35,-0.35) circle(0.08);
\fill[blue][line width=1pt] ( 1. ,-1.732) circle(0.08);
\fill[blue][line width=1pt] (-0.35,-0.35) circle(0.08);
\fill[red] (-2.  , 0.  ) circle(0.08); \fill[white] (-2.  , 0.  ) circle(0.06);
\fill[red] ( 0.19, 0.46) circle(0.08); \fill[white] ( 0.19, 0.46) circle(0.06);
\fill[red] ( 0.76, 1.85) circle(0.08); \fill[white] ( 0.76, 1.85) circle(0.06);
\fill[red] ( 1.85, 0.76) circle(0.08); \fill[white] ( 1.85, 0.76) circle(0.06);
\fill[red] ( 0.5 , 0.  ) circle(0.08); \fill[white] ( 0.5 , 0.  ) circle(0.06);
\fill[red] ( 1.85,-0.76) circle(0.08); \fill[white] ( 1.85,-0.76) circle(0.06);
\fill[red] ( 1.23,-1.57) circle(0.08); \fill[white] ( 1.23,-1.57) circle(0.06);
\fill[red] ( 0.  ,-0.5 ) circle(0.08); \fill[white] ( 0.  ,-0.5 ) circle(0.06);
\fill[red] (-0.5 , 0.  ) circle(0.08); \fill[white] (-0.5 , 0.  ) circle(0.06);
\draw[red] (-2.  , 0.  ) to[out=  85,in= 100] ( 0.19, 0.46) -- ( 0.76, 1.85);
\draw[red] ( 0.19, 0.46) to[out=  50,in= 150] ( 1.85, 0.76);
\draw[red] ( 0.5 , 0.  ) to[out=   0,in= 230] ( 1.85, 0.76);
\draw[red] ( 0.5 , 0.  ) -- ( 1.85,-0.76);
\draw[red] ( 0.5 , 0.  ) -- ( 1.23,-1.57) to[out= 200,in=  -80] (-2.  , 0.  ) -- (-0.5 , 0.  );
\draw[red] (-2.  , 0.  ) to[out= -65,in= -90] ( 0.  ,-0.5 );
{\tiny
\draw ( 0.  , 0.75) node[right]{1} ( 0.8 , 1.65) node[below]{2} ( 0.96, 0.96) node[right]{3}
      ( 1.6 , 0.1 ) node[above]{4} ( 1.74,-0.56) node[ left]{5} ( 1.  ,-1.  ) node[right]{6}
      ( 0.8 ,-1.28) node[below]{7} ( 0.  ,-0.99) node[above]{8} (-1.3 , 0.4 ) node[ left]{9};}
\draw[black!25][line width=1.2pt] ( 0.  , 0.5 ) to[out=  45, in=  90] ( 1.25, 0.  )
     to[out= -90, in=   0] ( 0.  ,-1.25) to[out= 180, in= -90] (-1.25, 0.  )
     to[out=  90, in= 180] ( 0.  , 1.25) to[out=   0, in= 195] ( 1.41, 1.41);
\draw[black!25][line width=1.2pt]  ( 0.  ,-1.25) to[out= 180, in= -90] (-1.25, 0.  ) [<-];
\draw[violet][line width=1.5pt] ( 0.  , 0.5 ) to[out=  45, in=  90] ( 1.35, 0.  ) node[right]{\tiny$\tc'$}
     to[out= -90, in=   0] ( 0.  ,-1.35) to[out= 180, in= -90] (-1.35, 0.  )
     to[out=  90, in= 180] ( 0.  , 1.35) to[out=   0, in=  75] ( 0.35, 0.35);
\draw[violet][line width=1.5pt]  ( 0.  ,-1.35) to[out= 180, in= -90] (-1.35, 0.  ) [<-];
\end{tikzpicture}
\
\begin{tikzpicture}[scale=0.9]
\draw[white] (0,-2) -- (0,2);
\draw (0,0) node{$\longrightarrow$};
\end{tikzpicture}
\
\begin{tikzpicture}
\draw[line width=1pt] (0,0) circle(2);
\fill[gray!50] (0,0) circle(0.5);
\draw[line width=1pt] (0,0) circle(0.5);
\fill[blue][line width=1pt] ( 0.  , 2.  ) circle(0.08);
\fill[blue][line width=1pt] ( 0.  , 0.5 ) circle(0.08);
\fill[blue][line width=1pt] ( 1.41, 1.41) circle(0.08);
\fill[blue][line width=1pt] ( 0.35, 0.35) circle(0.08);
\fill[blue][line width=1pt] ( 2.00, 0.00) circle(0.08);
\fill[blue][line width=1pt] ( 1.41,-1.41) circle(0.08);
\fill[blue][line width=1pt] ( 0.35,-0.35) circle(0.08);
\fill[blue][line width=1pt] ( 1. ,-1.732) circle(0.08);
\fill[blue][line width=1pt] (-0.35,-0.35) circle(0.08);
\fill[red] (-2.  , 0.  ) circle(0.08); \fill[white] (-2.  , 0.  ) circle(0.06);
\fill[red] ( 0.19, 0.46) circle(0.08); \fill[white] ( 0.19, 0.46) circle(0.06);
\fill[red] ( 0.76, 1.85) circle(0.08); \fill[white] ( 0.76, 1.85) circle(0.06);
\fill[red] ( 1.85, 0.76) circle(0.08); \fill[white] ( 1.85, 0.76) circle(0.06);
\fill[red] ( 0.5 , 0.  ) circle(0.08); \fill[white] ( 0.5 , 0.  ) circle(0.06);
\fill[red] ( 1.85,-0.76) circle(0.08); \fill[white] ( 1.85,-0.76) circle(0.06);
\fill[red] ( 1.23,-1.57) circle(0.08); \fill[white] ( 1.23,-1.57) circle(0.06);
\fill[red] ( 0.  ,-0.5 ) circle(0.08); \fill[white] ( 0.  ,-0.5 ) circle(0.06);
\fill[red] (-0.5 , 0.  ) circle(0.08); \fill[white] (-0.5 , 0.  ) circle(0.06);
\draw[red] (-2.  , 0.  ) to[out=  85,in= 100] ( 0.19, 0.46) -- ( 0.76, 1.85);
\draw[red] ( 0.19, 0.46) to[out=  50,in= 150] ( 1.85, 0.76);
\draw[red] ( 0.5 , 0.  ) to[out=   0,in= 230] ( 1.85, 0.76);
\draw[red] ( 0.5 , 0.  ) -- ( 1.85,-0.76);
\draw[red] ( 0.5 , 0.  ) -- ( 1.23,-1.57) to[out= 200,in=  -80] (-2.  , 0.  ) -- (-0.5 , 0.  );
\draw[red] (-2.  , 0.  ) to[out= -65,in= -90] ( 0.  ,-0.5 );
{\tiny
\draw ( 0.  , 0.75) node[right]{1} ( 0.8 , 1.65) node[below]{2} ( 0.96, 0.96) node[right]{3}
      ( 1.6 , 0.1 ) node[above]{4} ( 1.74,-0.56) node[ left]{5} ( 1.  ,-1.  ) node[right]{6}
      ( 0.8 ,-1.28) node[below]{7} ( 0.  ,-0.99) node[above]{8} (-1.3 , 0.4 ) node[ left]{9};}
\draw[black!25][line width=1.5pt] ( 0.  , 0.5 ) to[out=  45, in=  90] ( 1.25, 0.  )
     to[out= -90, in=   0] ( 0.  ,-1.25) to[out= 180, in= -90] (-1.25, 0.  )
     to[out=  90, in= 180] ( 0.  , 1.25) to[out=   0, in=  75] ( 0.35, 0.35);
\draw[black!25][line width=1.5pt]  ( 0.  ,-1.25) to[out= 180, in= -90] (-1.25, 0.  ) [<-];
\draw[violet][line width=1.5pt] ( 0.  , 0.5 ) to[out=  45, in=  90] ( 1.35, 0.  ) node[right]{\tiny$\tc''$}
     to[out= -90, in=   0] ( 0.  ,-1.35) to[out= 180, in= -90] (-1.35, 0.  )
     to[out=  90, in= 180] ( 0.  , 0.5);
\draw[violet][line width=1.5pt]  ( 0.  ,-1.35) to[out= 180, in= -90] (-1.35, 0.  ) [<-];
\end{tikzpicture}
\caption{The curves $\tc'$ and $\tc''$}
\label{fig:coholength-1}
\end{figure}
Since $\H^0(\X(\tc'))=1\oplus{^1_9}$, we have $\hl(\X(\tc'))=3=\hl(\X(\tc))$.
Furthermore,
\[k' = \min\{1\le i\le n(\tc')=11 \mid \varsigma^{\tc}_i = 0\} = 3.\] Thus, $c_{[3,4]}'\cdots c_{[11,12]}'$
induces a complex as follows:
\[ 6
\mathop{-\!\!\!-\!\!\!\longrightarrow}\limits^{\left(\begin{smallmatrix} \alpha_7\alpha_6 \\ \alpha_5 \\ 0 \end{smallmatrix}\right)}
\begin{smallmatrix} 8 \\ 7 \\ 6 \end{smallmatrix} \oplus {^5_6} \oplus 3
\mathop{-\!\!\!-\!\!\!-\!\!\!-\!\!\!\longrightarrow}\limits^{
\left(\begin{smallmatrix}
 \alpha_8 & 0 & 0\\
 0 & \alpha_4 & \alpha_3\\
 0 & 0 & \alpha_2
\end{smallmatrix}\right)}
{^9_8} \oplus {{_5}{^4}{_3}} \oplus ^2_3
\mathop{-\!\!\!-\!\!\!-\!\!\!-\!\!\!\longrightarrow}\limits^{
\left(\begin{smallmatrix}
 \hat{h} & 0 & 0 \\
 0 & 0 & \alpha_1
\end{smallmatrix}\right) }
\mathrm{coker}\ \alpha_8 \oplus {{_9}{^1}{_2}}.
\]
It is quasi-isomorphic to
\[\X(\tc'') = 6
\mathop{-\!\!\!-\!\!\!\longrightarrow}\limits^{\left(\begin{smallmatrix} \alpha_6 \\ \alpha_5 \\ 0 \end{smallmatrix}\right)}
{^7_6}\oplus{^5_6}\oplus 3
\mathop{-\!\!\!-\!\!\!\longrightarrow}\limits^{
\left(\begin{smallmatrix}
0 & \alpha_4 & \alpha_3 \\
0 & 0 & \alpha_2
\end{smallmatrix}\right)}
{{_5}{^4}{_3}\oplus {^2_3}}
\mathop{-\!\!\!-\!\!\!\longrightarrow}\limits^{\big(0\ \alpha_1\big)}
{{_9}{^1}{_2}}, \]
where $\tc''$ is shown in the right picture in \Pic \ref{fig:coholength-1}.
Then, we obtain that
\begin{center}
$\H^{-3}(\X(\tc''))=0$, $\H^{-2}(\X(\tc''))={^7_6}$, $\H^{-1}(\X(\tc''))={^4_3}$ and $\H^0(\X(\tc''))={^1_9}$,
\end{center}
this shows that $\hl(\X(\tc''))=2 = \hl(\X(\tc))-1$.

\begin{figure}[htbp]
\centering
\begin{tikzpicture}[scale=2]
\draw[line width=1pt] (0,0) circle(2);
\fill[gray!50] (0,0) circle(0.5);
\draw[line width=1pt] (0,0) circle(0.5);
\fill[blue][line width=1pt] ( 0.  , 2.  ) circle(0.08);
\fill[blue][line width=1pt] ( 0.  , 0.5 ) circle(0.08);
\fill[blue][line width=1pt] ( 1.41, 1.41) circle(0.08);
\fill[blue][line width=1pt] ( 0.35, 0.35) circle(0.08);
\fill[blue][line width=1pt] ( 2.00, 0.00) circle(0.08);
\fill[blue][line width=1pt] ( 1.41,-1.41) circle(0.08);
\fill[blue][line width=1pt] ( 0.35,-0.35) circle(0.08);
\fill[blue][line width=1pt] ( 1. ,-1.732) circle(0.08);
\fill[blue][line width=1pt] (-0.35,-0.35) circle(0.08);
\draw[blue][dash pattern=on 2pt off 2pt] ( 0.  , 0.5 )--( 0.  , 2.  )--( 1.41, 1.41)--( 0.35, 0.35)--( 2.00, 0.  )
      --( 1.41,-1.41)--( 0.35,-0.35)--( 1. ,-1.732);
\draw[blue][dash pattern=on 2pt off 2pt] ( 0.35,-0.35) to[out=-80,in=0] (0,-1) to[out=180,in=-100] (-0.35,-0.35)
     to[out=170,in=-90] (-1.3,0.4) to[out=90,in=135] (0,0.5);
\fill[red] (-2.  , 0.  ) circle(0.08); \fill[white] (-2.  , 0.  ) circle(0.06);
\fill[red] ( 0.19, 0.46) circle(0.08); \fill[white] ( 0.19, 0.46) circle(0.06);
\fill[red] ( 0.76, 1.85) circle(0.08); \fill[white] ( 0.76, 1.85) circle(0.06);
\fill[red] ( 1.85, 0.76) circle(0.08); \fill[white] ( 1.85, 0.76) circle(0.06);
\fill[red] ( 0.5 , 0.  ) circle(0.08); \fill[white] ( 0.5 , 0.  ) circle(0.06);
\fill[red] ( 1.85,-0.76) circle(0.08); \fill[white] ( 1.85,-0.76) circle(0.06);
\fill[red] ( 1.23,-1.57) circle(0.08); \fill[white] ( 1.23,-1.57) circle(0.06);
\fill[red] ( 0.  ,-0.5 ) circle(0.08); \fill[white] ( 0.  ,-0.5 ) circle(0.06);
\fill[red] (-0.5 , 0.  ) circle(0.08); \fill[white] (-0.5 , 0.  ) circle(0.06);
\draw[red][dash pattern=on 2pt off 2pt] (-2.  , 0.  ) to[out=  85,in= 100] ( 0.19, 0.46) -- ( 0.76, 1.85);
\draw[red][dash pattern=on 2pt off 2pt] ( 0.19, 0.46) to[out=  50,in= 150] ( 1.85, 0.76);
\draw[red][dash pattern=on 2pt off 2pt] ( 0.5 , 0.  ) to[out=   0,in= 230] ( 1.85, 0.76);
\draw[red][dash pattern=on 2pt off 2pt] ( 0.5 , 0.  ) -- ( 1.85,-0.76);
\draw[red][dash pattern=on 2pt off 2pt] ( 0.5 , 0.  ) -- ( 1.23,-1.57) to[out= 200,in=  -80] (-2.  , 0.  ) -- (-0.5 , 0.  );
\draw[red][dash pattern=on 2pt off 2pt] (-2.  , 0.  ) to[out= -65,in= -90] ( 0.  ,-0.5 );
\draw[line width=1pt] (1.414,1.414) to[out=180,in=90] (-1.7,0.5) to [out=-90,in=180] (-0.35,-0.35)
                      (-1.65, 0.7 ) node[left]{$s$};
\draw[violet][line width=1.5pt] ( 0.35, 0.35) to[out=45,in=0] ( 0.  , 1.3 ) to [out=180,in=90] (-1.6 ,0.  )
           to[out=-90,in=165] (-0.  ,-1.6 ) -- ( 1. ,-1.732)
           ( 0.  , 1.26) node[below right]{$s^{\circlearrowleft}$};
\draw[violet][line width=1.5pt] ( 0.35, 0.35) to[out=45,in=0] ( 0.  , 1.3 ) [->];
\draw[orange][line width=1.5pt] ( 0.76, 1.85) to[out= 180,in=  90] (-0.5 , 0.  ) ( 0.07, 1.75) node[right]{$p_1$};
\draw[ green][line width=1.5pt] ( 0.76, 1.85) -- ( 2.  , 0.  )  ( 1.47, 0.77) node[left]{$p_2$};
\draw[ green][line width=1.5pt] (-0.52,-0. ) to[out= 90, in=   0] (-0.85, 1.1 )
           to[out= 180,in=  90] (-1.5 , 0.  ) to[out=-90, in= 180] ( 0.  ,-0.5 )
           (-0.85,-0.5 ) node[below]{$p_9$};
\end{tikzpicture}
\caption{The permissible curve $s$ and the cohomology curve of $s^{\circlearrowright}$}
\label{fig:projresol}
\end{figure}
\end{example}

(2) Let $M = S(1) = \MM(s)$ be the simple module corresponding to $1\in\Q_0$. By Theorem \ref{thm:string embedding}, we can embedding $s$ to an admissible curve $\tilde{s}^{\circlearrowleft}$, See \Pic \ref{fig:projresol}. Then, we know that $n(\tilde{s}^{\circlearrowleft})=6$. Now assume that $\varsigma^{\tilde{s}^{\circlearrowleft}}_1 = 2$. Then the complex $\X(\tilde{s}^{\circlearrowleft})$ is
\[0 \longrightarrow {^7_6} \longrightarrow \begin{smallmatrix}8\\ 7 \\ 6 \end{smallmatrix} \oplus 3 \longrightarrow {^9_8} \oplus {^2_3} \longrightarrow {{_9}{^1}{_2}} \longrightarrow 0. \]
As we all know, it is the deleted projective resolution of $S(1)$.

On the other hand, the index of intersection $c(t^{\tilde{s}^{\circlearrowleft}}_3)$ given by $\tilde{s}^{\circlearrowleft}$ intersecting with $\rbullet$-FFAS is $0$. For the intersection $c(t^{\tilde{s}^{\circlearrowleft}}_3)$, the $3$-rd cohomological curve of $\tilde{s}^{\circlearrowleft}$ is the truncation $\tru_2^2(p_1)$ of $p_1$.
Thus,  by Theorem \ref{thm:homologies}, we obtain that $\widehat{\tru_2^2(p_1)}=s$. Furthermore, we have
\[\H^{0}(\X(\tilde{s}^{\circlearrowleft})) \cong \MM(\widehat{\tru_2^2(p_1)}) \cong \MM(s) = S(1). \]

\begin{example}\rm
Let $A=\kk\Q/\I$ be the algebra given by the $2$-Kronecker quiver
\[\Q= \xymatrix{1 \ar@/^0.25pc/[r]^{\alpha} \ar@/_0.25pc/[r]_{\beta} & 2}.\]
The marked ribbon surface $\mathbf{S}^{\mathcal{F}}(A)$ is shown in Figure \ref{fig:exp}.
\begin{figure}[htbp]
\centering
\begin{tikzpicture}
\draw[line width=1pt] (0,0) circle(2);
\fill[gray!50] (0,0) circle(0.5);
\draw[line width=1pt] (0,0) circle(0.5);
\fill[blue][line width=1pt] ( 0. , 2. ) circle(0.1);
\fill[blue][line width=1pt] ( 0. , 0.5) circle(0.1);
\draw[blue][dash pattern=on 2pt off 2pt] ( 0. , 2. )
  to[out= -30, in=  90] ( 1.5, 0. ) to[out= -90, in=   0] ( 0. ,-1.5)
  to[out= 180, in= -90] (-1.2, 0. ) to[out=  90, in= 135] ( 0. , 0.5);
\draw[blue][dash pattern=on 2pt off 2pt] ( 0. , 2. )--( 0. , 0.5);
\draw[red][line width=1pt] ( 0. ,-2. ) circle(0.1);
\draw[red][line width=1pt] ( 0. ,-0.5) circle(0.1);
\draw[red][dash pattern=on 2pt off 2pt][rotate=180] ( 0. , 2. )
  to[out= -30, in=  90] ( 1.5, 0. ) to[out= -90, in=   0] ( 0. ,-1.5)
  to[out= 180, in= -90] (-1.2, 0. ) to[out=  90, in= 135] ( 0. , 0.5);
\draw[red][dash pattern=on 2pt off 2pt][rotate=180] ( 0. , 2. )--( 0. , 0.5);
\draw[violet][line width=1.5pt] (0,2) to[out=-10,in=90] (1.7,0)
     to[out=-90,in=  0] (0,-1.7) to[out=180,in=-90] (-1.4,0) node[left]{$\tc'$}
     to[out= 90,in=180] (0, 1.4) to[out=  0,in= 90] ( 1.1,0)
     to[out=-90,in=  0] (0,-1.1) to[out=180,in=-90] (-1.1,0)
     to[out= 90,in=180] (0, 1.1) to[out=  0,in= 90] ( 0.9,0)
     to[out=-90,in=  0] (0,-0.9) to[out=180,in=-90] (-0.9,0)
     to[out= 90,in=225] (0, 2. );
\draw[green][line width=1.5pt] (0,0) circle(1.25); \draw[green] (1.15, 0) node[right]{$\tc$};
\end{tikzpicture}
\caption{Marked ribbon surface of 2-Kronecker algebra and the admissible curve $c'$ corresponded by band complex}
\label{fig:exp}
\end{figure}
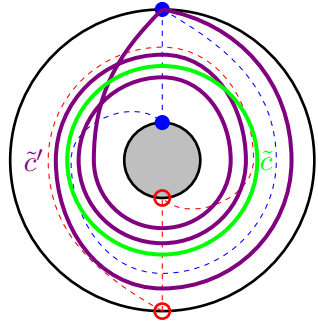
Consider the band complex
\[\comp{B}=\ \ \xymatrix@C=2cm{
P(2)^{\oplus 2} \ar[r]^{\left(\begin{smallmatrix}
\alpha\cdot\text{\i d} & \\
 & \beta \pmb{J}_{2}(\lambda)
\end{smallmatrix}\right)} & P(1)^{\oplus 2}}, \]
where $P(1)$ is $0$-th component of $\comp{B}$, $\text{\i d}=\pmb{E}_{2\times 2}$ and $\pmb{J}_{2}(\lambda)=\left({_1^\lambda}\ {_\lambda^0}\right)$ ($\lambda\ne 0$).
Then, we have
\[\H^0(\comp{B}) = \xymatrix{\kk^2 \ar@/^0.25pc/[r]^{\left({_0^1}\ {_1^0}\right)}
\ar@/_0.25pc/[r]_{\left({_1^\lambda}\ {_\lambda^0}\right)} & \kk^2}\]
and $\H^{-1}(\comp{B})$ is a quotient of $P(2)^{\oplus 2}$.
Thus, $\hl(\comp{B})=\dim_{\kk}\H^0(\comp{B})=4$. According to the proof of Lemma \ref{lemm:bandcase} and Proposition \ref{prop:bandcase},
there is an admissible curve $\tc'$ such that the indecomposable complex
\[\X(\tc')=\xymatrix@C=2cm{ P(2)^{\oplus 3} \ar[r]^{d=\left(\begin{smallmatrix}
\beta & \alpha & 0  \\
0 & \beta  & \alpha \\
0 & 0  & \beta
\end{smallmatrix}\right)} & P(1)^{\oplus 3} }\]
satisfies that
\[\dim_{\kk}\H^0(\X(\tc')) = \dim_{\kk}(P(1)^{\oplus 3}/\im d) 
= 6 \ge \hl(\comp{B})=4.\]

Let $\tc''$ be the admissible curve induced by a truncation of $\tc$ which is shown in
the second picture of \Pic \ref{fig:exp:dim-1}, then
\[\X(\tc'') = \xymatrix@C=2cm{ P(2)^{\oplus 3} \ar[r]^{\left(\begin{smallmatrix}
\beta & \alpha & 0  \\
0 & \beta  & \alpha
\end{smallmatrix}\right)} & P(1)^{\oplus 2} }\]
is a indecomposable complex with
\[\hl(\X(\tc''))  =\dim_{\kk}\H^{-1}(\X(\tc''))=\dim_{\kk}\H^0(\X(\tc'')) = 3 = \hl(\comp{B})-1.\]
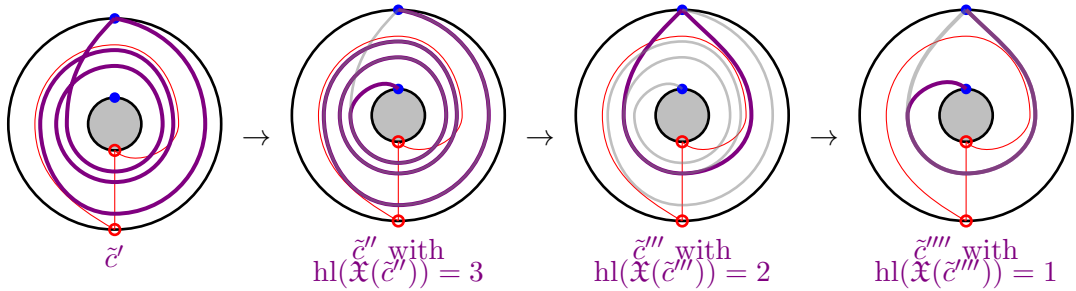
\begin{figure}[htbp]
\centering
\begin{tikzpicture}[scale=0.7]
\draw[line width=1pt] (0,0) circle(2);
\fill[gray!50] (0,0) circle(0.5);
\draw[line width=1pt] (0,0) circle(0.5);
\fill[blue][line width=1pt] ( 0. , 2. ) circle(0.1);
\fill[blue][line width=1pt] ( 0. , 0.5) circle(0.1);
\draw[red][line width=1pt] ( 0. ,-2. ) circle(0.1);
\draw[red][line width=1pt] ( 0. ,-0.5) circle(0.1);
\draw[red][rotate=180] ( 0. , 2. )
  to[out= -30, in=  90] ( 1.5, 0. ) to[out= -90, in=   0] ( 0. ,-1.5)
  to[out= 180, in= -90] (-1.2, 0. ) to[out=  90, in= 135] ( 0. , 0.5);
\draw[red][rotate=180] ( 0. , 2. )--( 0. , 0.5);
\draw[violet][line width=1.5pt] (0,2) to[out=-10,in=90] (1.7,0)
     to[out=-90,in=  0] (0,-1.7) to[out=180,in=-90] (-1.4,0)
     to[out= 90,in=180] (0, 1.4) to[out=  0,in= 90] ( 1.1,0)
     to[out=-90,in=  0] (0,-1.1) to[out=180,in=-90] (-1.1,0)
     to[out= 90,in=180] (0, 1.1) to[out=  0,in= 90] ( 0.9,0)
     to[out=-90,in=  0] (0,-0.9) to[out=180,in=-90] (-0.9,0)
     to[out= 90,in=225] (0, 2. )
     (0,-2.5) node{$\tc'$} (0,-3) node{\color{white}$x$};
\end{tikzpicture}
\begin{tikzpicture}[scale=0.7]
\draw[white] (0,-3) -- (0,2);
\draw (0,0) node{$\to$};
\end{tikzpicture}
\begin{tikzpicture}[scale=0.7]
\draw[line width=1pt] (0,0) circle(2);
\fill[gray!50] (0,0) circle(0.5);
\draw[line width=1pt] (0,0) circle(0.5);
\fill[blue][line width=1pt] ( 0. , 2. ) circle(0.1);
\fill[blue][line width=1pt] ( 0. , 0.5) circle(0.1);
\draw[red][line width=1pt] ( 0. ,-2. ) circle(0.1);
\draw[red][line width=1pt] ( 0. ,-0.5) circle(0.1);
\draw[red][rotate=180] ( 0. , 2. )
  to[out= -30, in=  90] ( 1.5, 0. ) to[out= -90, in=   0] ( 0. ,-1.5)
  to[out= 180, in= -90] (-1.2, 0. ) to[out=  90, in= 135] ( 0. , 0.5);
\draw[red][rotate=180] ( 0. , 2. )--( 0. , 0.5);
\draw[violet][line width=1.5pt] (0,2) to[out=-10,in=90] (1.7,0)
     to[out=-90,in=  0] (0,-1.7) to[out=180,in=-90] (-1.4,0)
     to[out= 90,in=180] (0, 1.4) to[out=  0,in= 90] ( 1.1,0)
     to[out=-90,in=  0] (0,-1.1) to[out=180,in=-90] (-1.1,0)
     to[out= 90,in=180] (0, 1.1) to[out=  0,in= 90] ( 0.9,0)
     to[out=-90,in=  0] (0,-0.9) to[out=180,in=-90] (-0.9,0)
     to[out= 90,in=135] (0, 0.5)
     (0,-2.5) node{$\tc''$ with} (0,-3) node{$\hl(\X(\tc''))=3$};
\draw[gray][line width=1pt][opacity=0.5] (0,2) to[out=-10,in=90] (1.7,0)
     to[out=-90,in=  0] (0,-1.7) to[out=180,in=-90] (-1.4,0)
     to[out= 90,in=180] (0, 1.4) to[out=  0,in= 90] ( 1.1,0)
     to[out=-90,in=  0] (0,-1.1) to[out=180,in=-90] (-1.1,0)
     to[out= 90,in=180] (0, 1.1) to[out=  0,in= 90] ( 0.9,0)
     to[out=-90,in=  0] (0,-0.9) to[out=180,in=-90] (-0.9,0)
     to[out= 90,in=225] (0, 2. );
\end{tikzpicture}
\begin{tikzpicture}[scale=0.7]
\draw[white] (0,-3) -- (0,2);
\draw (0,0) node{$\to$};
\end{tikzpicture}
\begin{tikzpicture}[scale=0.7]
\draw[line width=1pt] (0,0) circle(2);
\fill[gray!50] (0,0) circle(0.5);
\draw[line width=1pt] (0,0) circle(0.5);
\fill[blue][line width=1pt] ( 0. , 2. ) circle(0.1);
\fill[blue][line width=1pt] ( 0. , 0.5) circle(0.1);
\draw[red][line width=1pt] ( 0. ,-2. ) circle(0.1);
\draw[red][line width=1pt] ( 0. ,-0.5) circle(0.1);
\draw[red][rotate=180] ( 0. , 2. )
  to[out= -30, in=  90] ( 1.5, 0. ) to[out= -90, in=   0] ( 0. ,-1.5)
  to[out= 180, in= -90] (-1.2, 0. ) to[out=  90, in= 135] ( 0. , 0.5);
\draw[red][rotate=180] ( 0. , 2. )--( 0. , 0.5);
\draw[violet][line width=1.5pt] (0,2) to[out=-45,in=90] (1.3,0)
     to[out=-90,in=  0] (0,-1.1) to[out=180,in=-90] (-1.1,0)
     to[out= 90,in=225] (0, 2  )
     (0,-2.5) node{$\tc'''$ with} (0,-3) node{$\hl(\X(\tc'''))=2$};
\draw[gray][line width=1pt][opacity=0.5] (0,2) to[out=-10,in=90] (1.7,0)
     to[out=-90,in=  0] (0,-1.7) to[out=180,in=-90] (-1.4,0)
     to[out= 90,in=180] (0, 1.4) to[out=  0,in= 90] ( 1.1,0)
     to[out=-90,in=  0] (0,-1.1) to[out=180,in=-90] (-1.1,0)
     to[out= 90,in=180] (0, 1.1) to[out=  0,in= 90] ( 0.9,0)
     to[out=-90,in=  0] (0,-0.9) to[out=180,in=-90] (-0.9,0)
     to[out= 90,in=135] (0, 0.5);
\end{tikzpicture}
\begin{tikzpicture}[scale=0.7]
\draw[white] (0,-3) -- (0,2);
\draw (0,0) node{$\to$};
\end{tikzpicture}
\begin{tikzpicture}[scale=0.7]
\draw[line width=1pt] (0,0) circle(2);
\fill[gray!50] (0,0) circle(0.5);
\draw[line width=1pt] (0,0) circle(0.5);
\fill[blue][line width=1pt] ( 0. , 2. ) circle(0.1);
\fill[blue][line width=1pt] ( 0. , 0.5) circle(0.1);
\draw[red][line width=1pt] ( 0. ,-2. ) circle(0.1);
\draw[red][line width=1pt] ( 0. ,-0.5) circle(0.1);
\draw[red][rotate=180] ( 0. , 2. )
  to[out= -30, in=  90] ( 1.5, 0. ) to[out= -90, in=   0] ( 0. ,-1.5)
  to[out= 180, in= -90] (-1.2, 0. ) to[out=  90, in= 135] ( 0. , 0.5);
\draw[red][rotate=180] ( 0. , 2. )--( 0. , 0.5);
\draw[violet][line width=1.5pt] (0,2) to[out=-45,in=90] (1.3,0)
     to[out=-90,in=  0] (0,-1.1) to[out=180,in=-90] (-1.1,0)
     to[out= 90,in=135] (0, 0.5)
     (0,-2.5) node{$\tc''''$ with} (0,-3) node{$\hl(\X(\tc''''))=1$};
\draw[gray][opacity=0.5] [line width=1.5pt] (0,2) to[out=-45,in=90] (1.3,0)
     to[out=-90,in=  0] (0,-1.1) to[out=180,in=-90] (-1.1,0)
     to[out= 90,in=225] (0, 2  );
\end{tikzpicture}
\caption{Admissible curves and the cohomological lengths of complexes}
\label{fig:exp:dim-1}
\end{figure}

\noindent Next, we can construct $\tc'''$ and $\tc''''$ (see the third and fourth pictures in \Pic \ref{fig:exp:dim-1}) by the method given in the proof of Proposition \ref{prop:stringcase}, we have
\[\X(\tc''') = \xymatrix@C=1cm{ P(2)\ar[r]^{\beta} & P(1)} \text{ and } \X(\tc'''') = \xymatrix{P(2) \ar[r]^{0} & 0}. \]
Thus, $\hl(\X(\tc'''))=2 = \hl(\X(\tc''))-1$ and $\hl(\X(\tc'''')) = 1 = \hl(\X(\tc'''))-1$.
\end{example}


\section*{Acknowledgements}
This work is supported by the National Natural Science Foundation of China (Grant No. 12171207, 11961007), Scientific Research Foundations of Guizhou
University (Grant No. [2022]65) and Natural Science Research Start-up Foundation
of Recruiting Talents of Nanjing University of Posts and Telecommunications (Grant No.
NY222092).


%

\def\cprime{$'$}

\end{document}